\documentclass[11pt]{article}
\usepackage{amsfonts}
\usepackage{amssymb}
\usepackage{graphicx}
\usepackage{amsmath}
\usepackage{amsthm}
\usepackage{tikz}
\usepackage{enumitem}
\usepackage{cancel}
\usepackage{todonotes}
\usetikzlibrary{matrix}
\usetikzlibrary{arrows,shapes}
\usepackage{cite}
\usepackage{hyperref}
\hypersetup{colorlinks=true, urlcolor=blue}
\expandafter\let\expandafter\oldproof\csname\string\proof\endcsname
\let\oldendproof\endproof
\renewenvironment{proof}[1][\proofname]{\oldproof[\ttfamily\scshape\bf #1.]
}{\oldendproof}

\def\ve{\varepsilon}

\def\tilde{\widetilde}
\def\emp{\emptyset}

\def\dist {{\rm dist}}
\def\dom{{\rm dom}\,}

\def\epi{{\rm epi\,}}

\def\N{{\cal N}}

\def\O{{\cal O}}

\def\Q{{\cal Q}}
\def\d{{\rm d}}
\def\sub{\partial}
\def\B{\mathbb B}

\def\ox{\overline{x}}
\def\oy{\overline{y}}
\def\oz{\overline{z}}
\def\olm{\overline{\lambda}}

\def\cl{{\rm cl}\,}

\def\disp{\displaystyle}

\def\Limsup{\mathop{{\rm Lim}\,{\rm sup}}}

\def\tto{\rightrightarrows}
\def\Hat{\widehat}

\def\Bar{\overline}
\def\ra{\rangle}
\def\la{\langle}
\def\ve{\varepsilon}
\def\B{I\!\!B}

\def\R{{\rm I\!R}}
\def\N{{\rm I\!N}}
\def\ox{\bar{x}}
\def\oy{\bar{y}}
\def\oz{\bar{z}}

\def\ov{\bar{v}}
\def\ow{\bar{w}}
\def\ou{\bar{u}}

\def\op{\bar{p}}

\def\co{\mbox{\rm co}\,}

\def\inte{\mbox{\rm int}\,}
\def\gph{\mbox{\rm gph}\,}
\def\epi{\mbox{\rm epi}\,}

\def\dom{\mbox{\rm dom}\,}

\def\ker{\mbox{\rm ker}\,}

\def\cl{\mbox{\rm cl}\,}

\def\d{{\mathrm d}}

\def\dn{\downarrow}
\def\O{\Omega}
\def\ph{\varphi}
\def\emp{\emptyset}
\def\st{\stackrel}
\def\oR{\Bar{\R}}
\def\lm{\lambda}

\def\dd{\delta}
\def\al{\alpha}
\def\Th{\Theta}
\def\th{\theta}
\def\vt{\vartheta}

\def\ss{\scriptsize }

\def\bd{\mbox{\rm bd}\,}
\def\sm{\hbox{${1\over 2}$}}
\def\sce{\setcounter{equation}{0}}

\begin{document}
\newtheorem{Theorem}{Theorem}[section]
\newtheorem{Proposition}[Theorem]{Proposition}
\newtheorem{Remark}[Theorem]{Remark}
\newtheorem{Lemma}[Theorem]{Lemma}
\newtheorem{Corollary}[Theorem]{Corollary}
\newtheorem{Definition}[Theorem]{Definition}
\newtheorem{Example}[Theorem]{Example}
\renewcommand{\theequation}{{\thesection}.\arabic{equation}}
\renewcommand{\thefootnote}{\fnsymbol{footnote}}
\begin{center}
{\bf\Large Parabolic Regularity in Geometric Variational Analysis}
\\[3ex]
ASHKAN MOHAMMADI\footnote{Department of Mathematics, Wayne State University, Detroit, MI 48202 (ashkan.mohammadi@wayne.edu). Research of this author was partly supported by the National Science Foundation under grant DMS-1808978 and by the US Air Force Office of Scientific Research under grant \#15RT0462.}
BORIS S. MORDUKHOVICH\footnote{Department of Mathematics, Wayne State University, Detroit, MI 48202 (boris@math.wayne.edu). Research of this author was partly supported by the National Science Foundation under grants DMS-1512846 and DMS-1808978, by the US Air Force Office of Scientific Research under grant \#15RT0462, and by the Australian Research Council Discovery Project DP-190100555} and M. EBRAHIM SARABI\footnote{Department of Mathematics, Miami University, Oxford, OH 45065 (sarabim@miamioh.edu).}
\end{center}
\vspace*{0.05in}
\small{\bf Abstract.} The paper is mainly devoted to systematic developments and applications of geometric aspects of second-order variational analysis that are revolved around the concept of parabolic regularity of sets. This concept has been known in variational analysis for more than two decades while being largely underinvestigated. We discover here that parabolic regularity is the key to derive new calculus rules and computation formulas for major second-order generalized differential constructions of variational analysis in connection with some properties of sets that go back to classical differential geometry and geometric measure theory. The established results of second-order variational analysis and generalized differentiation, being married to the developed calculus of parabolic regularity, allow us to obtain novel applications to both qualitative and quantitative/numerical aspects of constrained optimization including second-order optimality conditions, augmented Lagrangians, etc. under weak constraint qualifications.\\[1ex]
\vspace*{-0.1in}

\normalsize
\section{Introduction}\label{intro}\sce

{\em Modern variational analysis} has been recognized as an active and rapidly developed area of mathematics, which is based on {\em variational principles} while addressing broad classes of problems in mathematics and its applications with and without variational structures. Powerful variational principles and techniques used in this field of mathematics involve perturbation and approximation procedures and require dealing with appropriate constructions of {\em generalized differentiation} applied to sets, set-valued mappings, and nonsmooth functions. Another underlying feature of modern variational analysis is a pivoting role of {\em geometric ideas} in both finite-dimensional and infinite-dimensional settings. In fact, several basic notions widely used in variational analysis were first introduced in the framework of {\em differential geometry}; see below.

This paper concerns {\em second-order variational analysis}, which is now on the front line of research and applications. We refer the reader to the books \cite{bs,mor06,mor18,rw} with the extensive bibliographies and commentaries therein for the major methods, constructions, theoretical results, and applications established in variational analysis and related areas by using appropriate tools of second-order generalized differentiation. Here we aim at novel developments and applications that significantly increase our knowledge on the subject and open new gates for further research.

Our main attention is paid to {\em geometric aspects} of second-order analysis with focussing on local properties of {\em nonconvex sets} in finite dimensions under infinitesimal second-order perturbations. The main concept investigated and utilized in the paper is of {\em parabolic regularity} of sets. It was introduced and briefly studied by Rockafellar and Wets in \cite{rw}, but since that time it has not been further investigated and applied in variational analysis and optimization. Our goal is to reveal that this notion is {\em truly fundamental} from both viewpoints of variational theory and applications. We show that it is preserved under various operations performed on sets, and---while being combined with more recent developments in variational analysis---allows us to derive new calculus rules for major second-order generalized differential constructions of variational analysis with significant and rather surprising applications to constrained optimization.

It is conventional in modern variational analysis to deal with {\em extended-real-valued} functions $\ph\colon\R^n\to\oR:=(-\infty,\infty]$, which may attain the value of infinity in addition to real numbers. This provides, in particular, a convenient way to represent geometric constraints of the type $x\in\O$ via the {\em indicator function} $\dd_\O(x)$ of the set $\O$ that equals to $0$ for $x\in\O$ and to $\infty$ for $x\notin\O$. From this viewpoint, the study of local properties of sets corresponds to the consideration of their indicator functions, which is the main object of our analysis here.

To the best of our knowledge, the first attempts to investigate second-order generalized differential properties of extended-real-valued functions started in 1980s with the papers by Lemar\'echal and Nurminskii \cite{ln} and by Hiriart-Urruty \cite{hu1,hu2} that addressed directional derivatives of {\em convex} functions defined by using standard difference quotients. About the same time, Ben-Tal and Zowe \cite{bz1,bz2} initiated a new path toward defining second-order generalized derivatives of nonconvex but {\em finite-valued} functions by exploring the second-order difference quotients along {\em parabolic curves}. Employing a penalization technique, they established in this way some second-order optimality conditions for problems of nonlinear programming. Such a parabolic approach was further advanced by many researchers including Bonnans, Cominetti and Shapiro while coming to complete fruition in \cite{bcs,bcs2} (see also the book \cite{bs}), where second-order optimality conditions were obtained for a large class of constrained and composite optimization problems under a certain second-order regularity condition discussed below.

Other important contributions to second-order generalized differentiation in variational analysis were made by Chaney in \cite{ch0,ch1,ch2} who employed pointwise upper and lower limits of some second-order difference quotients for locally Lipschitzian functions. Similarly to Ben-Tal and Zowe, Chaney utilized a penalization method to achieve second-order optimality conditions for nonlinear programs under the classical linear independence constraint qualification (LICQ). Furthermore, he established a remarkable duality relationship between his second-order generalized derivative and the one introduced by Ben-Tal and Zowe.

In his seminal paper \cite{r88}, Rockafellar achieved a breakthrough in second-order differentiation of extended-real-valued functions by introducing the {\em epi-convergence} of second-order difference quotients, which resembled those in Chaney \cite{ch0} and were not parabolic as those in Ben-Tal and Zowe \cite{bz1,bz2}. As he showed later in \cite{r892}, the proposed approach provided a unified framework for deriving second-order optimality conditions for problems of {\em unconstrained} optimization dealing with extended-real-valued functions. To handle in this way valuable classes of (explicitly) {\em constrained} optimization problems, we require establishing relevant calculus rules for the second subderivative used in \cite{r88,r892} under appropriate constraint qualifications. This line of developments was accomplished by Rockafellar for composite models constructed from {\em fully amenable} functions of the polyhedral structure that are defined via a certain {\em metric regularity} qualification condition, which has been well understood and characterized in variational analysis. The main applications of Rockafellar's theory of {\em twice epi-differentiability} for amenable functions were provided to problems of {\em nonlinear programming} (NLPs) with ${\cal C}^2$-smooth data under the Mangasarian-Fromovitz constraint qualification (MFCQ), which is much weaker than LICQ.

Further developments on twice epi-differentiability and related issues of second-order variational analysis have been recently done in our paper \cite{mms}, where we replaced fully amenable compositions in \cite{r88,r892} by {\em fully subamenable} ones. The main difference between these two classes of extended-real-valued nonconvex functions is that the latter employs a new {\em metric subregularity qualification condition}, which significantly improves the previously used metric regularity (MFCQ, etc.) counterparts. Nevertheless, the sets and functions considered in \cite{mms} are still of the polyhedral structure, which largely restricts the spectrum of possible applications to optimization while being just revolved around NLPs and their polyhedral extensions.\vspace*{0.05in}

In this paper we make a strong {\em move away from polyhedrality} by developing a second-order geometric variational theory that does not involve any polyhedrality requirements. The key here is the concept of {\em parabolically regular sets}, which was introduced in \cite[Definition~13.65]{rw} in the functional framework, but was not explored and applied therein beyond the fully amenable setting.  Now we develop a rather comprehensive variational theory of parabolic regularity that leads us, in particular, to novel applications to nonpolyhedral classes of problems in constrained optimization with deriving no-gap second-order optimality conditions, second-order generalized differential formulas for solution maps to constrained problems, complete characterizations of quadratic growth for augmented Lagrangians. The obtained results constitute the basis of our ongoing projects on the design of new primal-dual algorithms of constrained optimization problems with justifying their superlinear convergence.\vspace*{0.05in}

The rest of the paper is organized as follows. Section~\ref{sect02} recalls and discusses some important notions of variational analysis and generalized differentiation that are broadly used throughout the whole paper. In Section~\ref{sect03} we present the underlying definition of parabolically regular sets, establish the validity of this property for important classes of sets, and reveal relationships between parabolic regularity and some generalized differential notions of second-order variational analysis that are revolved around {\em twice epi-differentiability}.

Sections~\ref{sect04}--\ref{sect05a} mainly focus on the study of second-order properties of the so-called {\em constraint systems}, which are of their own significance in geometric variational analysis while playing a crucial role in constrained optimization as sets of feasible solutions to major classes of constrained problems. We first derive new {\em calculus rules} for {\em second-order tangents} under the (very weak) {\em metric subregularity constraint qualification} (MSCQ). Then this condition is used to establish parabolic regularity for important classes of constraint systems with verifying the {\em preservation} of parabolic regularity under basic operations performed over sets. We also derive in Sections~\ref{sect05} and \ref{sect05a} precise {\em computation formulas} for the {\em second subderivatives} of the indicator functions of general parabolically regular sets and their remarkable specifications.
Sections~\ref{sect06} and \ref{sect06a} are devoted to applications of the developed theory of parabolic regularity to {\em constrained optimization} problems, where sets of feasible solutions are given by the constraint systems studied above. The main theorem of Section~\ref{sect06} provides {\em no-gap second-order necessary and sufficient optimality conditions} for a broad setting in constrained optimization under parabolic regularity. The obtained results cover, in particular, nonpolyhedral problems of {\em conic programming}, where they properly extend previously known developments under the so-called ${\cal C}^2$-{\em cone reducibility}.  Section~\ref{sect06a} concerns the study of the {\em augmented Lagrangians} associated with the constrained optimization problems that are considered here. Besides establishing new second-order properties of augmented Lagrangians with deriving precise formulas for their second subderivatives via {\em Moreau envelopes}, we obtain complete characterizations of their {\em second-order growth} under parabolic regularity, which is a new result even for classical NLPs while being of great importance for subsequent theoretical and numerical applications.

Section~\ref{sect07} deals with the {\em normal cone mappings} associated with the constraint systems under consideration, and hence it is ultimately related to {\em optimal solutions} of constrained optimization problems via first-order optimality conditions. The main result is a precise calculation of the graphical derivative of such normal cone mappings in terms of the given system data under the parabolic regularity and MSCQ conditions, which is a {\em second-order} generalized differential construction for constraint systems known as the {\em subgradient graphical derivative}. The obtained formula gives us an important second-order information on parabolically regular constraint systems that is instrumental for their subsequent study and applications.

The concluding Section~\ref{conc} summarizes the main contributions of the paper and discusses perspectives of further developments and applications of the obtained results.\vspace*{0.05in}

Our notation and terminology are standard in variational analysis; see, e.g., the books by Rockafellar and Wets \cite{rw} and Mordukhovich \cite{mor06,mor18}. For the reader's convenience and notational unification we use as a rule small Greek letters to denote scalar and extended-real-valued functions, small Latin letters for vectors and single-valued mappings/vector functions, and capital letters for sets, set-valued mappings, and matrices. Given a nonempty set $\O$ in the Euclidean space $\R^n$, the symbols $\bd\O$, $\inte\O$, $\cl\O$, and $\O^*$ stand for the boundary, interior, closure, and polar of $\O$, respectively. By $\B$ we denote the closed unit ball in the space in question and by $\B_r(x):=x+r\B$ the closed ball centered at $x$ with radius $r>0$. The distance between $x\in\R^n$ and a set $\O$ is denoted by ${\rm dist}(x;\O)$, while the projection of $x$ onto $\O$ by $P_\O(x)$. Recall also that the vector quantity $x(t)=o(t)$ with $t>0$ means that $\|x(t)\|/t\to 0$ as $t\dn 0$, that $\R_+$ and $\R_-$ signify, respectively, the collections of nonnegative and nonpositive real numbers, and that $\N:=\{1,2,\ldots\}$. The symbol $x\st{\O}{\to}\ox$ indicates that $x\to\ox$ with $x\in\O$. Given a scalar function $\ph\colon\R^n\to\R$, denote by $\nabla\ph(\ox)$ and $\nabla^2\ph(\ox)$ the gradient and Hessian of $\ph$ at $\ox$, respectively. If $f=(f_1,\ldots,f_m)\colon\R^n\to\R^m$ is a vector function that is twice differentiable at $\ox\in\R^n$, its second derivative $\nabla^2f(\ox)$ at this point is a bilinear mapping from $\R^n\times\R^n$ into $\R^m$. In what follows we use the notation $\nabla^2f(\ox)(w,v)$ meaning that
\begin{equation*}
\nabla^2f(\ox)(w,v)=\big(\big\la\nabla^2 f_1(\ox)w,v\big\ra,\ldots,\big\la\nabla^2 f_m(\ox)w,v\big\ra\big)\;\mbox{ for all }\;v,w\in\R^n.
\end{equation*}
Finally, we mention that the notation $F\colon\R^n\tto\R^m$ indicates the possibility of set values $F(x)\subset\R^m$ (including the empty set $\emp$) of $F$ for some $x\in\R^n$, in contrast to the standard notation $f\colon\R^n\to\R^m$ for single-valued mappings as well as extended-real-valued functions. The (Painlev\'e-Kuratowski) {\em outer/upper limit} of $F\colon\R^n\tto\R^m$ as $x\to\ox$ is defined as
\begin{equation}\label{pk}
\Limsup_{x\to\ox}F(x):=\big\{y\in\R^m\big|\;\exists\,x_k\to\ox,\;y_k\to y\;\mbox{ with }\;y_k\in F(x_k),\;k\in\N\big\}.
\end{equation}

\section{Preliminaries from Variational Analysis}\sce\label{sect02}

We begin this section with recalling some notions of geometric variational analysis that are broadly used throughout the paper. It is said that a family of sets $\{\O_t\}$, $t>0$, in $\R^n$ {\em converges} to a set $\O\subset\R^n$ as $t\dn 0$ if $\O$ is closed and
\begin{equation*}
\lim_{t\dn 0}{\rm dist}(w;\O_t)={\rm dist}(w;\O)\;\mbox{ for all }\;w\in\R^n.
\end{equation*}
Given a nonempty set $\O\subset\R^n$ with $\ox\in\O$, the {\em tangent cone} $T_\O(\ox)$ to $\O$ at $\ox\in\O$ is defined by
\begin{equation}\label{tan}
T_\O(\ox):=\big\{w\in\R^n\big|\;\exists\,t_k{\dn}0,\;\;w_k\to w\;\mbox{ as }\;k\to\infty\;\mbox{ with }\;\ox+t_kw_k\in\O\big\}.
\end{equation}
This notion was first introduced in differential geometry independently by Bouligand \cite{bou} and Severi \cite{se} as the set of ``contingent directions" and is often used in variational analysis under their names; see, e.g., \cite{mor06} for more details.

We say a tangent vector $w\in T_\O(\ox)$ is {\em derivable} if there exists $\xi\colon[0,\ve]\to\O$ with $\ve>0$, $\xi(0)=\ox$, and $\xi'_+(0)=w$, where $\xi'_+$ signifies the right derivative of $\xi$ at $0$ defined by
\begin{equation*}
\xi'_+(0):=\lim_{t\dn 0}\frac{\xi(t)-\xi(0)}{t}.
\end{equation*}
The set $\O$ is {\em geometrically derivable} at $\ox$ if every tangent vector $w$ to $\O$ at $\ox$ is derivable. The geometric derivability of $\O$ at $\ox$ can be equivalently described by saying that the sets $[\O-\ox]/{t}$ converge to $T_\O(\ox)$ as $t\dn 0$. Convex sets are important examples of geometrically derivable sets.

The {\em regular/Fr\'echet normal cone} to $\O$ at $\ox\in\O$ is
\begin{equation}\label{pre}
\Hat N_\O(\ox):=\Big\{v\in\R^n\Big|\;\limsup_{x\st{\O}{\to}\ox}\frac{\la v,x-\ox\ra}{\|x-\ox\|}\le 0\Big\},
\end{equation}
which can be equivalently described as $\Hat N_\O(\ox)=T_\O(\ox)^*$, the polar of the tangent cone \eqref{tan}. Note that the regular normal cone \eqref{pre} may be trivial (i.e., $\Hat N_\O(\ox)=\{0\}$) at boundary points of closed sets as, e.g., for $\O:=\big\{(x,\al)\in\R^2|\;\al\ge-|x|\}$ at $\ox=(0,0)$. This contradicts the meaning of normal vectors to sets while being a source of poor calculus for \eqref{pre}, etc. However, taking the outer limit \eqref{pk} of $\Hat N_\O(x)$ at points $x\in\O$ close to $\ox$ leads us to the following robust collection of normal vectors to sets known as the {\em limiting/Mordukhovich normal cone} to $\O$ at $\ox$:
\begin{eqnarray}\label{2.4}
N_\O(\ox):=\Limsup_{x\st{\O}{\to}\ox}\Hat N_\O(x)=\big\{v\in\R^n\big|\;\exists\,x_k\st{\O}{\to}\ox,\;v_k\to v\;\;\mbox{with}\;\;v_k\in\Hat N_\O(x_k)\big\},
\end{eqnarray}
which was introduced in \cite{m76}. Due to the usual nonconvexity of the normal cone \eqref{2.4}, it cannot be obtained as the dual/polar cone to any tangential approximation of $\O$ near $\ox$ while nevertheless enjoying {\em full calculus} based on variatiobal/extremal principles; see \cite{mor06,mor18,rw}.

A vector $v\in\R^n$ is called a {\em proximal normal} to $\O$ at $\ox$ if there exists $r\ge 0$ such that
\begin{equation}\label{proxn}
\la v, x-\ox\ra\le r\|x-\ox\|^2\;\mbox{ for all }\;x\in\O.
\end{equation}
The collection of all proximal normals to $\O$ at $\ox$ is denoted by $N_\O^p(\ox)$. To the best of our knowledge, proximal normals to nonconvex sets first appeared in Federer's paper \cite{fed} on geometric measure theory. In the general case of closed set $\O$ we always have the inclusions $N_\O^p(\ox)\subset\Hat N_\O(\ox)\subset N_\O(\ox)$, where all the cones agree and reduce to the normal cone of convex analysis if $\O$ is convex. The set $\O$ is said to be {\em normally regular} at $\ox\in\O$ if $\Hat N_\O(\ox)=N_\O(\ox)$.\vspace*{0.05in}

Consider further a set-valued mapping/multifunction $F\colon\R^n\rightrightarrows\R^m$ and define some generalized differential notions for it induced by the aforementioned constrictions for sets. Denote the domain and graph of $F$ by, respectively,
\begin{equation*}
\dom F:=\big\{x\in\R^n\big|\;F(x)\ne\emp\big\}\;\mbox{ and }\;\gph F:=\big\{(x,y)\in\R^n\times\R^m\big|\;y\in F(x)\big\}.
\end{equation*}
The {\em graphical derivative} of $F$ at $(\ox,\oy)\in\gph S$ is defined via \eqref{tan} by
\begin{equation}\label{gder}
DF(\ox,\oy)(u):=\big\{v\in\R^m\big|\;(w,v)\in T_{\scriptsize{\gph F}}(\ox,\oy)\big\},\quad u\in\R^n,
\end{equation}
while the {\em coderivative} of $F$ at this point is defined via \eqref{2.4} by
\begin{equation}\label{cod}
D^*F(\ox,\oy)(v):=\big\{u\in\R^n\big|\;(u,-v)\in N_{\scriptsize{\gph F}}(\ox,\oy)\big\},\quad v\in\R^m.
\end{equation}
Note that the generalized derivative constructions \eqref{gder} and \eqref{cod} are {\em not} dual to each other. Besides enjoying comprehensive calculus rules, an advantage of \eqref{cod} is the possibility to obtain in its terms complete pointwise characterizations of fundamental {\em well-posedness} properties of nonlinear analysis. One of these properties and its characterization is used below.

Recall that a set-valued mapping $F\colon\R^n\tto\R^m$ is {\em metrically regular} around $(\ox,\oy)\in\gph $ if there exist $\ell\ge 0$ and neighborhoods $U$ of $\ox$ and $V$ of $\oy$ such that we have the distance estimate
\begin{equation}\label{metreg}
{\rm dist}\big(x;F^{-1}(y)\big)\le\ell\,{\rm dist}\big(y;F(x)\big)\;\mbox{ for all }\;(x,y)\in U\times V.
\end{equation}
The following coderivative characterization of \eqref{metreg} is known as the {\em coderivative/Mordukhovich criterion} \cite{mor93,mor18,rw}: If the graph of $F$ is locally closed around $(\ox,\oy)$, then $F$ is metrically regular around this point if and only if we have
\begin{equation}\label{cod-cr}
{\rm ker}\,D^*F(\ox,\oy):=\big\{v\in\R^m\big|\;0\in D^*F(\ox,\oy)(v)\big\}=\{0\}.
\end{equation}
A more subtle property of mappings, that is broadly employed below, corresponds to the case where $y=\oy$ is fixed in \eqref{metreg} and is known as {\em metric subregularity} of $F$ at $(\ox,\oy)$.

Note that the coderivative criterion \eqref{cod-cr} is the key to convert the metric regularity property into pointwise constraint qualifications, which reduce for particular classes of optimization problems to well-known ones as MFCQ, Robinson's constraint qualification, etc. It is not the case for metric subregularity; see Section~\ref{sect04} for more discussions.\vspace*{0.05in}

To continue, we recall now some generalized differential constructions for extended-real-valued functions while mainly concentrating on {\em second-order} ones by following the book of Rockafellar and Wets \cite{rw}. These constructions are also geometric in nature, but it is convenient for the subsequent applications to present their analytic representations. Note that in this paper we mostly study and apply the {\em primal-space} generalized second-order derivatives for extended-real-valued functions without appealing to the {\em dual-space} second-order subdifferentials (or generalized Hessians) in the sense of \cite{m92}, which are defined via the coderivative \eqref{cod} to the first-order subgradient mapping \eqref{sub} generated by \eqref{2.4}. The reader is referred to the books \cite{mor06,mor18,rw} with the bibliographies and commentaries therein for the dual-spaces generalized differential theory and a variety of applications.

Given a function $\ph\colon\R^n\to\oR$ with its domain and epigraph formed by
\begin{equation*}
\dom\ph:=\big\{x\in\R^n\big|\;\ph(x)<\infty\}\;\mbox{ and }\;\epi\ph:=\big\{(x,\al)\in\R^{n+1}\big|\;\al\ge\ph(x)\big\},
\end{equation*}
respectively, the {\em subderivative} of $\ph$ at $\ox\in\dom\ph$ is defined by
\begin{equation}\label{subder}
\d\ph(\ox)(\ow):=\liminf_{\substack{t\dn 0\\w\to\ow}}{\frac{\ph(\ox+tw)-\ph(\ox)}{t}},\quad\ow\in\R^n,
\end{equation}
whose epigraph is the tangent cone \eqref{tan} to $\epi\ph$ at $(\ox,\ph(\ox))$. Yet another geometric relationship for \eqref{subder} is via the set indicator function $\d\dd_\O(\ox)=\dd_{T_\O(\ox)}$ for all $\ox\in\O$.

The {\em subdifferential} (collections of subgradients) of $\ph$ at $\ox\in\dom\ph$ is generally defined geometrically via the normal cone \eqref{2.4} to $\epi\ph$ by
\begin{equation}\label{sub}
\partial\ph(\ox):=\big\{v\in\R^n\big|\;(v,-1)\in N_{\scriptsize{\epi\ph}}\big(\ox,\ph(\ox)\big)\big\}
\end{equation}
while reducing in the case of convex functions to the classical subdifferential of convex analysis
\begin{equation*}
\partial\ph(\ox):=\big\{v\in\R^n\big|\;\la v,x-\ox\ra\le\ph(x)-\ph(\ox)\;\mbox{ for all }\;x\in\R^n\big\}.
\end{equation*}

Fixing further $\ox\in\dom\ph$ and $\ov\in\R^n$, consider the parametric family of second-order difference quotients for $\ph$ at $(\ox,\ov)$ given by
\begin{equation*}
\Delta_t^2\ph(\bar x,\ov)(u):=\dfrac{\ph(\ox+tu)-\ph(\ox)-t\langle\ov,u\rangle}{\sm t^2}\;\mbox{ with }\;u\in\R^{n}\;\mbox{ and }\;t>0.
\end{equation*}
The {\em second subderivative} of $\ph$ at $\ox$ for $\ov$ is the function $w\mapsto\d^2\ph(\bar x,\ov)(w)$ defined by
\begin{equation}\label{ssd}
\d^2\ph(\bar x,\ov)(w):=\liminf_{\substack{t\dn 0\\ u\to w}}\Delta_t^2\ph(\ox,\ov)(u).
\end{equation}

The following statement taken from \cite[Proposition~13.5]{rw} collects some elementary properties of the second subderivative \eqref{ssd} used throughout the paper.

\begin{Proposition}[\bf properties of second subderivative]\label{ssp} Let $\ph\colon\R^n\to\oR$, and let $(\ox,\ov)\in(\dom\ph)\times\R^n$. Then we have the assertions:

{\bf(i)} The second subderivative $\d^2\ph(\bar x,\ov)$ is a lower semicontinuous $($l.s.c.$)$ function.

{\bf(ii)} The second subderivative $\d^2\ph(\bar x,\ov)$ is positive homogeneous of degree $2$, i.e., $\d^2\ph(\bar x,\ov)(tw)=t^2\d^2\ph(\bar x,\ov)(w)$
for all $w\in\R^n$ and $t>0$.

{\bf(iii)} Whenever $w\in\R^n$, the mapping $\ov\mapsto\d^2\ph(\bar x,\ov)(w)$ is concave.

{\bf(iv)} If the second subderivative $\d^2\ph(\bar x,\ov)$ is a proper function, i.e., $\d^2\ph(\bar x,\ov)(w)>-\infty$ for all $w\in\R^n$ and $\dom\d^2\ph(\bar x,\ov)\ne\emp$, then
\begin{equation*}
\dom\d^2\ph(\bar x,\ov)\subset\big\{w\in\R^n\big|\;\d\ph(\ox)(w)=\la\ov,w\ra\big\}.
\end{equation*}
\end{Proposition}

A function $\ph\colon\R^n\to\oR$ is said to be {\em twice epi-differentiable} at $\bar x$ for $\ov$ if the sets $\epi\Delta_t^2\ph(\ox,\ov)$ converge to $\epi\d^2\ph(\bar x,\ov)$ as $t\dn 0$. If in addition the second subderivative \eqref{ssd} is a proper function, then $\ph$ is said to be {\em properly twice epi-differentiable} at $\bar x$ for $\ov$. It follows from \cite[Proposition~7.2]{rw} that the twice epi-differentiability of $f$ at $\bar x$ for $\ov$ can be equivalently described as follows: for every $w\in\R^n$ and every sequence $t_k\downarrow 0$ there is a sequence $w_k\to w$ such that
\begin{equation}\label{twice-epi}
\Delta_{t_k}^2\ph(\bar x,\ov)(w_k)\to\d^2\ph(\bar x,\ov)(w).
\end{equation}

The main attention in this paper is paid to the study of twice epi-differentiability of {\em sets} via their indicator functions. Given a set $\O\subset\R^n$ and a pair $(\ox,\ov)\in\O\times\R^n$, we can deduce from \eqref{twice-epi} and the definition of geometric derivability that the proper twice epi-differentiability of $\dd_\O$ at $\ox$ for $\ov$ amounts to saying that $\d^2\dd_\O(\bar x,\ov)$ is proper, and that for any $w\in\R^n$ there exist $\ve>0$ and an arc $\xi\colon[0,\ve]\to\R^n$ with $\xi(0)=\ox$ and $\xi'_+(0)=w$ with
\begin{equation}\label{df02}
\Delta_{t}^2\dd_\O(\bar x,\ov)(w_t)\to\d^2\dd_\O(\bar x,\ov)(w)\;\; \mbox{as}\;\;t\dn 0\;\mbox{ with }\;w_t:=\frac{\xi(t)-\xi(0)}{t}.
\end{equation}
Recall also that the {\em second-order tangent set} to $\O$ at $\ox$ for $w$ with $\ox\in\O$ and $w\in T_\O(\ox)$ is
\begin{equation}\label{2tan}
T_\O^2(\ox,w)=\big\{u\in\R^n\big|\;\exists\,t_k{\downarrow}0,\;u_k\to u\;\mbox{ as }\;k\to\infty\;\;\mbox{with}\;\;\ox+t_kw+\sm t_k^2u_k\in\O\big\}.
\end{equation}
It is easy to see that $T_\O^2(\ox,0)=T_\O(\ox)$. If in addition $\O$ is a closed cone, then we have $T_\O^2(0,w)=T_{\O}(w)$ for all $w\in \O $.

Finally in this section, we say that $\O$ is {\em parabolically derivable} at $\ox$ for $w\in\R^n$ if $T_\O^2(\ox,w)\ne\emp$ and for each $u\in T_\O^2(\ox,w)$ there exist a number $\ve>0$ and an arc $\xi\colon[0,\ve]\to\O$ such that $\xi(0)=\ox$, $\xi'_+(0)=w$, and $\xi''_+(0)=u$ with
\begin{equation*}
\xi''_+(0):=\lim_{t\dn 0}\frac{\xi(t)-\xi(0)-t\xi'_+(0)}{\sm t^2}.
\end{equation*}
It is well known that if $\O\subset\R^n$ is convex and parabolically derivable at $\ox\in\O$ for every vector $w\in T_\O(\ox)$, then the second-order tangent set $T_\O^2(\ox,w)$ is a nonempty convex subset of $\R^n$.

\section{Parabolic Regularity and Twice Epi-Differentiability}\sce\label{sect03}

In this section we define and study the underlying notion of {\em parabolically regular sets} and reveal its role in the study of crucial second-order variational properties of extended-real-valued functions. It is also shown that parabolic regularity  provides a unified framework under  which {\em twice epi-differentiability of indicator functions} can be justified.
\vspace*{0.05in}

Let us start with the basic definition of parabolic regularity for sets.

\begin{Definition}[\bf parabolic regularity of sets]\label{par-reg} A nonempty set $\O\subset\R^n$ is {\sc parabolically regular} at $\ox\in\O$ for $\ov\in\R^n$ if for any $w\in\R^n$ with $\d^2\dd_\O(\ox,\ov)(w)<\infty$ there exist, among all the sequences $t_k\dn 0$ and $w_k\to w$ satisfying the condition
\begin{equation*}
\Delta_{t_k}^2\dd_\O(\bar x,\ov)(w_k)\to\d^2\dd_\O(\bar x,\ov)(w)\;\mbox{ as }\;k\to\infty,
\end{equation*}
those with the additional property that
\begin{equation}\label{par-reg1}
\limsup_{k\to\infty}\frac{\|w_k-w\|}{t_k}<\infty.
\end{equation}
\end{Definition}\vspace*{0.02in}

Parabolic regularity was introduced in \cite[Definition~13.65]{rw} for extended-real-valued functions but was not further elaborated either in \cite{rw}, or in subsequent publications. However, it was understood therein and discussed in the commentary section of \cite[Chapter~13]{rw} (p.\ 640) that such a second-order regularity in the functional framework has the potential for applications to second-order sufficient optimality conditions in terms of second subderivatives.

A different notion of second-order regularity for sets was introduced by Bonnans, Cominetti and Shapiro \cite[Definition~3]{bcs}. As explained in the commentaries to \cite[Chapter~13]{rw} (p.\ 640), the parabolic regularity from Definition~\ref{par-reg} is implied by the second-order regularity in the sense of \cite{bcs}. Furthermore, the example given in \cite[p.\ 215]{bs} shows that the converse implication fails in general. This tells us that the parabolic regularity from Definition~\ref{par-reg} is {\em strictly weaker} than the second-order regularity from \cite[Definition~3]{bcs}.\vspace*{0.05in}

To proceed further, recall from \cite[Definition~13.59]{rw} that  the {\em parabolic subderivative} of a proper function $\ph\colon\R^n\to\oR$ at $\ox\in\dom\ph$ relative to a vector $w\in\R^n$ with $\d\ph(\ox)(w)$  finite and a vector $z\in\R^n$ is defined by
\begin{equation}\label{par-subder}
\d^2\ph(\bar x)(w,z):=\liminf_{\substack{t\dn 0\\ u\to z}}\dfrac{\ph(\ox+tw+\sm t^2u)-\ph(\ox)-t\d\ph(\ox)(w)}{\sm t^2}.
\end{equation}
As shown in \cite[Proposition~13.64]{rw}, for any $\ox\in\dom\ph$ and for any $\ov$ and $w$ satisfying $\d\ph(\ox)(w)=\la\ov,w\ra$, we always have the following relationships between \eqref{ssd} and \eqref{par-subder}:
\begin{equation}\label{spr}
\d^2\ph(\bar x,\ov)(w)\le\inf_{z\in\R^n}\big\{\d^2\ph(\ox)(w,z)-\la\ov,z\ra\big\}.
\end{equation}

This paper is mostly devoted to  the case where $\ph$ is the indicator function of a set, and thus it is helpful to get an explicit set counterpart of \eqref{spr}. We can do this by using the well known construction of the critical cone associated with a given set $\O\subset\R^n$. Picking a pair $(\ox,\ov)\in\gph N_\O$, the {\em critical cone} to $\O$ at $(\ox,\ov)$ is defined by
\begin{equation}\label{crit1}
K_\O(\ox,\ov):=T_\O(\ox)\cap\{\ov\}^\perp
\end{equation}
via the tangent cone \eqref{tan} to $\O$ at $\ox$ and the orthogonal complement of $\ov$ in $\R^n$. The next proposition is instrumental to establish the main results of this section on calculating second subderivatives of indicator functions for parabolically derivable and parabolically regular sets with proving their twice epi-differentiability.

\begin{Proposition}[\bf relationship between second and parabolic subderivatives for sets]\label{pard} Let $\O\subset\R^n$ be closed set with $\ox\in\O$. Then for any $w\in K_\O(\ox,\ov)$ we always have the inequality
\begin{equation}\label{pr2}
\d^2\dd_\O(\bar x,\ov)(w)\le-\sigma_{T^2_\O(\ox,w)}(\ov),
\end{equation}
where $\sigma_{T^2_\O(\ox,w)}$ stands for the support function of the second-order tangent set $T^2_\O(\ox,w)$.
\end{Proposition}
\begin{proof} As mentioned above, $\d\dd_\O(\ox)=\dd_{T_\O(\ox)}$ for any point $\ox\in\O$. Thus we can equivalently express the critical cone \eqref{crit1} to $\O$ at $(\ox,\ov)$ by
\begin{equation*}
K_\O(\ox,\ov)=\big\{w\in\R^n\big|\;\d\dd_\O(\ox)(w)=\la\ov,w\ra\big\}.
\end{equation*}
Furthermore, it follows directly from the definition that $\d^2\dd_\O(\ox)(w,\cdot)=\dd_{T^2_\O(\ox,w)}(\cdot)$, and hence \eqref{pr2} is a consequence of \eqref{spr} for the case where $\ph=\dd_\O$.
\end{proof}

It is important for our subsequent results to find efficient conditions ensuring that \eqref{pr2} holds as equality. The first result in this direction was obtained by Rockafellar \cite[Proposition~3.5]{r88} who proved the equality in \eqref{pr2} for the class of convex piecewise linear-quadratic functions. Furthermore, it can be deduced from \cite[Theorem~4.5]{r88} that this equality holds for the large class of fully amenable functions introduced later in \cite{pr96} as compositions a piecewise linear-quadratic functions and ${\cal C}^2$-smooth mappings and the metric regularity qualification condition. As we show below, the equality in \eqref{pr2} is actually {\em equivalent} to the parabolic regularity of $\O$, which goes far beyond full amenability. Although we establish this result for sets, it can be derived for a larger class of extended-real-valued functions with the corresponding properties.

\begin{Theorem}[\bf second-order subderivatives and parabolic regularity for indicator functions]\label{pri} Let $\O\subset\R^n$ be a closed set with $\ox\in\O$, and let $\ov\in N_{\O}^p(\ox)$. Assume that $\O$ is parabolically derivable at $\ox$ for every vector $w\in K_\O(\ox,\ov)$. Then we have:

{\bf(i)} The second subderivative $\d^2\dd_\O(\bar x,\ov)$ is proper and l.s.c.\ on $\R^n$. Furthermore, there exists a number $r\ge 0$ such that $\d^2\dd_\O(\bar x,\ov)(w)\ge-r\|w\|^2$ for all $w\in\R^n$, which implies that
\begin{equation*}
\dom\d^2\dd_\O(\bar x,\ov)= K_\O(\ox,\ov).
\end{equation*}

{\bf(ii)} If $\O$ is parabolically regular at $\ox$ for $\ov$, then for any vector $w\in K_\O(\ox,\ov)$ there exists a second-order tangent $u\in T^2_\O(\bar x,w)$ such that
\begin{equation*}
\d^2\dd_\O(\bar x,\ov)(w)=-\sigma_{T^2_\O(\ox,w)}(\ov)=-\la\ov,u\ra.
\end{equation*}

{\bf(iii)} The set $\O$ is parabolically regular at $\ox$ for $\ov$ if and only if
\begin{equation}\label{pri1}
\d^2\dd_\O(\bar x,\ov)(w)=-\sigma_{T^2_\O(\ox,w)}(\ov)\;\mbox{ for all }\;w\in K_\O(\ox,\ov).
\end{equation}
\end{Theorem}
\begin{proof} We begin with verifying (i) and observe first that $\d^2\dd_\O(\ox,\ov)$ is l.s.c.\ due to Proposition~\ref{ssp}(i).
Since $\ov\in N^p_\O(\ox)$, it follows from definition \eqref{proxn} that there exists $r\ge 0$ such that
\begin{equation*}
\la\ov,x-\ox\ra\le\hbox{${r\over 2}$}\|x-\ox\|^2\;\mbox{ for all }\;x\in\O.
\end{equation*}
Picking further any $w\in\R^n$ together with $t\dn 0$ and $u\to w$, we get
\begin{equation*}
\Delta_t^2\dd_\O(\ox,\ov)(u)=\frac{\dd_\O(\ox+tu)-\dd_\O(\ox)-t\langle\ov,u\rangle}{\sm t^2}
\ge\begin{cases}
\infty&\mbox{if}\;\;\ox+tu\notin\O,\\
-r\|u\|^2&\mbox{if}\;\;\ox+tu\in\O.
\end{cases}
\end{equation*}
This implies by definition \eqref{ssd} of the second subderivative that
\begin{equation*}
\d^2\dd_\O(\ox,\ov)(w)\ge-r\|w\|^2,
\end{equation*}
and therefore that $\d^2\dd_\O(\ox,\ov)(0)\ge 0$. Since $\d^2\dd_\O(\ox,\ov)$ is positive homogeneous of degree 2 by Proposition~\ref{ssp}(ii),
we get $\d^2\dd_\O(\ox,\ov)(0)=0$, which  proves that $\d^2\dd_\O(\ox,\ov)$ is a proper function. Combining this with Proposition~\ref{ssp}(iv) yields
\begin{equation}\label{domd}
\dom\d^2\dd_\O(\bar x,\ov)\subset\big\{w\in\R^n\big|\:\d\dd_\O(\ox)(w)=\la\ov,w \ra\big\}= K_\O(\ox,\ov).
\end{equation}
To verify the opposite inclusion, we deduce from the parabolic derivability  of $\O$ at $\ox$ for $w$ that $T^2_\O(\ox,w)\ne\emp$ and thus $-\sigma_{T^2_\O(\ox,w)}<\infty$, which tells us together with \eqref{pr2} that $\d^2\dd_\O(\bar x,\ov)(w)$ is finite for all $w\in K_\O(\ox,\ov)$. Thus we arrive at the equality $\dom\d^2\dd_\O(\bar x,\ov)=K_\O(\ox ,\ov)$,
which completes the proof of (i).

To prove now part (ii), pick $w\in K_\O(\ox,\ov)$ and conclude from (i) that $\d^2\dd_\O(\bar x,\ov)(w)$ is finite. Thus it follows from parabolic regularity of $\O$ at $\ox$ for $\ov$ that there exist sequences $t_k\dn 0$ and $w_k\to w$ as $k\to\infty$ for which
\begin{equation}\label{pri4}
\Delta_{t_k}^2\dd_\O(\ox,\ov)(w_k)\to\d^2\dd_\O(\bar x,\ov)(w)\;\mbox{ and }\;\limsup_{k\to\infty}\frac{\|w_k-w\|}{t_k}<\infty.
\end{equation}
Using again that $\d^2\dd_\O(\bar x,\ov)(w)$ is finite, we have   $\ox+t_kw_k\in\O$ whenever $k\in\N$ is sufficiently large. The boundedness of the sequence $\big\{(w_k-w)/t_k\big\}$ by the assumed parabolic regularity leads us to the convergence of $(w_k-w)/\sm t_k\to u$ for some $u \in \R^n$ through passing to a convergent subsequence if it is necessary. This tells us that $t_k(w_k-w)-\sm t_k^2u=o(t_k^2)$, and so we arrive at
\begin{equation*}
\ox+t_kw+\sm t_k^2u+o(t_k^2)=\ox+t_kw_k\in\O\;\mbox{ for large }\;k\in\N,
\end{equation*}
which yields $u\in T^2_\O(\ox,w)$. Using this and the first condition in \eqref{pri4} together with the parabolic regularity of $\O$ at $\ox$ for $\ov$ brings us to the relationships
\begin{eqnarray*}\label{pri5}
\d^2\dd_\O(\bar x,\ov)(w)&=&\lim_{k\to \infty}\dfrac{\dd_\O(\ox+t_k w_k)-\dd_\O(\ox)-t_k\langle\ov,w_k\rangle}{\sm t_k^2}\\
&=&\lim_{k\to\infty}-\Big\la\ov,\dfrac{w_k-w}{\frac{1}{2}t_k}\Big\ra=-\la\ov,u\ra\ge-\sigma_{T^2_\O(\ox,w)}(\ov).
\end{eqnarray*}
Combining them with \eqref{pr2} verifies assertion (ii).

Turing to (iii), observe that the validity of \eqref{pri1} for $w\in K_\O(\ox,\ov)$ under the parabolic regularity of $\O$ at $\ox$ for $\ov$ was proved in (ii). To verify the opposite implication in (iii), suppose that \eqref{pri1} holds for all $w\in K_\O(\ox,\ov)$ and let $\d^2\dd_\O(\bar x,\ov)(w)<\infty$, i.e., $w\in\dom\d^2\dd_\O(\bar x,\ov)$. It follows from (i) that $w\in K_\O(\ox,\ov)$. Employing \cite[Proposition~13.64]{rw} yields
\begin{equation*}
\d^2\dd_\O(\bar x,\ov)(w)=-\sigma_{T^2_\O(\ox,w)}(\ov)=\liminf_{\substack{t\dn 0,\,u\to w\\
[u-w]/t\,\,{\ss\mbox{bounded}}}}\Delta_t^2\dd_\O(\ox,\ov)(u).
\end{equation*}
which shows that $\O$ is parabolically regular at $\ox$ for $\ov$ and hence completes the proof.
\end{proof}

When $\O$ is convex, the properness of $\d^2\dd_\O(\bar x,\ov)$ in Theorem~\ref{pri}(i) follows from \cite[Proposition~13.20(a)]{rw} since in this case we have $N_\O^p(\ox)=N_\O(\ox)$. The general nonconvex case of Theorem~\ref{pri} deals with normal vectors $\ov$ from the cone of proximal normals $N^p_\O(\ox)$, and it seems to be restrictive for some applications where we require parabolic regularity for all normal vectors from the basic normal cone $N_\O(\ox)$. This can be adjusted by narrowing our attention to some particular class of nonconvex sets for which we have $N^p_\O(\ox)=N_\O(\ox)$; see Proposition~\ref{norm} for a class of nonconvex sets enjoying this property. Recall that a set $\O\subset\R^n$ is called {\em prox-regular} at $\ox$ for $\ov$ with $(\ox,\ov)\in\gph N_\O$ if there exist $\ve>0$ and $r >0$ such that
\begin{equation}\label{proxr}
\la v,u-x\ra\le\hbox{${r\over 2}$}\|u-x\|^2\;\mbox{ whenever }\;(x,v)\in(\gph N_\O)\cap\B_\ve(\ox,\ov),\;u\in\O\cap\B_\ve(\ox).
\end{equation}
This notion was introduced in variational analysis by Poliquin and Rockafellar \cite{pr96}, but in fact it goes back to Federer \cite{fed} in geometric measure theory who called such sets as those with {\em positive reach}; see also \cite{col-thi} for further elaborations. Many important sets that are overwhelmingly encountered in variational analysis, optimization and their applications are prox-regular; see, e.g., \cite{col-thi,lev-pol-thi,loe,rw} and the references therein for more details.\vspace*{0.05in}

The obtained descriptions of parabolic regularity in Theorem~\ref{pri} help us to check that this fundamental property holds for many classes of sets important in applications. Let us start with {\em polyhedral convex sets}, which are intersections of finitely many half-spaces.

\begin{Example}[\bf parabolic regularity of polyhedral sets]\label{poly}{\rm Let $\O\subset\R^n$ be a polyhedral convex set with $\ox\in\O$, and let $\ov\in N_\O(\ox)$. We claim that $\O$ is parabolically regular at $\ox$ for $\ov\in N_\O(\ox)$. To check it, note first that $N_\O(\ox)=N^p_\O(\ox)$ by the convexity of $\O$ and deduce from \cite[Theorem~13.12]{rw} that $\O$ is parabolically derivable at $\ox$ for any vector $w\in T_\O(\ox)$. Thus we get by Theorem~\ref{pri}(i) that $\d^2\dd_\O(\bar x,\ov)(w)\ge 0$ for all $w\in\R^n$ and that $\dom\d^2\dd_\O(\bar x,\ov)=K_\O(\ox,\ov)$. Let us further show that
\begin{equation}\label{sdpo}
\d^2\dd_\O(\bar x,\ov)(w)=\dd_{K_\O(\ox,\ov)}(w)\;\mbox{ for all }\;w\in\R^n.
\end{equation}
To proceed, pick $w\in\dom\d^2\dd_\O(\bar x,\ov)$, which implies that $w\in T_\O(\ox)$. Appealing now to \cite[Exercise~6.47]{rw} ensures the existence of $\ve>0$ with $\ox+tw\in\O$ for all $t\in[0,\ve]$. Take a sequence $t_k\dn 0$ such that $t_k\in [0,\ve]$ and denote $w_k:=w$ for all $k\in\N$. Then we get
\begin{equation*}
0\le\d^2\dd_\O(\bar x,\ov)(w)\le\lim_{k\to\infty}\Delta_{t_k}^2\dd_\O(\ox,\ov)(w_k)=0,
\end{equation*}
which shows that $\Delta_{t_k}^2\dd_\O(\ox,\ov)(w_k)\to\d^2\dd_\O(\bar x,\ov)(w)$ as $k\to\infty$ and hence verifies \eqref{sdpo}. Since we obviously have \eqref{par-reg1} in this case, the parabolic regularity of the polyhedron $\O$ is verified.}
\end{Example}

\begin{Remark}[\bf parabolic regularity of unions of polyhedral sets]\label{poly-un} {\rm Arguing similarly to Example~\ref{poly}, we can show that if $\O$ is a finite union of polyhedral convex sets, then it is parabolically regular at $\ox\in \O$ for any $\ov\in N^p_\O(\ox)=\Hat N_\O(\ox)$. Observe that in this case we may have the {\em strict} inclusion $N^p_\O(\ox)\subset N_\O(\ox)$, and thus parabolic regularity is not achieved for any vector from the basic normal cone $N_\O(\ox)$.}
\end{Remark}

Other particular classes of parabolically regular sets are discussed below, where we also show that the property of parabolic regularity is preserved under various operations performed on sets.\vspace*{0.05in}

The next theorem reveals that the parabolic regularity of a closed set always yields the proper twice epi-differentiability of its indicator function with an explicit formula for computing the corresponding second subderivative.

\begin{Theorem}[\bf twice epi-differentiability from parabolic regularity]\label{pri92} Let $\O$ be a closed subset of $\R^n$ with $\ox\in\O$, and let $\ov\in N_{\O}^p(\ox)$. Assume further that  $\O$ is parabolically derivable at $\ox$ for every vector $w\in K_\O(\ox,\ov)$. If $\O$ is parabolically regular at $\ox$ for $\ov$, then it is properly twice epi-differentiable at $\ox$ for this normal vector and its second subderivative is computed by
\begin{equation*}
\d^2\dd_\O(\bar x,\ov)(w)=\left\{\begin{array}{ll}
-\sigma_{T^2_\O(\bar x,w)}(\ov)\;&\mbox{if }\;\;w\in K_\O(\ox,\ov),\\
\infty&\mbox{otherwise}.
\end{array}
\right.
\end{equation*}
\end{Theorem}
\begin{proof} The second subderivative formula follows from Theorem~\ref{pri}(iii). To establish the claimed twice epi-differentiability,  pick any $w\in\dom\d^2\dd_\O(\bar x,\ox^*)=K_\O(\ox,\ov)$. Theorem~\ref{pri}(i) ensures the existence of $u\in T^2_\O(\ox,w)$ with $\d^2\dd_\O(\bar x,\ov)(w)=-\la\ov,u\ra$. Using the parabolic derivability of $\O$ at $\ox$ for $w$, we find a number $\ve>0$ and an arc $\xi\colon[0,\ve]\to\O$ satisfying
\begin{equation*}
\xi(0)=\ox, \;\xi'_{+}(0)=w,\;\mbox{ and }\;\xi''_{+}(0)=u.
\end{equation*}
Define now $w_t:=\disp\frac{\xi(t)-\xi(0)}{t}$ for all $t\in[0,\ve]$ and get $\ox+tw_t=\xi(t)\in\O$ whenever $t\in[0,\ve]$. Thus we have $w_t\to w$ as $t\dn 0$. On the other hand, we deduce from $w\in K_\O(\ox,\ov)$ that $\la\ov,w\ra=0$, which gives us in turn that
\begin{eqnarray*}
\Delta_t^2\dd_\O(\bar x,\ov)(w_t)&=&\dfrac{\dd_\O(\ox+tw_t)-\dd_\O(\ox)-t\langle\ov,w_t\rangle}{\sm t^2}\\
&=&-\Big\la\ov,\dfrac{\xi(t)-\xi(0)-t w}{\sm t^2}\Big\ra\to - \la\ov,u\ra=\d^2\dd_\O(\bar x,\ov)(w).
\end{eqnarray*}
This justifies \eqref{df02} when $w\in K_\O(\ox,\ov)$. If $w\notin K_\O(\ox,\ov)$, consider the arc $\xi(t):=\ox+tw$ for all $t>0$. We clearly have $\xi(0)=\ox$ and $\xi'_{+}(0)=w$. Put $w_t:=\dfrac{\xi(t)-\xi(0)}{t}=w$ and observe that
\begin{equation*}
\infty=\d^2\dd_\O(\bar x,\ov)(w)\le\liminf_{t\dn 0}\Delta_t^2\dd_\O(\bar x,\ov)(w_t)\le\limsup_{t\dn 0}\Delta_t^2\dd_\O(\bar x,\ov)(w_t)\le\infty=\d^2\dd_\O(\bar x,\ov)(w),
\end{equation*}
which verifies \eqref{df02} for such a vector $w$ and hence completes the proof of the theorem.
\end{proof}

To the best our knowledge, the above theorem is the first result in the literature establishing a {\em systematic approach} to verify twice epi-differentiability of set indicator functions via parabolic regularity. This approach allows us to justify in the next section the twice epi-differentiability of various important classes of nonpolyhedral sets that naturally and frequently appear in the framework of constrained optimization.

Recall here the result of \cite[Theorem~7.2]{bcs2} telling us that the projection mapping for a convex set, which is second-order order regular in the sense therein, is directionally differentiable. This result in combination with \cite[Corollary~13.43(c)]{rw} ensures that the indicator function of such a set is twice epi-differentiable. The only known fact concerning twice epi-differentiability of indicator functions for nonconvex sets was established in \cite[Corollary~13.43(d)]{rw} by showing that fully amenable sets enjoy this property.\vspace*{0.05in}

We conclude this section by revealing, via the usage of Theorem~\ref{pri92}, a connection between parabolic regularity of sets and the proto-differentiability property of the associated normal cone mappings. Recall that the normal cone mapping $N_\O$ is {\em proto-differentiable} at $\ox$ for $\ov\in N_\O(\ox)$ if the set $\gph N_\O$ is geometrically derivable at $(\ox,\ov)$. The proto-differentiability notion for set-valued mappings was introduced by Rockafellar in \cite{r89} and since that has drawn much attention in variational analysis and applications; see, e.g., the recent paper \cite{ab} and the references therein. As proved in \cite{r89}, the normal cone mappings associated with fully amenable sets are always proto-differentiable. We show in what follows that this result can be extended to a much broader class of parabolically regular sets.

To proceed in this direction, recall that a single-valued mapping $f\colon\R^n\to\R^m$ is {\em semidifferentiable} at $\ox\in\R^n$ if the limit
\begin{equation}\label{semid}
\lim_{\substack{t\dn 0\\u\to w}}\frac{f(\ox+tu)-f(\ox)}{t}
\end{equation}
exists for any $w\in\R^n$. It is easy to check that if $f$ is Lipschitz continuous around $\ox$, then its semidifferentiability  at this point is equivalent to its directional differentiability at $\ox$ in the classical sense, i.e., to the existence of the one-sided limit
\begin{equation*}
\lim_{t\downarrow 0}\frac{f(\ox+t w)-f(\ox)}{t}\;\mbox{ for all }\;w\in\R^n.
\end{equation*}

Now we are ready to establish the aforementioned result on the proto-differentiability of normal cone mappings associated with parabolically regular sets. Note that assertion (ii) of the following theorem concerns the graphical derivative \eqref{gder} of the normal cone mapping. This construction is a set specification of the primal-dual second-order generalized derivative for extended-real-valued functions, which is known in variational analysis and optimization as the {\em subgradient graphical derivative}; see, e.g. \cite{mor18}.

\begin{Theorem}[\bf proto-differentiability of normal cone mappings for parabolically regular sets]\label{proto} Let $\O$ be a closed subset of $\R^n$ with $\ox\in\O$, and let $\ov\in N_{\O}(\ox)$. Assume further that $\O$ is parabolically derivable at $\ox$ for every vector $w\in K_\O(\ox,\ov)$, and that $\O$ is prox-regular and parabolically regular at $\ox$ for $\ov$. Then the following equivalent conditions hold:

{\bf(i)} The indicator function $\dd_\O$ is twice epi-differentiable at $\ox$ for $\ov$.

{\bf(ii)} The normal cone mapping $N_\O$ is proto-differentiable at $\ox$ for $\ov$, and we have the subgradient graphical derivative representation
\begin{equation}\label{gdpd}
DN_\O(\ox,\ov)(w)=\sub\big(\sm\d^2\dd_\O(\ox,\ov)\big)(w)\;\mbox{ for all }\;w\in \R^n.
\end{equation}

{\bf(iii)} For any $r>0$ sufficiently small the projection mapping $P_\O$ is single-valued and semidifferentiable at $\ou:=\ox+r\ov$ with $DP_\O(\ou)(w)=\big(I+rDN_\O(\ox,\ov)\big)^{-1}(w)$ for all $w\in\R^n$.
\end{Theorem}
\begin{proof} The imposed prox-regularity of $\O$ ensures that $\ov\in N_\O^p(\ox)$. Hence assertion (i) is a direct consequence of Theorem~\ref{pri92}. Furthermore, it follows from \cite[Theorem~13.40]{rw} that the twice epi-differentiability of $\dd_\O$ is equivalent to the proto-differentiability of $N_\O$. This justifies the equivalence between (i) and (ii). Moreover, equality \eqref{gdpd} comes also from \cite[Theorem~13.40]{rw}.

To verify (iii), observe first that it follows from \cite[Theorem~13.37]{rw} that for any $r>0$ sufficiently small we find a neighborhood $U$ of $\ou=\ox+r\ov$ with
\begin{equation}\label{proj}
P_{\O}(u)=(I+rT)^{-1}(u)\quad \mbox{for all}\;\; u\in U,
\end{equation}
and that $P_\O(u)$ is single-valued for any $u\in U$, where $T$ stands for a graphical localization of $N_\O$ around $(\ox,\ov)$.  It is not hard to see that the latter equality for the projection mapping of $\O$ gives us the claimed formula for the graphical derivative of $P_\O$. We now proceed to show $P_\O$ is semidifferentiable at $\ou$.

To this end, since $N_\O$ is proto-differentiable at $\ox$ for $\ov$, so is $(I+rT)^{-1}$ at $\ou$ for $\ox$, which verifies this property for $P_\O$. Employing now \cite[Proposition~9.50]{rw} together with the Lipschitz continuity of $P_\O$ around $\ou$ (taken, e.g., from \cite[Proposition~13.37]{rw}) justifies that $P_\O$ is semidifferentiable at $\ou$. This shows that implication (ii)$\implies$(iii) holds. A similar argument as above via \eqref{proj} justifies  the opposite implication (iii)$\implies$(ii), which thus completes the proof of the theorem.
\end{proof}

Note finally that the last assertion (iii) of Theorem~\ref{proto} provides a far-going extension of a well known result for convex sets. Indeed, it is proved in \cite[Theorem~7.2]{bcs2} that the projection mapping associated with a second-order regular convex set is in fact directionally differentiable. As mentioned earlier, for Lipschitz continuous mappings the semidifferentiability and directional differentiability notions agree. Hence Theorem~\ref{proto}(iii) significantly extends the aforementioned result for convex sets to the general case of prox-regular sets under parabolic regularity.
Observe that when $\O$ is convex in Theorem~\ref{proto}, we can simply let $r=1$ in (iii).

\section{Second-Order Tangents under Metric Subregularity}\label{sect04}\sce

After revealing in the previous section general properties of parabolically regular sets and establishing close relations of parabolic regularity with second subderivatives and twice epi-differentiability of indicator functions, in what follows we intend to develop basic {\em calculus rules} ensuring the {\em preservation} of parabolic regularity under various operations on sets together with extensive {\em chain rules} for the corresponding constructions of {\em second-order generalized differentiation}. This would allow us, on the one hand, to largely extend the collection of sets that occur to be parabolically regular while, on the other hand, to derive new calculus rules of second-order generalized differentiation under the most appropriate qualification conditions.

The main class of sets $\O\subset\R^n$ under our subsequent consideration are represented locally in the following form referred to as {\em constraint systems}, which naturally appear, e.g., in  constrained optimization: there exist a neighborhood ${\cal O}$ of $\ox$, a single-valued mapping $f\colon\R^n\to\R^m$ twice differentiable at $\ox$, and a closed subset $\Th$ of $\R^m$ such that
\begin{equation}\label{CS}
\O\cap{\cal O}=\big\{x\in{\cal O}\big|\;f(x)\in\Th\big\}.
\end{equation}
Second-order variational analysis of the constraint systems from \eqref{CS} always requires some constraint qualifications. The following mild one is used throughout the rest of the paper.

\begin{Definition}[\bf metric subregularity constraint qualification]\label{defmscq} Let $\O$ be locally represented as \eqref{CS} around a point $\ox\in\O$. We say that the {\sc metric subregularity constraint qualification} $($MSCQ$)$ holds for $\O$ at $\ox$ with modulus $\kappa>0$ if the mapping $x\mapsto f(x)-\Th$ is metrically subregular at $(\ox,0)$ with this modulus.
\end{Definition}

Observe that MSCQ at $\ox$ with modulus $\kappa$ for the constraint system \eqref{CS} can be equivalently described as the existence of a neighborhood $U$ of $\ox$ such that the distance estimate
\begin{equation}\label{mscq}
{\rm dist}(x;\O)\le\kappa\,{\rm dist}\big(f(x);\Th\big)\;\mbox{ for all }\;x\in U
\end{equation}
holds. It is clear that MSCQ is strictly (may be very significantly) weaker than the {\em metric regularity  constraint qualification} (MRCQ) for $\O$ at $\ox$, which corresponds to Definition~\ref{defmscq} with the replacement  of the metric subregularity of the mapping $x\mapsto f(x)-\Th$ at $(\ox,0)$ by the metric regularity of this mapping around the reference point. In contrast to MSCQ, the latter MRCQ condition admits a {\em complete pointwise characterization} via the coderivative criterion \eqref{cod-cr}, which in the case of the mapping $x\mapsto f(x)-\Th$ can be equivalently written as
\begin{equation}\label{mrcq}
N_ \Th\big(f(\ox)\big)\cap\ker\nabla f(\ox)^*=\{0\}
\end{equation}
in terms of the basic normal cone \eqref{2.4}. This condition, known as  the {\em basic constraint qualification},  has been used in numerous aspects of variational analysis and constrained optimization; see, e.g., \cite{mor06,mor18,rw} with the commentaries and references therein. When $\Th$ and $f$ are given in particular settings, \eqref{mrcq} reduces to familiar classical forms of constraint qualifications, e.g., to MFCQ in NLPs, to Robinson's constraint qualification in conic programming, etc.

The  MSCQ, however, turns out to be more subtle and challenging, and so far  no pointwise characterization for this property has been achieved. In \cite{chi,chnt,gf,gm17,go16,hms,hjo,ho,mm,mms} and the bibliographies therein the reader can find a number of constructive sufficient conditions for the validity of MSCQ with their important applications. Mentioning this, we are positive that MSCQ has strong potential for further developments and applications in variational analysis and optimization. New second-order ones are presented below in rather general settings.\vspace*{0.05in}

The main attention of this section is to establish the {\em preservation of parabolic derivability} for constraint systems via a chain rule for second-order tangent sets under MSCQ. This will be strongly used in the subsequent material. To begin, we recall the required first-order chain rules for tangents and normals to nonconvex sets under MSCQ. Both chain rules presented in the following proposition are consequences of essentially more general ones from \cite{mms} (Theorems~3.3 and 3.5, respectively), where the reader can find references to previous results in this direction.

\begin{Proposition}[\bf chain rules for first-order tangents and normals]\label{2chin} Let $\O$ be taken from \eqref{CS}, and let $\ox\in\O$ with $f(\ox)\in\Th$. If MSCQ holds for $\O$ at $\ox$ and if $\Th$ is normally regular at $f(\ox)$, then we have the equalities
\begin{equation}\label{first}
T_\O(\ox)=\big\{w \in\R^n\big|\;\nabla f(\ox)w\in T_\Th\big(f(\ox)\big)\big\}\;\mbox{ and }\;N_\O(\ox)=\widehat{N}_\O(\ox)=\nabla f(\ox)^*N_\Th\big(f(\ox)\big).
\end{equation}
\end{Proposition}

To proceed with our tangential second-order analysis, we use the first equality in \eqref{first} and for each $w\in\R^n$ satisfying $\nabla f(\ox)w\in T_\Th(f(\ox))$ define the parameterized set-valued mapping $S_w\colon\R^m\tto\R^n$ involving the second-order tangent set \eqref{2tan} by
\begin{equation}\label{mapt}
S_w(p):=\big\{u\in\R^n\big|\;\nabla f(\bar x)u+\nabla^2f(\bar x)(w,w)+p\in T^2_\Th\big(f(\bar x),\nabla f(\ox)w\big)\big\}.
\end{equation}
This mapping describes a canonically perturbed {\em second-order tangential approximation} of the constraint system \eqref{CS}. The next result of its own interest proves under MSCQ the {\em uniform outer/upper Lipschitz} property of \eqref{mapt} in the sense of Robinson \cite{rob} broadly employed below.

\begin{Theorem}[\bf uniform outer Lipschitzian property of second-order tangential approximations]\label{fors} Let $\O$ be a constraint system represented by \eqref{CS} around $\ox\in\O$, and let $w\in\R^n$ be such that $\nabla f(\ox)w\in T_\Th(f(\ox))$. Assume that MSCQ holds for $\O$ at $\ox$ with modulus $\kappa>0$.   Then the approximating mapping \eqref{mapt} satisfies the inclusion
\begin{equation}\label{tlip}
S_w(p)\subset S_w(0)+\kappa\|p\|\B\;\mbox{ for all }\;p\in\R^m\;\mbox{ uniformly in }\;w,
\end{equation}
which means the uniform outer Lipschitzian property of $S_w$ at the origin.
\end{Theorem}
\begin{proof}
Fixing some $p\in\R^n$ and $u\in S_w(p)$, we get by \eqref{mapt} that
\begin{equation*}
\nabla f(\bar x)u+\nabla^2f(\bar x)(w,w)+p\in T^2_\Th\big(f(\bar x),\nabla f(\ox)w\big).
\end{equation*}
We deduce from definition \eqref{2tan} of second-order tangents that there exists a sequence $t_k\dn 0$ with
\begin{equation*}
f(\ox)+t_k\nabla f(\ox)w+\sm t_k^2\big(\nabla f(\bar x)u+\nabla^2f(\bar x)(w,w)+p\big)+o(t^2_k)\in\Th,\quad k\in\N.
\end{equation*}
For any $k$ sufficiently large we get by the twice differentiability of $f$ at $\ox$ that
\begin{equation*}
f\big(\ox+t_k w+\sm t_k^2u\big)=f(\ox)+t_k\nabla f(\ox)w+\sm t_k^2\big(\nabla f(\ox)u+\nabla^2f(\bar x)(w,w)\big)+o(t_k^2),
\end{equation*}
which in turn implies via MSCQ \eqref{mscq} that
\begin{eqnarray*}\label{pl01}
{\rm dist}\big(\ox+t_k w+\sm t_k^2u;\O\big)&\le&\kappa\,{\rm dist}\big(f(\ox+t_k w+\sm t_k^2u);\Th\big)\\
&\le&\frac{1}{2}\kappa t_k^2\Big(\|p\|+\frac{o(t_k^2)}{t_k^2}\Big).
\end{eqnarray*}
Thus there exists a vector $y_k\in\O$ satisfying
\begin{equation*}
\|d_k\|\le\frac{1}{2}\kappa\Big(\|p\|+\frac{o(t_k^2)}{t_k^2}\Big)\;\mbox{ with }\;d_k:=\frac{\ox+t_k w+\sm t_k^2u-y_k}{t_k^2}.
\end{equation*}
Passing to a subsequence if necessary ensures the existence of $d\in\R^n$ such that $d_k\to d$ as $k\to\infty$. This yields the estimate
\begin{equation}\label{pl02}
\|d\|\le\sm\kappa\|p\|.
\end{equation}
On the other hand, we can suppose without loss of generality that $\ox+t_k w+\sm  t_k^2 u-t_k^2 d_k=y_k\in\O\cap{\cal O}$ for $k$ sufficiently large, and hence it follows from \eqref{CS} that  $f(\ox+t_k w+\frac{1}{2}t_k^2u-t_k^2 d_k)\in\Th$. Taking into account the representation
\begin{equation*}
f\big(\ox+t_k w+\sm t_k^2u-t_k^2 d_k\big)=f(\ox)+t_k\nabla f\big(\ox)w+\sm t_k^2\big(\nabla f(\ox)(u-2d_k)+\nabla^2f(\bar x)(w,w)\big)+o(t_k^2),
\end{equation*}
we readily arrive at the inclusion
\begin{equation*}
f(\ox)+t_k\nabla f(\ox)w+\frac{1}{2}t_k^2\Big(\nabla f(\ox)(u-2d_k)+\nabla^2f(\bar x)(w,w)+\frac{2 o(t_k^2)}{t_k^2}\Big)\in\Th,
\end{equation*}
which in turn implies that $\nabla f(\ox)(u-2d)+\nabla^2f(\bar x)(w,w)\in T^2_\Th(f(\bar x),\nabla f(\ox)w)$. The latter reads as $u-2d\in S_w(0)$, which
together with \eqref{pl02} justifies the claimed inclusion \eqref{tlip} that gives us the uniform outer Lipschitzian property of the mapping $S_w$ from \eqref{mapt} at $p=0$.
\end{proof}

Let us make some comments to the second-order result obtained in Theorem~\ref{fors}.

\begin{Remark}[\bf discussions on the outer Lipschitzian property]\label{olp-comm}{\rm The following hold:

{\bf(i)} The result of Theorem~\ref{fors} reduces to \cite[Proposition~3.1]{gm17} in the case where the set $\Th$ is a closed convex cone and $f(\ox)=0$; neither of these conditions is in our assumptions. Indeed, we can easily observe that the assumptions of \cite{gm17} ensure that $T^2_\Th(f(\bar x),\nabla f(\ox)w)= T_\Th(\nabla f(\ox)w)$, which allows us to derive the result of  \cite[Proposition~3.1]{gm17} from Theorem~\ref{fors}.

{\bf(ii)} Although Theorem~\ref{fors} is verified for vectors $w\in\R^n$ with $\nabla f(\ox) w\in T_\Th(f(\ox))$, it is clear that the outer Lipschitzian property \eqref{tlip} holds in fact for all vectors $w\in\R^n$. To check  \eqref{tlip} for $w$ with $\nabla f(\ox)w\not\in T_\Th(f(\ox))$, we observe directly from the definition that  $T^2_\Th(f(\bar x),\nabla f(\ox)w )=\emp$ and hence $S_w(p)=\emp$ for all $p\in\R^m$. This clearly yields \eqref{tlip}.

{\bf(iii)} Note finally that Theorem~\ref{fors} implies the outer Lipschitzian property of the mapping
\begin{equation*}
p\mapsto\big\{w\in\R^n\big|\;\nabla f(\bar x)w+p\in T_\Th\big(f(\bar x)\big)\big\}.
\end{equation*}
This can be easily deduced from Theorem~\ref{fors} by letting $w=0\in\R^n$ and by observing that $T^2_\Th(f(\bar x),0)=T_\Th(f(\bar x))$. This was already observed at  \cite[Proposition~2.1]{gf}.}
\end{Remark}

We are now ready to provide an application of Theorem~\ref{fors} to establishing the parabolic derivability of constraint systems \eqref{CS} via a chain rule for second-order tangent sets under MSCQ \eqref{mscq}. Such a chain rule for \eqref{CS} was obtained in \cite[Proposition~13.13]{rw} and also in \cite[Proposition~3.33]{bs} under the much stronger metric regularity condition for the mapping $x\mapsto f(x)-\Th$ around $(\ox,0)$. Furthermore, the latter result requires the ${\cal C}^2$-smooth property of $f$ around $\ox$, which we replace by the twice differentiability of $f$ at $\ox$ under MSCQ.

\begin{Theorem}[\bf parabolic derivability of constraint systems]\label{chain} Let $\O$ admit representation \eqref{CS} around $\ox\in\O$, let MSCQ \eqref{mscq} hold for $\O$ at $\ox$ with modulus $\kappa>0$, and let $\Th$ be normally regular at $f(\ox)$. Then for all $w\in T_\O(\ox)$ we have the second-order tangent chain rule
\begin{equation}\label{chr2}
u\in T^2_\O(\ox,w)\Longleftrightarrow\nabla f(\ox)u+\nabla^2f(\bar x)(w,w)\in T^2_\Th\big(f(\bar x),\nabla f(\ox)w\big).
\end{equation}
If furthermore the set $\Th$ is parabolically derivable at $f(\ox)$ for $\nabla f(\ox)w$, then the constraint system \eqref{CS} is parabolically derivable at $\ox$ for $w$.
\end{Theorem}
\begin{proof} By a close look at the proof of \eqref{chr2}, which was given in \cite[Proposition~13.13]{rw} under the metric regularity property of the mapping $x\mapsto f(x)-\Th$ around $(\ox,0)$, we can observe that it actually utilizes merely MSCQ at this point.

To verify the claimed parabolic derivability of the constraint system \eqref{CS} under the assumptions made, pick any $w\in T_\O(\ox)$ and recall that $T^2_\O(\ox,w)$ in \eqref{2tan} can be reformulated via the outer limit \eqref{pk} of the sets $(\O-\ox-tw)/\sm t^2$ as $t\dn 0$. The first requirement of parabolic derivability is to show that this outer limit is actually achieved as the {\em full set limit} meaning that the outer and inner limits agree. This again
can be done by following the proof of \cite[Proposition~13.13]{rw}, which basically works under MSCQ. The second requirement of parabolic derivability is crucial: to show that $T^2_\O(\ox,w)\ne\emp$ for any tangent vector $w\in T_\O(\ox)$. The proof of the latter fact given in \cite{rw} heavily exploits the metric regularity of the constraint mapping and does not hold under MSCQ. Now we provide a new proof for this property, which needs merely MSCQ.

To proceed, employ the imposed parabolic derivability of $\Th$ at $f(\ox)$ for $\nabla f(\ox)w$ to conclude that $T^2_\Th(f(\bar x),\nabla f(\ox)w)\ne\emp$. Picking $z\in T^2_\Th(f(\bar x),\nabla f(\ox)w)$ gives us the inclusion
\begin{equation*}
\nabla f(\ox)u+\nabla^2f(\bar x)(w,w)+p\in T^2_\Th\big(f(\bar x),\nabla f(\ox)w\big)\;\mbox{ with }\;p:=z-\nabla f(\ox)u-\nabla^2f(\bar x)(w,w),
\end{equation*}
which can be equivalently expressed as $u\in S_w(p)$ via the mapping $S_w$ from \eqref{mapt}. Now we apply Theorem~\ref{fors} and deduce from the outer Lipschitzian property \eqref{tlip} that there exists a vector $\tilde{u}\in S_w(0)$ such that $\|u-\tilde{u}\|\le\kappa\|p\|$. This tells us that
\begin{equation*}
\nabla f(\ox)\tilde{u}+\nabla^2f(\bar x)(w,w)\in T^2_\Th\big(f(\bar x),\nabla f(\ox)w\big).
\end{equation*}
Using the chain rule \eqref{chr2} leads us to $\tilde{u}\in T^2_\O(\ox,w)$, which verifies the nonemptiness of the second-order tangent set $T^2_\O(\ox,w)$ and thus completes the proof of the theorem.
\end{proof}

\section{Second Subderivatives under Parabolic Regularity}\sce\label{sect05}

This section is devoted to the study of the second subderivative \eqref{ssd} for the indicator functions $\dd_\O$ of parabolically regular constraint systems \eqref{CS}. The main goals here are the following:
\begin{itemize}[noitemsep]
\item To compute the second subderivative of $\dd_\O$ when $\Th$ in \eqref{CS} is parabolically regular.
\item To show that $\dd_\O$ is twice epi-differentiable when $\Th$ in \eqref{CS} is parabolically regular.
\end{itemize}

The obtained results have many consequences in what follows. In particular, they are instrumental for deriving rules for the preservation of parabolic regularity under major operations on sets; we label such rules as {\em calculus of parabolic regularity}. This calculus allows us to establish parabolic regularity for important classes of constraint systems that overwhelmingly encountered in variational analysis and optimization.

To achieve these goals, we begin with a simple albeit useful technical result. Recall \cite[p.\ 322]{rw} that a function $\ph\colon\R^n\to\oR$ is said to be {\em calm at $\ox$ from below} with constant $\ell\ge 0$ if $\ph(\ox)$ is finite and there exists a neighborhood $U$ of $\ox$ such that
\begin{equation}\label{calmb}
\ph(x)\ge\ph(\ox)-\ell\,\|x-\ox\|\;\mbox{ for all }\;x\in U.
\end{equation}
As proved in \cite[Propositon~8.32]{rw}, $\ph$ is calm at $\ox$ from below if and only if $\d\ph(\ox)(0)=0$, or equivalently $\d\ph(\ox)(w)>-\infty$ for all $w\in\R^n$. Furthermore, it is shown therein that for functions $\ph$, which are l.s.c.\ around $\ox$ and such that $\epi\ph$ is normally regular at $(\ox,\ph(\ox))$, their calmness at $\ox$ from below amounts to saying that $\sub\ph(\ox)\ne\emp$. Now we recover this result for convex functions using a different approach and show that the lower semicontinuity can be dropped.

\begin{Proposition}[\bf subdifferentiability of calm convex functions]\label{nemp} Let $\ph\colon\R^n\to\oR$ be convex and calm at $\ox$ from below with some constant $\ell\ge 0$. Then there exists a subgradient $\ov\in\sub\ph(\ox)$ such that $\|\ov\|\le\ell$.
\end{Proposition}
\begin{proof} Define the function $\psi\colon\R^n\to\oR$ by $\psi(x):=\ph(x)+\ell\|x-\ox\|$ and note that it is convex.  According to \eqref{calmb}, $\ox$ is  a local minimizer for $\ph$, and thus $0\in\sub\psi(\ox)$ by the generalized Fermat stationary rule. Employing the classical subdifferential sum rule of convex analysis, we get
\begin{equation*}
0\in\sub\psi(\ox)=\sub\ph(\ox)+\ell\B,
\end{equation*}
which clearly ensures the existence of a subgradient $\ov\in\sub\ph(\ox)$ with $\|\ov\|\le\ell$.
\end{proof}

Recall that our main results in Section~\ref{sect03} require that $\ov$ be a proximal normal to the set in question. As explained therein, in many applications we need similar results for any normal vectors within $N_\O$, and this is achieved when $\O$ is a constraint system in the sense of \eqref{CS}. This was known when the function $f$ in \eqref{CS} is ${\cal C}^2$-smooth and the basic constraint qualification \eqref{mrcq} fulfills. As shown below, this property still holds for constraint systems when $f$ is merely twice differentiable and MSCQ is satisfied. The following result is a consequence of \cite[Theorem~3.5]{mms}, where the reader can see its detailed proof.

\begin{Proposition}[\bf normal regularity of constraint systems]\label{norm} Let $\O$ admit representation \eqref{CS} around $\ox\in\O$, let $\Th$ be convex, and let MSCQ hold for $\O$ at $\ox$. Then $N^p_\O(\ox)=N_\O(\ox)$.
\end{Proposition}

After these preparations, we are in a position to evaluate the second subderivative of $\dd_\O$ for the constraint system \eqref{CS}. Fix $(\ox,\ov)\in\gph N_\O$ and define the set of {\em Lagrange multipliers} associated with the pair $(\ox,\ov)$ by
\begin{equation}\label{lagn}
\Lambda(\ox,\ov):=\left\{\lambda\in N_\Th\big(f(\ox)\big)\big|\;\nabla f(\ox)^*\lambda=\ov\right\}.
\end{equation}

The basic assumptions for this section and the subsequent material are as follows:\\[1ex]
{\bf(H1)} The set $\O$ has representation \eqref{CS} around $\ox\in\O$, and $\ov\in N_\O(\ox)$.\\
{\bf(H2)} The set $\Th$ in \eqref{CS} is convex, and the mapping $f$ is twice differentiable at $\ox$.\\
{\bf(H3)} The metric subregularity constraint qualification holds for $\O$ at $\ox$ with modulus $\kappa>0$.\\
{\bf(H4)} For every $\olm\in\Lambda(\ox,\ov)$ the set $\Th$ is parabolically derivable at $f(\ox)$ for all vectors $\nabla f(\ox)w$ in the critical cone $K_\Th(f(\ox),\olm)$ from \eqref{crit1}.\vspace*{0.05in}

We briefly comment on the imposed basic assumptions. The ones in (H1) and (H2) are self-evident. Assumption (H4) is satisfied for virtually all of the important sets used in constrained optimization. They include, in particular, polyhedral sets (Example~\ref{poly}), the second-order cone (Example~\ref{ice}), and the cone of positive semidefinite matrices. The MSCQ property in (H3) was discussed above and if it holds, then (H4) is equivalent to saying that $\O$ in \eqref{CS} is parabolically derivable at $\ox$  for all the vectors within $K_\O(\ox,\ov)$. This follows from Theorem~\ref{chain}.\vspace*{0.05in}

Let us proceed by highlighting some useful lower and upper estimates for the second subderivative of $\dd_\O$ that are derived by employing the results of Section~\ref{sect03}.

\begin{Proposition}[\bf estimates for second subderivatives]\label{ule} The following hold:

{\bf(i)} Under the validity of the basic assumptions in {\rm(H1)--(H3)}, for all $w\in\R^n$ we have the lower estimate of the second subderivative
\begin{equation}\label{dine2}
\d^2\dd_\O (\ox,\ov)(w)\ge\sup_{\lm\in\Lambda(\ox,\ov)}\;\big\{\langle\lm,\nabla^2f(\bar x)(w,w)\rangle+\d^2\dd_\Th\big(f(\bar x),\lm\big)\big(\nabla f(\ox)w\big)\big\}.
\end{equation}

{\bf(ii)} If in addition {\rm(H4)} is satisfied, then the second subderivative $\d^2\dd_\O (\ox,\ov)$ is a proper l.s.c.\ function with its domain calculated by
\begin{equation}\label{critco}
\dom\d^2\dd_\O(\ox,\ov)=K_\O(\ox,\ov)=T_\O(\ox)\cap\{\ov\}^\perp.
\end{equation}
Furthermore, for every $w\in K_\O(\ox,\ov)$ we have the upper estimates in the form
\begin{eqnarray*}\label{dine}
-\infty<\d^2\dd_\O(\ox,\ov)(w)&\le&-\sigma_{T^2_\O(\ox,w)}(\ov)=\inf\big\{-\la\ov,u\ra\big|\;u\in T^2_\O(\bar x,w)\big\}\\
&=&\inf\big\{-\la\ov,u\ra\big|\;\nabla f(\bar x)u+\nabla^2f(\bar x)(w,w)\in T^2_\Th\big(f(\bar x),\nabla f(\ox)w\big)\big\}<\infty.
\end{eqnarray*}
\end{Proposition}
\begin{proof}
The lower estimate \eqref{dine2} can be directly verified  by following the proof of \cite[Theorem~13.14]{rw} with the replacement of the metric regularity assumption on the constraint mapping $x\mapsto f(x)-\Th$ around $(\ox,0)$ by our MSCQ from (H3).

To justify (ii), we first use Proposition~\ref{norm} ensuring that $N^p_\O(\ox)=N_\O(\ox)$. This allows us to employ Theorem~\ref{pri}(i) and get \eqref{critco}. The second inequality in (ii) follows from Proposition~\ref{pard}, and then the equality therein is due to the second-order chain rule from \eqref{chr2} applied to the infimum representation of the negative support function. Finally, the MSCQ assumption (H3) ensures by Proposition~\ref{2chin} the equivalence
\begin{equation}\label{krie}
w\in K_\O(\ox,\ov)\iff\nabla f(\ox)w\in K_\Th\big(f(\ox),\olm\big)\;\mbox{ for any }\;\olm\in\Lambda(\ox,\ov).
\end{equation}
Using this together with (H4) tells us by Theorem~\ref{chain} that $\O$ is parabolically derivable at $\ox$ for every vector $w\in K_\O(\ox,\ov)$. This yields $T^2_\O(\bar x,w)\ne\emp$, and thus $\inf\{-\la\ov,u\ra|\;u\in T^2_\O(\bar x,w)\}<\infty$, which completes the proof of the proposition.
\end{proof}

The upper and lower estimates of the second subderivative obtained in Proposition~\ref{ule} indicate that the precise calculation of $\d^2\dd_\O(\ox,\ov)$ requires deriving an efficient condition under which those lower and upper estimates agree. In order to find such a condition, consider the following {\em linear-convex optimization problem}
\begin{equation}\label{pr}
\min_{u\in\R^n}-\langle\bar v,u\rangle\;\mbox{ subject to }\;\nabla f(\bar x)u+\nabla^2f(\bar x)(w,w)\in T^2_\Th\big(f(\bar x),\nabla f(\ox)w\big)
\end{equation}
for $(\bar x,\bar v)\in\gph N_\O$, where $w\in\R^n$ satisfies the inclusion $\nabla f(\ox)w\in T_\Th(f(\ox))$. First we construct the {\em dual problem} for \eqref{pr} given in the next proposition.

\begin{Proposition}[\bf dual second-order programs]\label{2dual} Under the validity of {\rm(H1)--(H4)}, fix any $w\in K_\O(\ox,\ov)$. Then the dual problem of \eqref{pr} is represented in the form
\begin{equation}\label{du2}
\max_{\lambda\in\R^{m}}\;\langle\lm,\nabla^2f(\bar x)(w,w)\rangle-\sigma_{T^2_\Th(f(\bar x),\nabla f(\ox)w)}(\lm)\;\mbox{ subject to }\;\lm\in\Lambda(\ox,\ov),
\end{equation}
where the set of Lagrange multipliers $\Lambda(\ox,\ov)$ is taken from \eqref{lagn}.
\end{Proposition}
\begin{proof} Observe first that \eqref{pr} is indeed a problem of convex programming since the convexity of $\Th$ and the parabolic derivability in (H4) ensure that the constraint set $T^2_\Th(f(\bar x),\nabla f(\ox)w)$ in \eqref{pr} is convex. Further, problem \eqref{pr} can be written as the unconstrained form
\begin{equation}\label{pir2}
\min_{u\in\R^n}\;-\langle\bar v,u\rangle+\dd_{T^2_\Th(f(\bar x),\nabla f(\ox)w)}\big(\nabla f(\bar x)u+\nabla^2f(\bar x)(w,w)\big).
\end{equation}
Picking $\olm\in\Lambda(\ox,\ov)$ and using \eqref{krie}, for any $w\in K_\O(\ox,\ov)$ we have $\nabla f(\ox)w\in K_\Th(f(\ox),\olm)$. By assumption (H4) on the parabolic derivability of $\Th$ at $f(\ox)$ for $\nabla f(\ox)w$, the indicator function $\dd_{T^2_\Th(f(\bar x),\nabla f(\ox)w)}$ is a proper, l.s.c., and convex function; see Theorem~\ref{pri}(i). Furthermore, the result of \cite[Proposition~13.12]{rw} infers the inclusion
\begin{equation*}
T^2_\Th\big(f(\bar x),\nabla f(\ox)w\big)+T_\Th\big(f(\ox)\big)\subset T^2_\Th\big(f(\bar x),\nabla f(\ox)w\big).
\end{equation*}
The opposite inclusion follows immediately from $0\in T_\Th(f(\ox))$, and hence we arrive at
\begin{equation*}
T^2_\Th\big(f(\bar x),\nabla f(\ox)w\big)+T_\Th\big(f(\ox)\big)=T^2_\Th\big(f(\bar x),\nabla f(\ox)w\big).
\end{equation*}
Taking this into account and employing \cite[Example~11.41]{rw} give us the dual problem of \eqref{pir2} as
\begin{equation}\label{du}
\max_{\lambda\in\R^{m}}\;\langle\lm,\nabla^2f(\bar x)(w,w)\rangle-\sigma_{\Gamma}(\lm)\;\mbox{ subject to }\;\nabla f(\bar x)^*\lambda=\ov,
\end{equation}
where $\Gamma:=T^2_\Th(f(\bar x),\nabla f(\ox)w)+T_\Th(f(\ox))$. It follows from the parabolic derivability of $\Th$ at $f(\ox)$ for $\nabla f(\ox)w$ that  $T^2_\Th(f(\bar x),\nabla f(\ox)w)\ne\emp$, and thus we derive from \cite[Corollary~11.24]{rw} that
\begin{equation*}
\sigma_\Gamma=\sigma_{T^2_\Th(f(\bar x),\nabla f(\ox)w)}+\sigma_{T_\Th(f(\ox))}.
\end{equation*}
Since $\sigma_{T_\Th(f(\ox))}=\dd_{N_\Th(f(\ox))}$, the dual problem \eqref{du} can be equivalently written as \eqref{du2}.
\end{proof}

Comparing the dual problem \eqref{du2} and the right-hand side of \eqref{dine2} indicates that they {\em agree} if we assume further that $\Th$ is {\em parabolically regular} at $f(\ox)$ for every $\lambda\in\Lambda(\ox,\ov)$. By Theorem~\ref{pri}(iii), the
later condition amounts to
\begin{equation}\label{dine3}
\d^2\dd_\Th\big(f(\bar x),\lm\big)\big(\nabla f(\ox)w\big)=-\sigma_{T^2_\Th(f(\bar x),\nabla f(\ox)w)}(\lm)\quad \mbox{for all}\;\; \lambda\in\Lambda(\ox,\ov).
\end{equation}
This tells us that the lower and upper estimates of the second subderivative $\d^2\dd_\O(\ox,\ov)$ obtained in Proposition~\ref{ule} coincide if the optimal values of the primal and dual problems \eqref{pr} and \eqref{du2}, respectively, are the same. Let us address this issue by considering the {\em optimal value function} $\vt\colon\R^m\to\oR$ of the canonically perturbed problem \eqref{pr}, defined by
\begin{equation}\label{pl06}
\vt(p)=\inf\big\{-\la\bar v,u\ra\big|\;\nabla f(\bar x)u+\nabla^2f(\bar x)(w,w)+p\in T^2_\Th\big(f(\bar x),\nabla f(\ox)w\big)\big\}.
\end{equation}
Denote by $\Lambda(\ox,\ov,w)$ the set of optimal solutions to the dual problem \eqref{du2}.

\begin{Proposition}[\bf duality relationships]\label{dua} In the setting of Proposition~{\rm\ref{2dual}} we have:

{\bf(i)} $\vt(0)\in\R$, which means that the optimal value of the primal problem  \eqref{pr} is finite.

{\bf(ii)} $\Lambda(\ox,\ov,w)\cap(\kappa\|\ov\|\B)\ne\emp$, where $\kappa$ is taken from \eqref{mscq}.

{\bf(iii)} There is no duality gap between the optimal values of the primal and dual problems \eqref{pr} and \eqref{du2}, respectively. Moreover, it holds that
\begin{equation}\label{dua1}
\vt(0)=\max_{\lm\in\Lambda(\ox,\ov)\cap(\kappa\,\|\ov\|\B)}\big\{\langle\lm,\nabla^2f(\bar x)(w,w)\rangle-\sigma_{T^2_\Th(F(\bar x),\nabla f(\ox)w)}(\lm)\big\}.
\end{equation}
\end{Proposition}
\begin{proof} We have already shown that the dual problem of \eqref{pr} is \eqref{du2}. To verify (i), pick $w\in K_\O(\bar x,\ov)$ and deduce from  Proposition~\ref{ule}(ii) that
\begin{equation}\label{pr9}
\vt(0)=\inf\big\{-\la\bar v,u\ra\big|\;u \in T^2_\O(\bar x,w)\big\}=-\sigma_{T^2_\O(\ox,w)}(\ov),
\end{equation}
which implies that $\vt(0)$ is a finite number. Thus we are done with (i).

To verify (ii), observe that the feasible region for problem \eqref{pl06} is exactly the set $S_w(p)$ from \eqref{mapt}. Fix $p\in\R^m$ and $u\in S_w(p)$. Then using the outer Lipschitzian property for $S_w$ established in Theorem~\ref{fors}, we arrive at the estimate
\begin{equation*}
\vt(p)\ge\vt(0)-\kappa\|\ov\|\cdot\|p\|\;\mbox{ for all }\;p\in\R^m.
\end{equation*}
This along with $\vt(0)\in\R$ tells us that the optimal value function $\vt(\cdot)$ is calm at $\op=0$ from below with constant $\kappa\|\ov\|$. Moreover, we get from \cite[Proposition~2.22]{rw} that the value function $\vt(\cdot)$ is convex. Appealing now to Proposition~\ref{nemp} gives us a vector $\lm\in\sub\vt(0)$ with $\|\lm\|\le\kappa\|\ov\|$. On the other hand, it follows from \cite[Theorem~2.142(i)]{bs} that $\Lambda(\ox,\ov,w)=\sub\vt(0)$. Combining all of these implies that there is a vector $\lm\in\Lambda(\ox,\ov,w)$ such that $\|\lm\|\le\kappa\|\ov\|$, which justifies (ii).

Finally, due to $\sub\vt(0)\ne\emp$ it follows from \cite[Theorem~2.142(i)]{bs} that the optimal values of the primal and dual problems \eqref{pr} and \eqref{du2}, respectively, are equal to each other, i.e.,
\begin{equation*}
\vt(0)=\max_{\lm\in\Lambda(\ox,\ov)}\big\{\langle\lm,\nabla^2f(\bar x)(w,w)\rangle-\sigma_{T^2_\Th\big(f(\bar x),\nabla f(\ox)w\big)}(\lm)\big\}.
\end{equation*}
Combining this with (ii) justifies \eqref{dua1} and thus completes the proof of the proposition.
\end{proof}

Having in hand the above duality, we are now in a position to establish parabolic regularity of constraint systems and obtain a precise formula for computing their second subderivatives.

\begin{Theorem}[\bf second subderivatives of parabolically regular constraint systems]\label{epi} Suppose in addition to {\rm(H1)--(H4)} with $w\in K_\O(\ox,\ov)$ that the set $\Th$ in \eqref{CS} is parabolically regular at $f(\ox)$ for every $\lambda\in\Lambda(\ox,\ov)$. Then the constraint system $\O$ is parabolically regular at $\ox$ for $\ov$, and for any $w\in\R^n$ the second subderivative of $\dd_\O$ is calculated by
\begin{eqnarray}\label{epi2}
\d^2\dd_\O(\ox,\ov)(w)&=&\max_{\lm\in\Lambda(\ox,\ov)}\big\{\langle\lm,\nabla^2f(\bar x)(w,w)\rangle+\d^2\dd_\Th\big(f(\bar x),\lm\big)\big(\nabla f(\ox)w\big)\big\}\\
&=&\max_{\lm\in\Lambda(\ox,\ov)\,\cap(\kappa\,\|\ov\|\B)}\big\{\langle\lm,\nabla^2f(\bar x)(w,w)\rangle+\d^2\dd_\Th\big(f(\bar x),\lm\big)\big(\nabla f(\ox)w\big)\big\}.\nonumber
\end{eqnarray}
\end{Theorem}
\begin{proof} It follows from Proposition~\ref{norm} that $N^p_\O(\ox)=N_\O(\ox)$. This opens the door for using Theorem~\ref{pri}(iii) to justify the parabolic regularity of $\dd_\O$ at $\ox$ for $\ov$. To proceed, pick any $w\in K_\O(\ox,\ov)$. Then employing \eqref{dine2}, Proposition~\ref{dua}(ii), and the parabolic regularity of $\Th$ at $f(\ox)$ for every $\lambda\in\Lambda(\ox,\ov)$ brings us to the inequality
\begin{eqnarray*}
\max_{\lm\in\Lambda(\ox,\ov)}\big\{\langle\lm,\nabla^2f(\bar x)(w,w)\rangle+\d^2\dd_\Th\big(f(\bar x),\lm\big)\big(\nabla f(\ox)w\big)\big\}\le\d^2\dd_\O(\ox,\ov)(w).
\end{eqnarray*}
 Moreover, we deduce from Proposition~\ref{ule}(ii) that
\begin{eqnarray*}
\d^2\dd_\O(\ox,\ov)(w)&\le&-\sigma_{T^2_\O(\ox,w)}(\ov)=\inf\big\{-\la\bar v,u\ra\big|\;u \in T^2_\O(\bar x,w)\big\}=\vt(0)\\
&\overset{\ss\mbox{by}\;\eqref{dua1}}{=}&\max_{\lm\in\Lambda(\ox,\ov)\,\cap(\kappa\,\|\ov\|\B)}\;\;\big\{\langle\lm,\nabla^2f(\bar x)(w,w)\rangle-\sigma_{T^2_\Th(Ff(\bar x),\nabla f(\ox)w)}(\lm)\big\}\\
&\overset{\ss\mbox{by}\;\eqref{dine3}}{=}&\max_{\lm\in\Lambda(\ox,\ov)\,\cap(\kappa\,\|\ov\|\B)}\;\;\big\{\langle\lm,\nabla^2f(\bar x)(w,w)\rangle+\d^2\dd_\Th\big(f(\bar x),\lm\big)\big(\nabla f(\ox)w\big)\big\}.
\end{eqnarray*}
Combining the above relationships justifies the claimed formula for $\d^2\dd_\O(\ox,\ov)(w)$ whenever $w\in K_\O(\ox ,\ov)$. Moreover, it shows that
for all $w\in K_\O(\ox,\ov)$ we have
\begin{equation*}
\d^2\dd_\O(\ox,\ov)(w)=-\sigma_{T^2_\O(\ox,w)}(\ov).
\end{equation*}
Appealing now to Theorem~\ref{pri}(iii) verifies that $\O$ is parabolically regular at $\ox$ for $\ov$.

It remains to justify the claimed formula for $\d^2\dd_\O(\ox,\ov)(w)$ when $w\notin K_\O(\ox,\ov)$. Remember that by \eqref{critco} we have the equality $\dom\d^2\dd_\O(\ox,\ov)=K_\O(\ox,\ov)$. Since $\d^2\dd_\O(\ox,\ov)$ is a proper function due to Proposition~\ref{ule}(i), it follows that $\d^2\dd_\O(\ox,\ov)(w)=\infty$ for all $w\notin K_\O(\ox,\ov)$. On the other hand, by \eqref{krie} the inclusion $w\notin K_\O(\ox,\ov)$ is equivalent to  $\nabla f(\ox)w\notin K_\Th(f(\ox),\lambda)$ for all $\lambda\in\Lambda(\ox,\ov)$. Assumption (H4) postulates that whenever $\lambda\in\Lambda(\ox,\ov)$ the set $\Th$ is parabolically derivable at $f(\ox)$ for all vectors $\nabla f(\ox)w\in  K_\Th(f(\ox),\lm)$. This along with Theorem~\ref{pri}(i) implies that for any $\lambda\in\Lambda(\ox,\ov)$ we have the condition
\begin{equation*}
\d^2\dd_\Th\big(g(\bar x),\lm\big)\big(\nabla f(\ox)w\big)=\infty\;\mbox{ if }\;\nabla f(\ox)w\notin K_\Th\big(f(\ox),\lambda\big).
\end{equation*}
Since both sets $\Lambda(\ox,\ov)$ and $\Lambda(\ox,\ov)\,\cap(\kappa\,\|\ov\|\B)$ are nonempty due to assumption (H3), the later equality means that for any $w\notin K_\O(\ox,\ov)$ both sides in \eqref{epi2} are equal to $\infty$, and so the claimed formula holds in this case as well. Thus we complete the proof of the theorem.
\end{proof}

The proof of Theorem~\ref{epi} suggests useful complements to the second subderivative formula.

\begin{Remark}[\bf variations of the second subderivative formula]\label{rem5}{\rm The following assertions hold under the assumptions of Theorem~\ref{epi}:

{\bf(i)} It follows from the proof of Theorem~\ref{epi} that for any number $r\in\R$ with $r\ge\kappa\,\|\ov\|$ and any $w\in\R^n$ the second subderivative of $\dd_\O$ can be expressed as
\begin{equation}\label{epi4}
\d^2\dd_\O(\ox,\ov)(w)=\max_{\lm\in\Lambda(\ox,\ov)\,\cap r\B}\big\{\langle\lm,\nabla^2f(\bar x)(w,w)\rangle+\d^2\dd_\Th\big(f(\bar x),\lm\big)\big(\nabla f(\ox)w\big)\big\}.
\end{equation}

{\bf(ii)} Proposition~\ref{dua}(ii) tells us that $\Lambda(\ox,\ov,w)\ne\emp$ for all $w\in K_\O(\ox,\ov)$, where $\Lambda(\ox,\ov,w)$ stands for the set of optimal solutions to the dual problem \eqref{du2}. On the other hand, we know from Proposition~\ref{ssp}(iv) that for any $w\in K_\O(\ox,\ov)$ the function
\begin{equation*}
\lm\mapsto\langle\lm,\nabla^2f(\bar x)(w,w)\rangle+\d^2\dd_\Th\big(f(\bar x),\lm\big)\big(\nabla f(\ox)w\big)
\end{equation*}
is concave. This implies that the set of optimal solutions to the problem stated on the right-hand side in \eqref{epi4} is exactly $\Lambda(\ox,\ov,w)\cap r\B$.}
\end{Remark}

Let us now present an example, which provides a direct application of Theorem~\ref{epi} to establishing the parabolic regularity of the {\em second-order cone} in conic programming and deriving a precise formula for the computation of the second subderivative of its indicator function. It is worth mentioning that parabolic regularity of the latter cone can be justified by using the known fact that the second-order cone is second-order regular in the sense of \cite{bs} and then applying \cite[Proposition~3.103]{bs}. Now we [provide a direct proof of this fact by employing the chain rule for parabolic regularity established in Theorem~\ref{epi}.

\begin{Example}[\bf second-order cone in conic programming]\label{ice}{\rm Consider the following remarkable nonpolyhedral cone, which plays an important role in conic programming and its applications;
see, e.g., \cite{hms} and the references therein. It is known under the names of {\em second-order/Lorentz/ice-cream cone} and is defined by
\begin{equation}\label{Qm}
{\cal Q}:=\big\{x=(y,x_n)\in\R^{n-1}\times\R\big|\;\|y\|\le x_n\big\}.
\end{equation}
If $\ox\in\inte\Q$, then $T^2_\Q(\ox,w)=\R^n$ whenever $w\in T_\Q(\ox)=\R^n$, and so $\Q$ is parabolically derivable at $\ox$ for every $w\in T_\Q(\ox)$. Since $N_\Q(\ox)=\{0\}$, we have  $K_\Q(\ox,0)=\R^n$. Theorem~\ref{pri}(i) tells us that $\dom\d^2\dd_\Q(\ox,0)=K_\Q(\ox,0)$. It is easy to see that $\Q$ satisfies Definition~\ref{par-reg} in this case, and thus it is parabolically regular at $\ox$ for $\ov=0$ with $\d^2\dd_\Q(\ox,0)=\dd_{K_\Q(\ox,0)}$.

If $\ox=0\in\Q$, then we get $T^2_\Q(\ox,w)=T_\Q(w)$ whenever $w\in T_\Q(\ox)=\Q$, and hence $\Q$ is parabolically derivable at $\ox$ for every $w\in T_\Q(\ox)$. Pick any $\ov\in N_\Q(\ox)=-\Q$ and $w\in\R^n$ with $\d^2\dd_\Q(\ox,\ov)(w)<\infty$. Theorem~\ref{pri}(i) yields $w\in K_\Q(\ox,\ov)$, and so $w\in T_\Q(\ox)=\Q$. Therefore
\begin{equation*}
0\le\d^2\dd_\Q(\ox,\ov)(w)\le\lim_{t\dn 0}\Delta_{t}^2\dd_\Q(\ox,\ov)(w)=0,
\end{equation*}
which tells us by Definition~\ref{par-reg} that $\Q$ is parabolically regular at $\ox$ for every $\ov\in N_\Q(\ox)$.

Consider the remaining most challenging case where $\ox\in(\bd\Q)\setminus\{0\}$. Observe that in this case the cone $\Q$ can be equivalently described as the constraint system \eqref{CS} by
\begin{equation*}
\Q=\big\{x=(y,x_n)\in\mathbb{R}^{n-1}\times\R\big|\;f(x):=\big(\|y\|^2-x_n^2,-x_n\big)\in\mathbb{R}_-^2\big\}.
\end{equation*}
Since $\R^2_-$ is a polyhedral set and $\nabla f(\ox)$ has full rank due to $\ox\in(\bd\Q)\setminus\{0\}$, we deduce from Example~\ref{poly} and Theorem~\ref{chain} that $\Q$ is parabolically derivable at $\ox$ for every $w\in T_\Q(\ox)$. Furthermore, Theorem~\ref{epi} ensures that $\Q$ is parabolically regular at $\ox$ for every $\ov\in N_\Q(\ox)$. To obtain finally a formula for the second subderivative of $\dd_\Q$ in this case, observe that since $f(\ox)=(0,-\ox_n)\in\R^2$ with $\ox_n>0$, the set of Lagrange multiplier \eqref{lagn} can be expressed by
\begin{equation*}
\Lambda(\ox,\ov)=\big\{\lambda\in N_{\R^2_-}\big(f(\ox)\big)\big|\;\nabla f(\ox)^*\lambda=\ov\big\}=\big\{\olm\big\}\;\mbox{ with }\;\olm=\Big(\dfrac{\|\ov\|}{2\|\ox\|},0\Big)\in\R^2.
\end{equation*}
Appealing now to \eqref{epi2} gives us the precise computation of the second subderivative of $\dd_\Q$:
\begin{eqnarray*}
\d^2\dd_\Q(\bar x,\ov)(w)&=&\langle\olm,\nabla^2 f(\ox)(w,w)\rangle+\d^2\dd_{\R^2_-}\big(f(\ox),\olm\big)\big(\nabla f(\ox)w\big)\\
&\overset{\ss\mbox{by}\;\eqref{sdpo}}{=}&\dfrac{\|\ov\|}{\|\ox\|}\Big(-w_n^2+\|u\|^2\Big)+\dd_{K_{\R^2_-}(f(\ox),\olm)}\big(\nabla f(\ox)w\big)\\
&=&\dfrac{\|\ov\|}{\|\ox\|}\Big(-w_n^2+\|u\|^2\Big)+\dd_{K_\Q(\ox,\ov)}(w)\;\mbox{ for any }\;\;w=(u,w_n)\in\mathbb{R}^{n-1}\times\R,
\end{eqnarray*}
which completes our second-order analysis of the Lorentz cone \eqref{Qm}.}
\end{Example}

Next we proceed with several important consequences of Theorem~\ref{epi}. The first one and the subsequent discussions address the duality issues for which Theorem~\ref{epi} offers pieces of new information in comparison with the above Propositions~\ref{2dual} and \ref{dua}.

\begin{Corollary}[\bf existence of primal optimal solutions]\label{osp} Let the basic assumptions {\rm(H1)--(H4)} hold, and let the set $\Th$ from \eqref{CS} be parabolically regular at $f(\ox)$ for every $\lambda\in\Lambda(\ox,\ov)$. Then whenever $w\in K_\O(\ox,\ov)$ the primal problem \eqref{pr} admits an optimal solution.
\end{Corollary}
\begin{proof} Theorem~\ref{epi} ensures that the indicator function $\dd_\O$ is parabolically regular at $\ox$ for $\ov$. The claimed existence of optimal solutions to \eqref{pr} follows from the combination  of Theorem~\ref{pri}(ii), the value function formula \eqref{pr9}, and the second-order tangent chain rule \eqref{chr2}.
\end{proof}

Let us compare the obtained duality results in Propositions~\ref{2dual} and \ref{dua} complemented by Corollary~\ref{osp} with those known before.

\begin{Remark}[\bf discussions on duality]{\rm The primal and dual problems \eqref{pr} and \eqref{dua1}, respectively, were considered in some different while equivalent form in \cite{bcs,bcs2,bs} under the metric regularity/Robinson constraint qualification. Weakening the latter to MSCQ \eqref{mscq} took a while in order to come to complete fruition. The first duality result under MSCQ appeared in \cite[Theorem~4.7]{hms} for the case where $\Th=\Q$, the second-order cone \eqref{Qm}. Then it was extended in \cite[Proposition~3.2]{gm17} to any closed convex cone with the analysis taken place at its vertex. Here we establish the duality relationships in generality for any convex sets that is parabolic derivable at the point in question. To the best of our knowledge, all the aforementioned results do not justify that the primal problem \eqref{pr} admits an optimal solution as it is done in Corollary~\ref{osp} when in addition we assume that the set $\Th$ under consideration is parabolically regular. This can significantly simplify the proof of the main results in \cite{hms,gm17}.}
\end{Remark}

Now we are ready to establish the twice epi-differentiability of the indicator functions associated with parabolically regular constraint systems \eqref{CS}.

\begin{Corollary}[\bf twice epi-differentiability for constraint systems]\label{tedi} Let $\O$ be given in \eqref{CS} under the basic assumptions {\rm(H1)--(H4)}, and let $w\in K_\O(\ox,\ov)$. Suppose in addition that $\Th$ is parabolically regular at $f(\ox)$ for every $\lambda\in\Lambda(\ox,\ov)$. Then the indicator function $\dd_\O$ is properly twice epi-differentiable at $\ox$ for $\ov$.
\end{Corollary}
\begin{proof}
This is an immediate consequence of Theorem~\ref{pri92} and Theorem~\ref{epi}.
\end{proof}

We conclude this section with the following discussions on twice epi-differentiability.

\begin{Remark}[\bf discussions on twice epi-differentiability]\label{2epi-diss} {\rm Twice epi-differentiability for extended-real-valued functions was introduced by Rockafellar in \cite{r88}, where this property was justified for fully amenable functions. In particular, it is shown therein that the indicator function of a fully amenable set, i.e., a set admitting representation \eqref{CS} with $\Th$ being a polyhedral convex set under the metric regularity constraint qualification, is twice epi-differentiable.  We are not familiar with any result of this type for \eqref{CS} when the set $\Th$ is merely parabolically regular.  Corollary~\ref{tedi} can be viewed as a far-going extension of the aforementioned result from \cite{r88} for indicator functions of constraint systems under parabolic regularity and MSCQ.}
\end{Remark}

\section{Further Properties of Parabolically Regular Sets}\sce\label{sect05a}

In this section we continue the study of parabolically regular sets while particularly using the results for constraint systems taken from Section~\ref{sect05}. This allows us to obtain new sufficient conditions for parabolic regularity by establishing its relationships with some notions of different types well understood and applied in variational analysis and optimization. Furthermore, in this way we obtain an intersection rule for parabolically regular sets, which ensures the preservation of parabolic regularity under intersections of sets with deriving precise formulas for calculating second-order tangents and second subderivatives of set intersections.\vspace*{0.05in}

Let us now recall the notion of ${\cal C}^2$-cone reducible sets that plays an important role in constrained optimization, especially in its second-order aspects; see Bonnans and Shapiro \cite{bs}.

\begin{Definition}[\bf reducible sets]\label{defr} A closed set $\O\subset\R^m$ is said to be ${\cal C}^2$-{\sc cone reducible} at $\oy\in\O$ to a closed convex
cone $\Xi\subset\R^s$  if there exist a neighborhood ${\cal U}\subset\R^m$ of $\oy$ and a ${\cal C}^2$-smooth mapping $h\colon\R^m\to\R^s$ such that
\begin{equation}\label{red}
\O\cap{\cal U}=\big\{y\in{\cal U}\big|\;h(y)\in\Xi\big\},\quad h(\oy)=0,\;\mbox{ and }\;\nabla h(\oy)\;\mbox{ has full rank }\;s.
\end{equation}
If this holds for all $\oy\in\O$, then we say that $\O$ is ${\cal C}^2$-cone reducible.
\end{Definition}

It is well known that the   set reducibility encompasses polyhedral convex sets and also important classes of nonpolyhedral ones including the second-order cone generating problems of {\em second-order cone programming} (SOCPs), the cone of positive semidefinite symmetric matrices in problems of {\em semidefinite programming} (SDPs), etc.; see \cite{bs}. We show now that ${\cal C}^2$-cone reducible are always parabolically regular and their indicator functions are twice epi-differentiable.

\begin{Theorem}[\bf parabolic regularity of ${\cal C}^2$-cone reducible sets]\label{twi20} Let $\O\subset\R^m$ be a closed set that is ${\cal C}^2$-cone reducible at $\oy\in\O$ to a closed convex cone $\Xi\subset\R^s$, and let $(\oy,\lm)\in\gph N_\O$. Then $\O$ is parabolically derivable at $\oy$ for any vector $ w \in T_\O(\oy)$ and parabolically regular at $\oy$ for $\lm$. Consequently, its indicator function $\delta_\O$ is properly twice epi-differentiable at $\oy$ for $\lm$ with the second subderivative calculated by
\begin{equation}\label{tepired}
\d^2\dd_\O(\oy,\lm)(w)=\left\{\begin{array}{ll}
\langle\mu,\nabla^2h(\bar y)(w,w)\rangle&\mbox{if }\;w\in K_\O(\oy,\lm),\\
\infty&\mbox{otherwise},
\end{array}
\right.
\end{equation}
where $\mu\in\R^s$ is the unique solution to the system
\begin{equation}\label{mubar}
\lm=\nabla h(\oy)^*\mu,\quad\mu\in N_\Xi\big(h(\oy)\big).
\end{equation}
\end{Theorem}
\begin{proof} Since $h(\oy)=0$ and $\Xi$ is a cone, for any vector $u\in T_\Xi(h(\oy))$ we have $T_\Xi^2(h(\oy),u)=T_\Xi(u)$. This along with the convexity of
$\Xi$ tells us that $\Xi$ is parabolically derivable at $h(\oy)$ for any vector from $T_\Xi(h(\oy))$. Appealing now to Theorem~\ref{chain} ensures that the set $\O$ is parabolically derivable at $\oy$ for any vector $w\in T_\O(\oy)$.

Next let us prove that $\Xi$ is parabolically regular at $h(\oy)$ for the unique vector $\mu$ satisfying \eqref{mubar}. To do it, take any vector $w\in\R^s$ such that $\d^2\dd_\Xi(h(\oy),\mu)(w)<\infty$. We have by Theorem~\ref{pri}(i) that $w\in K_\Xi(h(\oy),\mu)=\Xi\cap\{\mu\}^\bot$. Denote $w_k:=w$ for all $k\in\N$ and take any sequence $t_k\dn 0$ as $k\to\infty$. This gives us the inequality
\begin{equation*}
0\le\d^2\dd_\Xi\big(h(\oy),\mu\big)(w)\le\lim_{k\to\infty}\Delta_{t_k}^2\dd_\Xi\big(h(\oy),\mu\big)(w_k)=0,
\end{equation*}
which shows that $\Delta_{t_k}^2\dd_\Xi(h(\oy),\mu)(w_k)\to\d^2\dd_\Xi(h(\oy),\mu)(w)$ as $k\to\infty$. Furthermore, we get
\begin{equation*}
\lim_{k\to\infty}\frac{\|w_k-w\|}{t_k}=0,
\end{equation*}
and therefore $\Xi$ is parabolically regular at $h(\oy)$ for $\mu$. It follows from Definition~\ref{red} of the ${\cal C}^2$-cone reducibility that the set $\O$  in question is represented as a constraint system in \eqref{CS}. Thus applying Theorem~\ref{epi} to this set ensures the parabolic regularity of $\O$ at $\oy$ for $\lm$. Using then the second-order tangent chain rule \eqref{chr2} from Theorem~\ref{chain} tells us that whenever $w\in T_\O(\oy)$ we have the equivalent representation
\begin{equation*}
u\in T^2_\O(\oy,w)\Longleftrightarrow\nabla h(\oy)u+\nabla^2h(\bar y)(w,w)\in T^2_\Xi\big(h(\oy),\nabla h(\oy)w\big)=T_\Xi\big(\nabla h(\oy)w\big).
\end{equation*}
This brings us in turn to the equalities
\begin{eqnarray*}
\sigma_{T^2_\O(\oy,w)}(\lm)&=&\sup\big\{\la\lm,u\ra\big|\;u\in T^2_\O(\oy,w)\big\}\\
&=&\sup\big\{\la\mu,\nabla h(\oy)u\ra\big|\;\nabla h(\oy)u+\nabla^2 h(\bar y)(w,w)\in T_\Xi\big(\nabla h(\oy)w\big)\big\}\\
&=&-\la\mu,\nabla^2h(\bar y)(w,w)\ra.
\end{eqnarray*}
Employing finally assertions (i) and (iii) of Theorem~\ref{pri} verifies the second subderivative formula \eqref{tepired} and thus completes the proof of the theorem.
\end{proof}

The obtained theorem encloses the class of ${\cal C}^2$-cone reducible sets into the collection of parabolically regular ones. As Example~\ref{ex-red} below shows, this inclusion is generally {\em strict}. Before proceeding with this example,
let us discuss a distinguished feature of ${\cal C}^2$-cone reducible sets that can shed more light on their differences with parabolically regular sets.

\begin{Remark}[\bf specification of reducible sets among parabolically regular ones]\label{spec-red}{\rm Let $\O\subset\R^m$ be convex, parabolically derivable at $\oy\in\O$ for every vector $w\in K_\O(\oy,\olm)$ with $\olm\in N_\O(\oy)$, and parabolically regular at $\oy$ for $\olm$. Then it follows from Theorem~\ref{pri}(ii) that for any $w\in K_\O(\oy,\olm)$ there exists a second-order tangent $u_w\in T^2_\O(\oy,\olm)$ such that $\d^2\dd_\O(\oy,\olm)(w)=-\la \olm,u_w\ra$. If we assume in addition that $\O$ is ${\cal C}^2$-cone reducible at $\oy$, then it follows from \eqref{tepired} and \eqref{mubar} that for any $w\in K_\O(\oy,\olm)$ we have
\begin{equation*}
\d^2\dd_\O(\oy,\olm)(w)=-\la\olm,u_w\ra\;\mbox{ with }\;  u_w:=- (\nabla h(\oy)\nabla h(\oy)^*\big)^{-1}\nabla^2h(\bar y)(w,w).
\end{equation*}
This indicates that if $\O$ is not just parabolically regular but ${\cal C}^2$-cone reducible at the reference point, then the second-order tangent $u_w$ that we find for any $w\in K_\O(\oy,\olm)$ is in fact a {\em quadratic} function of $w$. This seems to be a distinguished feature of ${\cal C}^2$-cone reducible sets in the class of all the parabolically regular ones.}
\end{Remark}

The following example constructs a closed and convex set in $\R^2$ that is parabolically regular while not ${\cal C}^2$-cone reducible at the origin.

\begin{Example}[\bf failure of ${\cal C}^2$-cone reducibility]\label{ex-red} {\rm Fix any $\al\in(1,2)$ and consider the set $\O:=\epi\ph\subset\R^2$, where the function $\ph\colon\R\to\R$ is defined by
\begin{equation*}
\ph(x):=\left\{\begin{matrix}
0&x\le 0,\\
x^{\alpha}&x\ge 0.
\end{matrix}\right.
\end{equation*}
It is easy to check that $\ph(0)=\ph'(0)=0$ while $\ph_{+}^{''}(0)=\infty$. Taking $\oy:=(0,0)\in\O$ and $\olm=(0,-1)$, we claim that the following hold:

{\bf(i)} $\O$ is a closed and convex set with
\begin{equation*}
T_\O(\oy)=\R\times\R_{+},\quad N_\O(\oy)=\R_+\olm,\quad K_\O(\oy,\olm)=\R\times\{0\}.
\end{equation*}

{\bf(ii)} $\O$ is parabolically derivable at $\oy$ for every vector $(w_1,0)\in K_\O(\oy,\olm)$ with $w_1\le0$, and for any vector in this form we have $T_\O^2(\oy,(w_1,0))=\R\times\R_{+}$.

{\bf(iii)} $\O$ is parabolically regular at $\ox$ for $\olm$.

{\bf(iv)} $T_\O^2(\oy,(1,0))=\emp$, and thus $\O$ is not parabolically derivable at $\oy$ for $(1,0)\in T_\O(\oy)$.\vspace*{0.05in}

The last statement demonstrates that $\O$ is not ${\cal C}^2$-cone reducible at $\oy$. Indeed, if the reducibility property is satisfied for $\O$ at $\oy$, then Theorem~\ref{twi20} implies that $\O$ must be parabolically derivable at $\oy$ for any tangent vector in $T_\O(\oy)$, which clearly contradicts (iv).

To verify our claims (i)--(iv), observe first that (i) follows directly from the definition of $\O$ and simple calculations.  To proceed with (ii), pick any  $(w_1,0)\in K_\O(\oy,\olm)$ with $w_1\le 0$ and $(u_1,u_2)\in\R\times\R_{+}$. Then it is not hard to check that for all $t>0$ sufficiently small we have
\begin{equation*}
\ph\big(tw_1+\sm t^2u_1\big)\le\sm t^2u_2+o(t^2)\;\mbox{ with }\;o(t^2):=\sm t^{2\alpha}|u_1|^{\alpha}.
\end{equation*}
This readily yields the inclusion
\begin{equation*}
\oy+t(w_1,0)+\sm t^2(u_1,u_2)+\big(0,o(t^2)\big)\in\epi\ph=\O
\end{equation*}
for all $t>0$ sufficiently small, which clearly implies that $(u_1,u_2)\in T_\O^2(\oy,(w_1,0))$.  On the other hand, it is easy to show that $T_\O^2(\oy,(w_1,0))\subset\R\times\R_+$, which proves (ii).

Turning to (iii), note that $\d^2\dd_\O(\bar y,\olm)(w)\ge 0$ for all $w\in\R^2$ due to Theorem~\ref{pri}(i) and convexity of $\O$. We intend to show that
\begin{equation}\label{domd2}
\dom\d^2\dd_\O(\bar y,\olm)=\big\{(w_1,0)\in K_\O(\oy,\olm)\big|\;w_1\le 0\big\}.
\end{equation}
To this end, pick $w=(w_1,0)\in K_\O(\oy,\olm)$ with $w_1\le 0$. It follows from Proposition~\ref{pard} that
\begin{equation*}
\d^2\dd_\O(\bar y,\olm)(w)\le-\sigma_{T^2_\O(\oy,w)}(\olm)=0,
\end{equation*}
which leads us to $\d^2\dd_\O(\bar y,\olm)(w)=-\sigma_{T^2_\O(\oy,w)}(\olm)=0$ and hence justifies $w\in\dom\d^2\dd_\O(\bar y,\olm)$.

Conversely, pick $w\in\dom\d^2\dd_\O(\bar y,\olm)$ and deduce from \eqref{domd} that $\dom\d^2\dd_\O(\bar y,\olm)\subset K_\O(\oy,\olm)$.
Fixing $w=(w_1,0)\in\R^2$ with $w_1>0$ and taking $\omega=(\omega_1,\omega_2)\to w$, we get
\begin{equation*}
\Delta_t^2\dd_\O(\oy,\olm)(\omega)=\frac{\dd_\O(\oy+t\omega)-\dd_\O(\oy)-t\langle\olm,\omega\rangle}{\sm t^2}
\ge\begin{cases}
\infty&\mbox{if}\;\;\oy+t\omega\notin\O,\\
\frac{2(\omega_1)^\al}{t^{2-\al}}&\mbox{if}\;\;\oy+t\omega\in\O.
\end{cases}
\end{equation*}
This together with $w_1>0$ and $\al\in (1,2)$ implies that $\d^2\dd_\O(\bar y,\olm)(w)=\infty$ and hence justifies \eqref{domd2}.
Now we pick $w=(w_1,0)\in\dom\d^2\dd_\O(\bar y,\olm)$ and get from \eqref{domd2} and the discussions above that $\d^2\dd_\O(\bar y,\olm)(w)=-\sigma_{T^2_\O(\oy,w)}(\olm)=0$. Combining this with \cite[Proposition~13.64]{rw} yields
\begin{equation*}
\d^2\dd_\O(\bar y,\olm)(w)=-\sigma_{T^2_\O(\oy,w)}(\olm)=\liminf_{\substack{t\dn 0,\,\omega\to w\\
[\omega-w]/t\,\,{\ss\mbox{bounded}}}}\Delta_t^2\dd_\O(\oy,\olm)(\omega),
\end{equation*}
which verifies therefore the parabolic regularity of $\O$ at $\oy$ for $\olm$.

It remains to verify (iv). Suppose on the contrary that there exists some second-order tangent $(u_1,u_2)\in T_\O^2(\oy,(1,0))$. This   gives us a sequence $t_k\dn 0$ with
\begin{equation*}
(0,0)+t_k(1,0)+\sm t_k^2(u_1,u_2)+o(t_k^2)\in\epi\ph,
\end{equation*}
which amounts to saying in turn that
\begin{equation*}
\frac{\ph\big(t_k+\sm t_k^2u_1+o(t_k^2)\big)-\ph(0)-t_k\ph'(0)}{\sm t_k^2}\le u_2+\frac{o(t_k^2)}{\sm t_k^2}.
\end{equation*}
Denote $s_k:=1+\sm t_k u_1+\frac{o(t_k^2)}{t_k}$ and get $s_k>0$ for all $k\in\N$ sufficiently large. This allows us to rewrite the above inequality in the equivalent form
\begin{equation*}
\frac{\ph(t_k s_k)-\ph(0)-(t_k s_k)\ph'(0)}{\sm(t_k s_k)^2}\le\frac{u_2}{s_k^2}+\frac{2}{s_k^2}\frac{o(t_k^2)}{t_k^2}.
\end{equation*}
Passing to the limit as $k\to\infty$ contradicts the fact that $\ph''_{+}(0)=\infty$ and thus completes the proof of (iv) and our consideration in this example.}
\end{Example}

We conclude this section by establishing the following intersection rules for parabolically regular sets and related second-order constructions.

\begin{Theorem}[\bf intersection rules for parabolically regular sets]\label{intrul} Let $\O_1$ and $\O_2$ be two closed and convex sets in $\R^n$, and let $\ox\in\O_1\cap\O_2$. Assume that there exist a constant $\kappa>0$ and a neighborhood $U$ of $\ox$ satisfying the metric qualification condition
\begin{equation}\label{mqc}
{\rm dist}(x;\O_1\cap\O_2)\le\kappa\big({\rm dist}(x;\O_1)+{\rm dist}(x;\O_2)\big)\;\mbox{ for all }\;x\in U.
\end{equation}
If both $\O_1$ and $\O_2$ are parabolically derivable for every vector $u\in T_{\O_1\cap\O_2}(\ox)$, then their intersection $\O_1\cap\O_2$ is also parabolically derivable at $\ox$ for every vector $u\in T_{\O_1\cap\O_2}(\ox)$, and we have the second-order tangent intersection rule
\begin{equation}\label{int1}
T_{\O_1\cap\O_2}^2(\ox,w)=T_{\O_1}^2(\ox,w)\cap T_{\O_2}^2(\ox,w)\;\mbox{ for all }\;w\in T_{\O_1}(\ox)\cap T_{\O_2}(\ox).
\end{equation}
Further, pick $\ov\in N_{\O_1\cap\O_2}(\ox)$ and define the set
\begin{equation*}
S(\ox,\ov):=\big\{(v_1,v_2)\in\R^{2n}\big|\:v_1+v_2=\ov,\:v_1\in N_{\O_1}(\ox),\;v_2\in N_{\O_2}(\ox)\big\}.
\end{equation*}
If for any pair $(v_1,v_2)\in S(\ox,\ov)$ the sets $\O_1$ and $\O_2$ are parabolically regular at $\ox$ for $v_1$ and $v_2$, respectively, then their intersection $\O_1\cap\O_2$ is parabolically regular at $\ox$ for $\ov$, and we have the second subderivative intersection rule
\begin{equation}\label{intformul}
\d^2\dd_{\O_1\cap\O_2}(\ox,\ov)(w)=\max_{(v_1,v_2)\in S(\ox,\ov)}\big\{\d^2\dd_{\O_1}(\ox,v_1)(w)+\d^2\dd_{\O_2}(\ox,v_2)(w)\big\}.
\end{equation}
\end{Theorem}
\begin{proof}
We know from Bauschke et al. \cite[Theorem~3]{bbl} that the metric qualification condition \eqref{mqc} ensures the (first-order) tangent and normal intersection rules for convex sets:
\begin{equation*}
T_{\O_1\cap\O_2}(\ox)=T_{\O_1}(\ox)\cap T_{\O_2}(\ox)\;\mbox{ and }\;N_{\O_1\cap\O_2}(\ox)=N_{\O_1}(\ox)+N_{\O_2}(\ox).
\end{equation*}
Define further the set $\O\subset\R^n$ by
\begin{equation*}
\O:=\O_1\cap\O_2:=\big\{x\in\R^n\big|\:(x,x)\in\O_1\times\O_2\big\}
\end{equation*}
and observe that $\O$ belongs to the class of {\em constraint systems} \eqref{CS} with $\Th:=\O_1\times\O_2$, ${\cal O}=\R^n$, and $f(x):=(x,x)$. Furthermore, it is not hard to check that \eqref{mqc} amounts to saying that the constraint mapping $x\mapsto f(x)-\Th$ is {\em metrically subregular} at $((\ox,\ox),0)$. It follows from the definitions that $T_{\O_1\times\O_2}(\ox,\ox)=T_{\O_1}(\ox)\times T_{\O_2}(\ox)$ and that
\begin{equation}\label{pj01}
T_{\O_1\times\O_2}^2\big((\ox,\ox),(w_1,w_2)\big)=T_{\O_1}^2(\ox,w_1)\times T_{\O_2}^2(\ox,w_2)\;\mbox{ for all }\;(w_1,w_2)\in T_{\O_1\times\O_2}(\ox,\ox).
\end{equation}
Pick now $w\in T_{\O_1\cap\O_2}(\ox)$ and note that $T^2_\Th(f(\ox),\nabla f(\ox)w)=T_{\O_1}^2(\ox,w)\times T_{\O_2}^2(\ox,w)$. Using this, the second-order tangent chain rule \eqref{chr2} from Theorem~\ref{chain}, and the obvious fact that $\nabla^2f(\ox)=0$ yields the representation
\begin{equation*}
u\in T_\Th^2(\ox,w)\iff(u,u)=\nabla f(\ox)u\in T_{\O_1}^2(\ox,w)\times T_{\O_2}^2(\ox,w),
\end{equation*}
which clearly justifies the claimed intersection rule \eqref{int1} for the second-order tangent sets. Since $\O_1$ and $\O_2$ are parabolic derivable at $\ox$ for the selected vector $w$, so is $\Th$ at $f(\ox)=(\ox,\ox)$ for $\nabla f(\ox)w=(w,w)$. Employing again Theorem~\ref{chain} tells us that $\O$ is parabolically regular at $\ox$ for $w$, which therefore proves the first part of the theorem.

We turn next to verifying that the set intersection $\O_1\cap\O_2$ is parabolically regular at $\ox$ for $\ov$. To this end, observe that the collection of Lagrange multipliers $\Lambda(\ox,\ov)$ from \eqref{lagn} in the setting under consideration can be equivalently expressed as
\begin{eqnarray*}
\Lambda(\ox,\ov)&=&\big\{(v_1,v_2)\big|\:\nabla f(\ox)^*(v_1,v_2)=\ov,\:(v_1,v_2)\in N_{\O_1\times\O_2}(\ox)\big\}\\
&=&\big\{(v_1,v_2)\big|\:v_1+v_2=\ov,\:v_1\in N_{\O_1}(\ox),\;v_2\in N_{\O_2}(\ox)\big\}=S(\ox,\ov).
\end{eqnarray*}
Pick $(v_1,v_2)\in S(\ox,\ov)$ and deduce from \eqref{pj01} and the support function definition that
\begin{eqnarray*}
\sigma_{\ss T_{\Th}^2((\ox,\ox),(w_1,w_2))}(v_1,v_2)=\sigma_{\ss T_{\O_1}^2(\ox,w_1)}(v_1)+\sigma_{\ss T_{\O_2}^2(\ox,w_2)}(v_2)\;\mbox{ for all }\;(w_1,w_2)\in T_{\O_1\times\O_2}(\ox,\ox).
\end{eqnarray*}
To prove that $\Th=\O_1\times\O_2$ is parabolically regular at $(\ox,\ox)$ for $(v_1,v_2)$, pick $(w_1,w_2)\in K_{\Th}((\ox,\ox),(v_1,v_2))$ and observe that $K_{\Th}((\ox,\ox), (v_1,v_2))=K_{\O_1}(\ox,v_1)\times K_{\O_2}(\ox,v_2)$. Since the sets $\O_i$ are parabolically regular at $\ox$ for $v_i$ as $i=1,2$, we deduce from Theorem~\ref{pri}(iii) that
\begin{equation*}
\d^2\dd_{\O_1}(\ox,v_1)(w_1)=-\sigma_{\ss T_{\O_1}^2(\ox,w_1)}(v_1)\;\mbox{ and }\;\d^2\dd_{\O_2}(\ox,v_2)(w_2)=-\sigma_{\ss T_{\O_2}^2(\ox,w_2)}(v_2).
\end{equation*}
It follows directly from definition \eqref{ssd} of the second subderivative that
\begin{equation*}
\d^2\dd_{\O_1}(\ox,v_1)(w_1)+\d^2\dd_{\O_2}(\ox,v_2)(w_2)\le\d^2\dd_{\Th}\big((\ox,\ox),(v_1,v_2)\big)(w_1,w_2).
\end{equation*}
On the other hand, Proposition~\ref{pard} leads us to the relationships
\begin{eqnarray*}
\d^2\dd_{\Th}\big((\ox,\ox),(v_1,v_2)\big)(w_1,w_2)&\le&-\sigma_{\ss T_{\Th}^2((\ox,\ox),(w_1,w_2))}(v_1,v_2)=-\big(\sigma_{\ss T_{\O_1}^2(\ox,w_1)}(v_1)+\sigma_{\ss T_{\O_2}^2(\ox,w_2)}(v_2)\big).
\end{eqnarray*}
Combining all of this, we arrive at the equality
\begin{equation*}
\d^2\dd_{\Th}\big((\ox,\ox),(v_1,v_2)\big)(w_1,w_2)=-\sigma_{\ss T_{\Th}^2((\ox,\ox),(w_1,w_2))}(v_1,v_2)\;\mbox{ for all }\;(w_1,w_2)\in K_{\Th}\big((\ox,\ox),(v_1,v_2)\big)
\end{equation*}
and thus conclude from Theorem~\ref{pri}(iii) that $\O=\O_1\cap\O_2$ is parabolically regular at $\ox$ for $\ov$. Applying Theorem \ref{epi} to the sets $\O$ written as a constraint system \eqref{CS} with $\Th=\O\times\O_2$ and taking into account that $\Th$ is parabolically regular at $(\ox,\ox)$ for any pair $(v_1,v_2)\in S(\ox,\ov)$ verify that $\O$ is parabolic regular at $\ox$ for $\ov$. Finally, the intersection rule \eqref{intformul} for the second subderivative of $\dd_{\O_1\cap\O_2}$ is an adaptation of \eqref{epi2} to the setting under consideration. This completes the proof of the theorem.
\end{proof}

Let us mention that somewhat related intersection results can be found in \cite[Theorem~3.90]{bs} for the second-order regular sets in the sense therein. Using the notation of Theorem~\ref{intrul}, the qualification condition utilized in \cite[Theorem~3.90]{bs} reads as $\mbox{int}\,\O_1\cap\O_2\ne\emp$, which is the standard qualification condition in convex analysis. It has been well recognized that the metric qualification condition \eqref{mqc} is much weaker than the latter interiority one.

\section{Second-Order Optimality Conditions with Quadratic Growth}\sce \label{sect06}

This section addresses applications of the developed theory of parabolic regularity and twice epi-differentiability to deriving new second-order optimality conditions in problems of constrained optimization. The problem under consideration here is formulated as follows:
\begin{equation}\label{coop}
\min_{x\in\R^n}\ph(x)\;\mbox{ subject to }\;f(x)\in\Th,
\end{equation}
where $\ph\colon\R^n\to\R$, $f\colon\R^n\to\R^m$, and $\Th\subset\R^m$. Throughout this and next sections, we assume that $\ph$ and $f$ are twice differentiable at the reference points, and that $\Th$ is a closed convex set. The constrained problem \eqref{coop} can be rewritten in the unconstrained optimization format
\begin{equation}\label{coop2}
\min_{x\in\R^n}\ph(x)+\dd_\O(x)\;\mbox{ with }\;\O:=\big\{x\in\R^n\big|\;f(x)\in\Th\big\}.
\end{equation}
The set $\O$ defined in \eqref{coop2} is a constraint system in the form of \eqref{CS} with ${\cal O}=\R^m$ therein. The {\em Lagrangian} function associated with \eqref{coop} is defined in the conventional way as $L(x,\lm):=\ph(x)+\la\lm,f(x)\ra$ for any pair $(x,\lm)\in\R^n\times\R^m$.

The next theorem collects the main results of this section while providing {\em no-gap} second-order optimality conditions for the constrained problem \eqref{coop} with a parabolically regular set $\Th$. Recall that by no-gap conditions we understand a pair of optimality conditions where the sufficient condition differs from the corresponding necessary condition by replacing the nonstrict inequality in the latter with the strict one. In fact, the obtained second-order sufficient condition offer more; namely, a {\em quadratic growth} of the cost function that is strongly used below.

\begin{Theorem}[\bf no-gap second-order optimality conditions under parabolic regularity]\label{nsop1} Let $\ox$ be a feasible solution to problem \eqref{coop}, and let $\ov:=-\nabla\ph(\ox)$. In addition to the basic assumptions {\rm(H1)--(H4)} imposed on $\O$ from \eqref{coop2}, suppose that the set $\Th$ in \eqref{coop} is parabolically regular at $f(\ox)$ for every $\lambda\in\Lambda(\ox,\ov)$ from \eqref{lagn}. Then we have the following second-order optimality conditions for the constrained problem \eqref{coop}:

{\bf(i)} If $\ox$ is a local minimizer of \eqref{coop}, then the second-order necessary condition
\begin{equation}\label{nopc1}
\max_{\lm\in\Lambda(\ox,\ov)}\big\{\langle\nabla_{xx}^2L(\bar x,\lm)w,w\rangle+\d^2\dd_\Th\big(f(\bar x),\lm\big)\big(\nabla f(\ox)w\big)\big\}\ge 0
\end{equation}
is satisfied for all $w\in K_\O(\ox,\ov)$.

{\bf(ii)} The validity of the second-order sufficient condition
\begin{equation}\label{sscc}
\max_{\lm\in\Lambda(\ox,\ov)}\big\{\langle\nabla_{xx}^2L(\bar x,\lm)w,w\rangle+\d^2\dd_\Th\big(f(\bar x),\lm\big)\big(\nabla f(\ox)w\big)\big\}> 0\;\mbox{ when }\; w\in K_\O(\ox,\ov)\setminus\{0\}
\end{equation}
amounts to the existence of positive constants $\ell$ and $\ve$ such that the quadratic growth condition
\begin{equation}\label{gro}
\psi(x)\ge\psi(\ox)+\frac{\ell}{2}\|x-\ox\|^2\;\mbox{ for all }\;x\in\B_{\ve}(\ox)
\end{equation}
holds, where $\psi:=\ph+\dd_\Th\circ f$ is the cost function in \eqref{coop2}. Thus $\ox$ provides is a strict local minimum for the constrained optimization problem \eqref{coop}.
\end{Theorem}
\begin{proof} To verify (i), we get from the imposed assumptions in the theorem and Proposition~\ref{2chin} that $\ov=-\nabla\ph(\ox)\in N_\O(\ox)$ and then
\begin{equation*}
0\in\nabla\ph(\ox)+N_\O(\ox)=\sub(\ph+\dd_\O)(\ox)=\sub\psi(\ox).
\end{equation*}
Employing Corollary~\ref{tedi} tells us that the indicator function $\dd_\O$ is properly twice epi-differentiable at $\ox$ for $\ov$. Using this and the assumed twice differentiability of $\ph$ at $\ox$, it is easy to derive from the definitions the following second subderivative sum rule:
\begin{equation}\label{ns1}
\d^2(\ph+\dd_\O)(\ox,0)(w)=\langle\nabla^2\ph(\ox)w,w\rangle+\d^2\dd_\O(\ox,\ov)(w)\;\mbox{ for all }\;w\in\R^n.
\end{equation}
Since $\ox$ is a local minimizer of $\psi=\ph+\dd_\O$, it follows from \cite[Theorem~13.24(a)]{rw} that $\d^2\psi(\ox,0)(w)\ge 0$ for all $w\in\R^n$. Applying then the second subderivative calculation for $\dd_\O$ from \eqref{epi2} and the second subderivative sum rule \eqref{ns1} readily justifies assertion (i). Observe that due to \eqref{critco} we have $\dom\d^2\dd_\O(\ox,\ov)=K_\O(\ox,\ov)$ and thus do not need to consider vectors $w\notin K_\O(\ox,\ov)$ in the second-order necessary optimality condition \eqref{nopc1}.

To proceed next with the proof of (ii), deduce from the proof of \cite[Theorem~13.24(c)]{rw} in the general unconstrained framework of minimizing an arbitrary proper function $\psi\colon\R^n\to\oR$ that the second-order condition $\d^2\psi(\ox,0)(w)>0$ for all $w\in\R^n\setminus\{0\}$ amounts to the existence of positive constants $\ell$ and $\ve$ such that the quadratic growth condition \eqref{gro} is satisfied. Taking into account the particular form of our function $\psi$ and combining it with the second subderivative sum rule \eqref{ns1}, the representation $\dom\d^2\dd_\O(\ox,\ov)=K_\O(\ox,\ov)$, and the second subderivative calculation for $\dd_\O$ in \eqref{epi2} verify assertion (ii) and thus complete the proof of the theorem.
\end{proof}

Some commentaries on second-order optimality conditions are now in order.

\begin{Remark}[\bf discussions on second-order optimality conditions]\label{opt-disc} {\rm Observe the following:

{\bf(i)} The second-order sufficient condition \eqref{sscc} can be equivalently expressed via the existence of $\ell>0$ for which we have the estimate
\begin{equation*}
\max_{\lm\in\Lambda(\ox,\ov)}\big\{\langle\nabla_{xx}^2L(\bar x,\lm)w,w\rangle+\d^2\dd_\Th\big(f(\bar x,\lm\big)\big(\nabla f(\ox)w\big)\big\}\ge\ell\|w\|^2\;\mbox{ whenever }\;w\in K_\O(\ox,\ov)\setminus\{0\}.
\end{equation*}
This is due to the fact that the second subderivative is l.s.c.\ by Proposition~\ref{ssp}(i). Labeling as $\Hat\ell$ the best/largest constant $\ell$ satisfying the above condition, we can easily compute it by
\begin{eqnarray*}
\Hat\ell&=&\min_{w\in S}\d^2(\ph+\dd_\O)(\ox,0)(w)\\
&=&\min_{w\in[K_\O(\ox,\ov)\cap S]}\max_{\lm\in\Lambda(\ox,\ov)}\big\{\langle\nabla_{xx}^2L(\bar x,\lm)w,w\rangle+\d^2\dd_\Th\big(f(\bar x),\lm\big)\big(\nabla f(\ox)w\big)\big\},
\end{eqnarray*}
where $S:=\{w\in\R^n|\;\|w\|=1\}$ stands for the unit sphere in $\R^n$.

{\bf(ii)} Second-order optimality conditions in constrained optimization have been studied in the literature under certain second-order regularity assumptions and different constraint qualifications. Let us mention those obtained in \cite{bcs} for problems \eqref{coop} with second-order regular sets $\Th$ under the metric regularity/Robinson constraint qualification; see \cite[Chapter~3]{bs} for more details. Quite recently no-gap second-order optimality conditions were derived in \cite{chnt} for constraint problems of type \eqref{coop} generated by ${\cal C}^2$-cone reducible sets $\Th$ under the metric subregularity constraint qualification. The latter qualification condition was also used in our paper \cite{mms} for similar problems of composite optimization with fully subamenable constraint functions. All of the aforementioned results are strict consequences of Theorem~\ref{nsop1}. Furthermore, the approach developed here, which is mainly based on parabolic regularity and second subderivative calculus under MSCQ, is fundamentally different from those mentioned above. It allows us to not only establish the strongest no-gap second-order optimality conditions for a large class of problems in constrained optimization, but also to unify previously known developments in this direction.}
\end{Remark}

To conclude this section, we present yet another second-order sufficient optimality condition for problem \eqref{coop} that is of type \eqref{nopc1} but is obtained under different assumptions. Note that we do not impose now any constraint qualification while assuming instead the validity of a first-order necessary optimality condition in the Karush-Kuhn-Tucker (KKT) form. The obtained result is particularly useful for the study of augmented Lagrangians in the next section.

\begin{Proposition}[\bf second-order sufficient condition without constraint qualifications]\label{sos} Let $\ox$ be a feasible solution to problem \eqref{coop}. Suppose that the pair $(\ox,\olm)$ satisfies the KKT system
\begin{equation}\label{kkt}
\nabla_x L(\ox,\olm)=0,\quad\olm\in N_\Th\big(f(\ox)\big),
\end{equation}
and that the set $\Th$ is parabolically derivable at $f(\ox)$ for every critical cone vector $u\in K_\Th(f(\ox),\olm)$. Assume also that the second-order condition
\begin{equation}\label{2gro}
\la\nabla_{xx}^2L(\bar x,\olm)w,w\ra+\d^2\dd_\Th\big(f(\ox),\olm\big)\big(\nabla f(\ox)w\big)>0
\end{equation}
is satisfied for all $w\in\R^n\setminus\{0\}$ with $\nabla f(\ox)w\in K_\Th(f(\ox),\olm)$. Then there exist positive constants $\ve$ and $\ell$ such
that the quadratic growth condition \eqref{gro} holds while ensuring in particular that $\ox$ is a strict local minimizer for problem \eqref{coop}.
\end{Proposition}
\begin{proof} It follows from \eqref{kkt}, the structure of \eqref{coop2}, and the convexity of $\Th$ that
\begin{equation*}
0\in\nabla\ph(\ox)+\nabla f(\ox)^*\olm\subset\nabla\ph(\ox)+\Hat N_\O(\ox)\subset\nabla\ph(\ox)+N_\O(\ox).
\end{equation*}
Using similar arguments as those for the proof of \eqref{dine2} leads us to
\begin{equation*}
\d^2\dd_\O(\ox,\ov)(w)\ge\langle\olm,\nabla^2f(\bar x)(w,w)\rangle+\d^2\dd_\Th\big(f(\bar x),\olm\big)\big(\nabla f(\ox)w\big)
\end{equation*}
for all $w\in\R^n$, where $\ov=-\nabla\ph(\ox)$. Hence for any $w\in\R^n$ we have
\begin{equation*}
\d^2\psi(\ox,0)(w)=\la\nabla^2\ph(\ox)w,w\ra+\d^2\dd_\O(\ox,\ov)(w)\ge\la\nabla_{xx}^2L(\bar x,\olm)w,w\ra+\d^2\dd_\Th\big(f(\bar x),\olm\big)\big(\nabla f(\ox)w\big),
\end{equation*}
where $\psi=\ph+\dd_\O$. If $\nabla f(\ox)w\in K_\Th(f(\ox),\olm)$ for some $w\in\R^n\setminus\{0\}$, then it follows from the above inequality and the assumed second-order condition \eqref{2gro} that $\d^2\psi(\ox,0)(w)>0$. If $\nabla f(\ox)w\notin K_\Th(f(\ox),\olm)$ for some $w\ne 0$, we deduce from Theorem~\ref{pri}(i) that $\d^2\dd_\Th(f(\bar x),\olm)(\nabla f(\ox)w)=\infty$. Using again the above inequality yields $\d^2\psi(\ox,0)(w)=\infty$. Hence for any $w\in\R^n\setminus\{0\}$ we get $\d^2\psi(\ox,0)(w)>0$. Appealing finally to \cite[Theorem~13.24(c)]{rw} verifies the quadratic growth condition \eqref{gro} and thus completes the proof of the proposition.
\end{proof}

\section{Augmented Lagrangians under Parabolic Regularity}\sce\label{sect06a}

In this section we present one of the most striking novel applications of the developed second-order variational theory under parabolic regularity. This  concerns augmented Lagrangians associated with the class of constrained optimization problems \eqref{coop}. The importance of augmented Lagrangians has been well recognized from the viewpoints of both theoretical and algorithmic developments in variational analysis and optimization, and the quadratic growth condition achieved below under parabolic regularity has been a goal of many previous efforts in particular settings; see Remark~\ref{ag-diss} for more comments.

To reach our goal, we need to involve additional tools of second-order variational analysis complemented to those discussed above.  Recall that a function $\ph\colon\R^n\to\oR$ is {\em twice semidifferentiable} at $\ox$ if it is semidifferentiable at $\ox$, defined as in \eqref{semid}, and the limit
\begin{equation*}
\lim_{\substack{t\dn 0\\
u\to w}}\Delta^2_t\ph(\ox)(u)\;\mbox{ with }\;\Delta^2_t\ph(\ox)(u):=\frac{\ph(\ox+tu)-\ph(\ox)-t\d\ph(\ox)(u)}{\sm t^2}
\end{equation*}
exists. The {\em second semiderivative} of $\ph$ at $\ox$ is denoted by $\d^2\ph(\ox)$. It is not hard to check the the existence of the above limit amounts to
saying that $\ph$ satisfies the second-order expansion
\begin{equation*}
\ph(x)=\ph(\ox)+\d\ph(\ox)(x-\ox)+\sm\d^2\ph(\ox)(x-\ox)+o(\|x-\ox\|^2),
\end{equation*}
and that $\d^2\ph(\ox)(w)$ is finite everywhere while depending continuously on $w$. As discussed in \cite[p.\ 590]{rw}, although twice semidifferentiability seems appealing due to its tie to the second-order expansion, it has limitations to handle nonsmoothness. Indeed, even first-order semidifferentiability of $\ph$ at $\ox$ may not hold unless the function is finite and continuous around $\ox$. This makes it impossible to deal with the boundary points of function domains. Nevertheless, there are optimization settings where twice semidifferentiablity is achievable. As shown below, the augmented Lagrangian associated with \eqref{coop} under parabolic regularity enjoys this property.\vspace*{0.05in}

We begin with a simple sum rule for twice semidifferentiability.

\begin{Proposition}[\bf sum rule for twice semidifferentiability]\label{sumr} Let the functions $\ph_i\colon\R^n\to\oR$ as $i=1,2$ be twice semidifferentiable at $\ox$. Then their sum $\ph_1+\ph_2$ is twice semidifferentiable at $\ox$, and we have the equality
\begin{equation*}
\d^2\big(\ph_1+\ph_2\big)(\ox)=\d^2\ph_1(\ox)+\d^2\ph_2(\ox).
\end{equation*}
\end{Proposition}
\begin{proof} It follows directly from the twice semidifferentiability of $\ph_1$ and $\ph_2$ at $\ox$ that the sum $\ph_1+\ph_2$ is also twice semidifferentiable at this point with
\begin{equation*}
\d\big(\ph_1+\ph_2\big)(\ox)(u)=\d\ph_1(\ox)(u)+\d\ph_2(\ox)(u)\;\mbox{ whenever }\;u\in\R^n.
\end{equation*}
This immediately implies that for any $u\in\R^n$ we have
\begin{equation*}
\Delta^2_t\big(\ph_1+\ph_2\big)(\ox)(u)=\Delta^2_t\ph_1(\ox)(u)+\Delta^2_t\ph_2(\ox)(u).
\end{equation*}
Passing now to the limit as $u\to w$ verifies the twice semidifferentiability of $\ph_1+\ph_2$ at $\ox$.
\end{proof}

Next we establish a chain rule for twice semidifferentiability that is particularly useful for calculating the second subderivative of
the augmented Lagrangian associated with \eqref{coop}.

\begin{Proposition}[\bf chain rule for twice semidifferentiability]\label{tsdch} Consider the composition $\ph=\vt\circ f$, where $f\colon\R^n\to\R^m$ is  twice differentiable at $\ox$, and where $\vt\colon\R^m\to\oR$ is  differentiable at $f(\ox)$ and twice semidifferentiable at this point.
Then the following assertions hold:

{\bf(i)} $\ph$ is  twice semidifferentiable at $\ox$, and its second semiderivative is calculated by
\begin{equation*}
\d^2\ph(\ox)(w)=\big\la\nabla\vt\big(f(\ox)\big),\nabla^2f(\ox)(w,w)\big\ra+\d^2\vt\big(f(\ox)\big)\big(\nabla f(\ox)w\big)\;\mbox{ for all }\;w\in\R^n.
\end{equation*}

{\bf(ii)} $\ph$ has the second-order expansion
\begin{equation}\label{secgof1}
\ph(x)=\ph(\ox)+\la\nabla\ph(\ox),x-\ox\ra+\sm\d^2\ph(\ox)(x-\ox)+o(\|x-\ox\|^2).
\end{equation}

{\bf(iii)} $\ph$ is twice epi-differentiable at $\ox$ for $\nabla \ph(\ox)$ with
\begin{equation*}
\d^2\ph\big(\ox,\nabla\ph(\ox)\big)=\d^2\ph(\ox).
\end{equation*}
\end{Proposition}
\begin{proof} Since $\ph$ is differentiable at $\ox$, it is semidifferentiable at this point. Pick any $w\in\R^n$ and let $u\to w$ and $t\dn 0$. It follows from the twice differentiability of $f$ at $\ox$ that
\begin{equation*}
f(\ox+t u)=f(\ox)+t\nabla f(\ox)u+\sm t^2\nabla^2f(\ox)(u,u)+o(t^2).
\end{equation*}
Denote $y_t:=f(\ox+t u)$ and $\oy:=f(\ox)$. It is not hard to derive from the twice semidifferentiability of $\vt$ at $f(\ox)$ that the second-order expansion
\begin{equation*}
\vt\big(y_t\big)=\vt(\oy)+\la\nabla\vt(\oy),y_t-\oy\ra+\sm\d^2\vt(\oy)(y_t-\oy)+o(t^2)
\end{equation*}
holds for all $t>0$ sufficiently small. Thus we get the chain of equalities
\begin{eqnarray*}
\disp\ph(\ox+t u)-\ph(\ox)-t\la\nabla\ph(\ox),u\ra &=&\vt(y_t)-\vt(\oy)-t\la\nabla\vt(\oy),\nabla f(\ox)u\ra\\
&=&\la\nabla\vt(\oy),y_t-y_0\ra+\sm\d^2\vt(\oy)(y_t-\oy)\\
&&-t\la\nabla\vt(\oy),\nabla f(\ox)u\ra+o(t^2)\\
&=&\big\la\nabla\vt(\oy),t\nabla f(\ox)u+\sm t^2\nabla^2f(\ox)(u,u)\big\ra\\
&&+\sm\d^2\vt(\oy)\big(t\nabla f(\ox)u+\sm t^2\nabla^2f(\ox)(u,u)+o(t^2)\big)\\
&&-t\big\la\nabla\vt(\oy),\nabla f(\ox)u\big\ra+o(t^2)\\
&=&\sm t^2\big\la\nabla\vt(\oy),\nabla^2f(\ox)(u,u)\big\ra+o(t^2)\\
&&+\sm\d^2\vt(\oy)\big(t\nabla f(\ox)u+\sm t^2\nabla^2f(\ox)(u,u)+o(t^2)\big).
\end{eqnarray*}
Remembering that $\d^2\vt(\oy)$ is a continuous function, dividing the above equalities by $\sm t^2$, and then letting $u\to w$ and $t\dn 0$ verify the twice semidifferentiability of $\ph$ at $\ox$ and justify the second-order expansion \eqref{secgof1}. Furthermore, this yields the twice epi-differentiable of $\ph$ at $\ox$ for $\nabla\ph(\ox)$. Finally, we observe that the claimed second-order expansion in (ii) comes from \cite[Exercise~13.7(c)]{rw}.
\end{proof}

Although the twice semidifferentiability assumption on the outer function $\vt$ in Proposition~\ref{tsdch} seems to be restrictive, it holds in some important
settings that appear in numerical algorithms for constrained optimization problems. As Rockafellar demonstrated in \cite[Theorem~4.3]{r20}, a convex function $\vt\colon\R^m\to\oR$ is twice semidifferentiable at $\oy\in\dom\nabla\vt$ if and only if it is twice epi-differentiable at $\oy$ for $\nabla\vt(\ox)$ and $\d^2\vt(\oy,\nabla\vt(\oy))(w)$ is finite for any $w\in\R^m$. Now we utilize this result for the augmented Lagrangians of \eqref{coop}. Given $(x,\lm,\rho)\in\R^n\times\R^m\times(0,\infty)$, the {\em augmented Lagrangian} associated with the constrained problem \eqref{coop} is defined by
\begin{equation}\label{aug}
{\cal L}(x,\lm,\rho):=\ph(x)+\frac{\rho}{2}\big[{\rm dist}\big(f(x)+\rho^{-1}\lm;\Th\big)^2-\|\rho^{-1}\lm\|^2\big].
\end{equation}
Given $\psi\colon\R^n\to\oR$ and $r>0$, define the {\em Moreau envelope} of $\psi$ relative to $r$ by
\begin{equation}\label{moreau}
e_r\psi(x)=\inf_w\Big\{\psi(w)+\frac{1}{2r}\|w-x\|^2\Big\},\quad x\in\R^n.
\end{equation}
When $\psi=\dd_\O$ for some $\O\subset\R^n$, we get $e_r\dd_\O(x)=\frac{1}{2r}{\rm dist}(x;\O)^2$. It is well known that
the Moreau envelope $e_r\dd_\O$ associated with a closed and convex set $\O$ is ${\cal C}^1$-smooth on $\R^n$ and its gradient is calculated by
\begin{equation*}
\nabla\big(e_r\dd_\O\big)(x)=\frac{1}{r}\big(x-P_\O(x)\big)
\end{equation*}
via the projection operator $P_\O$ for the set $\O$. Using the Moreau envelope \eqref{moreau} of $\psi$ relative to $r=\rho^{-1}$, we can equivalently express the corresponding augmented Lagrangian \eqref{aug} by
\begin{equation*}
{\cal L}(x,\lm,\rho)=\ph(x)+\big(e_{1/\rho}\dd_\O\big)\big(f(x)+\rho^{-1}\lm\big)-\frac{\rho}{2}\|\rho^{-1}\lm\|^2.
\end{equation*}

Taking now a pair $(\ox,\bar\lm)$ satisfying the KKT first-order necessary optimality condition \eqref{kkt} and remembering that the set $\Th$ in \eqref{coop} is closed and convex, we can easily check that
\begin{equation*}
\nabla\big(e_{1/\rho}\dd_\Th\big)\big(f(\ox)+\rho^{-1}\olm\big)=\olm.
\end{equation*}
Thus for any $\rho>0$ the augmented Lagrangian \eqref{aug} is differentiable at $(\ox,\olm,\rho)$, and we have
\begin{equation}\label{kkt2}
\nabla_x{\cal L}(\ox,\olm,\rho)=\nabla\ph(\ox)+\nabla f(\ox)^*\nabla\big(e_{1/\rho}\dd_\Th\big)\big(f(\ox)+\rho^{-1}\olm\big)=\nabla_xL(\ox,\olm)=0.
\end{equation}

The next theorem establishes twice semidifferentiability and twice epi-differentiability of the augmented Lagrangian associated with the constrained problem \eqref{coop} under parabolic regularity and derives precise formulas for computing its second semiderivative and second subderivative together with verifying the second-order expansion.

\begin{Theorem}[\bf second semiderivatives and subderivatives of augmented Lagrangians]\label{AL} Let $(\ox,\olm)$ satisfy the first-order optimality condition \eqref{kkt} for problem \eqref{coop}. Assume that $\Th$ is parabolic derivable at $f(\ox)$ for every vector from $K_\Th(f(\ox),\olm)$ and that $\Th$ is parabolically regular at $f(\ox)$ for $\olm$. For each $\rho>0$ consider the function
\begin{equation}\label{psi}
x\mapsto{\cal L}(x,\olm,\rho)\;\mbox{ for all }\;x\in\R^n
\end{equation}
defined via the augmented Lagrangian \eqref{aug}. Then the following hold:

{\bf(i)} Function \eqref{psi} is twice semidifferentiable at $\ox$, and for any $w\in\R^n$ we have
\begin{equation*}
\d_x^2{\cal L}(\ox,\olm,\rho)(w)=\big\la\nabla_{xx}^2L(\ox,\olm)w,w\big\ra+e_{1/2\rho}\big(\d^2\dd_\Th(f(\ox),\olm)\big)\big(\nabla f(\ox)w\big),
\end{equation*}
where $\d_x^2{\cal L}(\ox,\olm,\rho)$ is the second semiderivative of the augmented Lagrangian with respect to $x$.

{\bf(ii)} Function \eqref{psi} satisfies the second-order expansion
\begin{equation*}
{\cal L}(x,\olm,\rho)=\ph(\ox)+\sm\d_x^2{\cal L}(\ox,\olm,\rho)(x-\ox)+o(\|x-\ox\|^2).
\end{equation*}

{\bf(iii)} Function \eqref{psi} is twice epi-differentiable at $\ox$ for $0$, and we have the equality
\begin{equation*}
\d_x^2{\cal L}\big((\ox,\olm,\rho),0\big)=\d_x^2{\cal L}(\ox,\olm,\rho)
\end{equation*}
telling us that the second subderivative of the augmented Lagrangian with respect of $x$ at $(\ox,\olm,\rho)$ for $\ov=0$ agrees with its second semiderivative of \eqref{aug} with respect of $x$ at this triple.
\end{Theorem}
\begin{proof} Fix $\rho>0$ and define the function $\vt$ by $\vt(y):=(e_{1/\rho}\dd_\Th)(y)$ for all $y\in\R^m $. Hence the augmented Lagrangian \eqref{aug}
can be expressed as
\begin{equation*}
{\cal L}(x,\olm,\rho)=\ph(x)+\vt\big(f(x)+\rho^{-1}\olm\big)-\frac{\rho}{2}\|\rho^{-1}\olm\|^2.
\end{equation*}
We further proceed with the following claim.\\[1ex]
{\bf Claim}. {\em For any $\rho>0$ the function $e_{1/\rho}\dd_\Th$ is twice semidifferentiable at $f(\ox)+\rho^{-1}\olm$.}\vspace*{0.05in}

To justify this claim, we conclude from Theorem~\ref{pri92} that $\dd_\Th$ is properly twice epi-differentiable at $f(\ox)$ for $\olm$. Appealing now to \cite[Theorem~6.5]{pr96} (see also \cite[Proposition~4.1]{hms2}) implies that the Moreau envelope $e_{1/\rho}\dd_\Th$ is twice epi-differentiable at $f(\ox)+\rho^{-1}\olm$ for $\nabla  (e_{1/\rho}\dd_\Th)(f(\ox)+\rho^{-1}\olm)=\olm$. Let us observe here that since we employ \cite[Theorem~6.5]{pr96} for the convex function $\dd_\Th$, the constant $r$ in \cite[Theorem~6.5]{pr96} is $0$, and so it is not required to assume in our setting that $\rho$ is sufficiently large. Using \cite[Proposition~4.1]{hms2} tells us that
\begin{equation}\label{more}
\d^2\big(e_{1/\rho}\dd_\Th\big)\big(f(\ox)+\rho^{-1}\olm,\olm)=e_{1/2\rho}\big(\d^2\dd_\Th(f(\ox),\olm)\big).
\end{equation}
Remember that $\dd_\Th$ is a proper, l.s.c., and convex. This implies by \cite[Theorem~2.26(b)]{rw} that $e_{1/\rho}\dd_\Th$ is convex and ${\cal C}^1$-smooth.
It follows from \eqref{more} that the second subderivative $\d^2(e_{1/\rho}\dd_\Th)(f(\ox)+\rho^{-1}\olm,\olm)$ is finite on $\R^m$. Using further \cite[Theorem~4.3]{r20} ensures that $e_{1/\rho}\dd_\Th$ is twice semidifferentiable at $f(\ox)+\rho^{-1}\olm$ with the second subderivative
\begin{equation*}
\d^2\big(e_{1/\rho}\dd_\Th\big)\big(f(\ox)+\rho^{-1}\olm)=\d^2\big(e_{1/\rho}\dd_\Th\big)\big(f(\ox)+\rho^{-1}\olm,\olm\big).
\end{equation*}
The latter means that the second semiderivative of $e_{1/\rho}\dd_\Th$ at $f(\ox)+\rho^{-1}\olm$ agrees with the second subderivative of $e_{1/\rho}\dd_\Th$ at this point for $\olm$, which verifies the claim.

Combining the established claim with Proposition~\ref{tsdch} tells us that the function $x\mapsto\vt(f(x)+\rho^{-1}\olm)$ is twice semidifferentiable at $\ox$. Using further the sum rule from Proposition~\ref{sumr} ensures that the function $x\mapsto\ph(x)+\vt(f(x)+\rho^{-1}\olm)$ is twice semidifferentiable at $\ox$, and hence the augmented Lagrangian $x\mapsto{\cal L}(x,\olm,\rho)$ shares this property. Moreover, it follows from Proposition~\ref{sumr} and Proposition~\ref{tsdch}(i) that
\begin{eqnarray*}
\d_x^2{\cal L}(\ox,\olm,\rho)(w)&=&\d^2\ph(\ox)(w)+\big\la\olm,\nabla^2f(\ox)(w,w)\big\ra+\d^2\big(e_{1/\rho}\dd_\Th\big)\big(f(\ox)+\rho^{-1}\olm\big)\big(\nabla f(\ox)w\big)\\
&=&\la\nabla^2\ph(\ox)w,w\ra+\big\la\olm,\nabla^2f(\ox)(w,w)\big\ra+\d^2\big(e_{1/\rho}\dd_\Th\big)\big(f(\ox)+\rho^{-1}\olm,\olm\big)\big(\nabla f(\ox)w\big)\\
&=&\big\la\nabla_{xx}^2L(\ox,\olm)w,w\big\ra+e_{1/2\rho}\big(\d^2\dd_\Th(f(\ox),\olm)\big)\big(\nabla f(\ox)w\big),
\end{eqnarray*}
where the last equality comes from \eqref{more}. This shows therefore that (i) holds.

Assertions (ii) follows directly from Proposition~\ref{tsdch}(ii) combined with the facts that ${\cal L}(\ox,\olm,\rho)=\ph(\ox)$ and $\nabla_x{\cal L}(\ox,\olm,\rho)=0$ as shown in \eqref{kkt2}. To verify finally (iii), note by \eqref{kkt2} that $\nabla_x{\cal L}(\ox,\olm,\rho)=0$. Thus the twice epi-differentiability of the function $x\mapsto{\cal L}(x,\olm,\rho)$ at $\ox$ for $\ov=0$  is a consequence of the above discussion and Proposition~\ref{tsdch}(iii).
\end{proof}

The twice semidifferentiability of the augmented Lagrangian \eqref{aug} in Theorem~\ref{AL}(i) was discussed in \cite[equation~(3.25)]{ss} under the name of ``second-order Hadamard directional differentiability'' by using a different approach in the case where the set $\Th$ is second-order regular. Recall that the second-order regularity is strictly stronger than the parabolic regularity extensively developed in this paper. The second-order expansion in Theorem~\ref{AL}(ii) was derived for nonlinear programming problems in \cite[Proposition~7.2]{r93} by employing yet another approach.\vspace*{0.05in}

The next major result establishes the validity of the quadratic growth condition \eqref{gro} for the augmented Lagrangian \eqref{aug} associated with \eqref{coop} under the parabolic regularity of $\Th$. Moreover, we prove the {\em equivalence}--again under the parabolic regularity--of the latter growth condition to the second-order sufficient optimality condition \eqref{2gro} for \eqref{coop} as well as to the positivity of the second subderivative of \eqref{aug} with respect to $x$.

\begin{Theorem}[\bf quadratic growth condition for augmented Lagrangians]\label{augl} Let the pair $(\ox,\olm)$ satisfy the first-order optimality condition \eqref{kkt} under the assumptions that:

$\bullet$ $\Th$ is parabolically derivable at $f(\ox)$ for every vector from $K_\Th(f(\ox),\olm)$.

$\bullet$ $\Th$ is parabolically regular at $f(\ox)$ for $\olm$.

$\bullet$ The second subderivative $\d^2\dd_\Th(f(\ox),\olm)$ is continuous relative to its domain which is $K_\Th(f(\ox),\olm)$.\vspace*{0.05in}
Then the following assertions are equivalent:

{\bf(i)} The second-order sufficient condition \eqref{2gro} holds for all vectors $w\in\R^n\setminus\{0\}$ satisfying $\nabla f(\ox)w\in K_\Th(f(\ox),\olm)$.

{\bf(ii)} There exists a positive number $\bar\rho>0$ such that for any $\rho>\bar\rho$ we have
\begin{equation*}
\d_x^2{\cal L}\big((\ox,\olm,\rho),0\big)(w)>0\;\mbox{ whenever }\;w\in\R^n\setminus\{0\}.
\end{equation*}

{\bf(iii)} There are $\bar\rho>0$, $\ve>0$, and $\ell>0$ {\rm(}all dependent on $\olm${\rm)} such that for any $\rho\ge \bar \rho$ we have
\begin{equation}\label{aug2}
{\cal L}(x,\olm,\rho)\ge\ph(\ox)+\frac{\ell}{2}\,\|x-\ox\|^2\;\mbox{ whenever }\;x\in\B_\ve(\ox).
\end{equation}
\end{Theorem}
\begin{proof} Assume first that (ii) holds for $\rho=\bar\rho$. Employing \cite[Theorem~13.24(c)]{rw} tells us that there exist positive numbers $\ve$ and $\ell$ such that
\begin{equation*}
{\cal L}(x,\olm,\bar\rho)\ge f(\ox)+\frac{\ell}{2}\,\|x-\ox\|^2\;\mbox{ for all }\;x\in\B_{\ve}(\ox)
\end{equation*}
with the usage of the equality ${\cal L}(\ox,\olm,\bar\rho)=f(\ox)$. Since the function $\rho\mapsto{\cal L}(x,\olm,\rho)$ is nondecreasing due to \cite[Exercise~11.56]{rw}, we get (iii) for any $\rho\ge\bar\rho$. Then implication (iii)$\implies$(ii) comes from the definition of the second subderivative.

Assume now that (ii) holds and fix the numbers $\bar\rho,\rho$ therein. Theorem~\ref{AL} and the Moreau envelope construction \eqref{moreau} ensure that
\begin{eqnarray*}
\big\la\nabla_{xx}^2L(\bar x,\olm)w,w\big\ra+\d^2\dd_\Th\big(f(\ox),\olm\big)\big(\nabla f(\ox)w)&\ge&\la\nabla_{xx}^2L(\ox,\olm)w,w\ra+e_{1/2\rho}\big(\d^2\dd_\Th(f(\ox),\olm)\big)\big(\nabla f(\ox)w\big)\\
&=&\d_x^2{\cal L}(\ox,\olm,\rho)(w)=\d_x^2{\cal L}\big((\ox,\olm,\rho),0\big)(w)>0
\end{eqnarray*}
for all $w\in\R^n\setminus\{0\}$, which in turn justifies (i).

To verify the opposite implication, assume that (i) holds and define the sets
\begin{equation*}
S:=\big\{w\in\R^n\big|\;\|w\|=1\big\}\;\mbox{ and }\;E:=\big\{w\in\R^n\big|\;\nabla f(\ox)w\in K_\Th\big(f(\ox),\olm\big)\big\}.
\end{equation*}
Since $\Th$ is parabolically derivable at $f(\ox)$ for every vector in $K_\Th(f(\ox),\olm)$, we deduce from Theorem~\ref{pri}(i) that the function $\d^2\dd_\Th(f(\ox),\olm)$ is proper and lower semicontinuous. This ensures that the second-order condition \eqref{2gro} amounts to the existence of $\ell>0$ such that
\begin{equation}\label{soscp2}
\big\la\nabla_{xx}^2L(\bar x,\olm)w,w\big\ra+\d^2\dd_\Th\big(f(\ox),\olm\big)\big(\nabla f(\ox)w\big)\ge\ell\|w\|^2\;\mbox{ for all }\; w\in E.
\end{equation}
Consider further the function $\chi_{\rho}\colon\R^n\to\oR$ given by
\begin{equation*}
\chi_{\rho}(w):=\d_x^2{\cal L}\big((\ox,\olm,\rho),0\big)(w)=\big\la\nabla_{xx}^2L(\bar x,\olm)w,w\big\ra+e_{1/2\rho}\big(\d^2\dd_\Th(f(\ox),\olm)\big)\big(\nabla f(\ox)w\big),\;\;w\in \R^n.
\end{equation*}
It follows from the convexity of $\Th$, the parabolic regularity of $\dd_\Th$ at $f(\ox)$ for $\olm$, and Theorem~\ref{pri92} that $\dd_\Th$ is properly twice epi-differentiable at $f(\ox)$ for $\olm$. Appealing now to \cite[Proposition~13.20(a)]{rw} indicates that $\d^2\dd_\Th(f(\ox),\olm)$ is a convex function. Hence we deduce from \cite[Theorem~2.26(b)]{rw} that the function $\chi_{\rho}$ is finite and continuous on $\R^n$ for any $\rho>0$.

Next we show that $\chi_{\rho}(w)>0$ for all $w\in S$ whenever $\rho>0$ is sufficiently large. This is accomplished by the following two steps.\\[1ex]
{\bf Step~1:} {\em There are an open set $V\subset\R^n$ and a number $\bar\rho_1>0$ with $S\cap E\subset S\cap V$ and
\begin{equation*}
\chi_{\rho}(w)>0\;\mbox{ for all }\;w\in S\cap V\;\mbox{ and all }\;\rho>\bar\rho_1.
\end{equation*}}
To verify this, consider the function
\begin{equation*}
\chi(w):=\big\la\nabla_{xx}^2L(\bar x,\olm)w,w\big\ra+\d^2\dd_\Th\big(f(\ox),\olm\big)\big(\nabla f(\ox)w\big),\quad w\in S.
\end{equation*}
Since $E$ is closed and $S$ is compact, the set $S\cap E$ is obviously compact as well. It follows from our assumptions in this theorem that $\chi$ is continuous relative to the compact set $S\cap E$, which allows us to deduce from \cite[Theorem~1.25]{rw} that
\begin{equation*}
\chi_{\rho}(w)\uparrow\chi(w)\;\mbox{ as }\;\rho\to\infty\;\mbox{ for all }\;w\in S\cap E.
\end{equation*}
Employing the Dini theorem from \cite[Theorem~7.13]{ru} ensures that the above pointwise convergence becomes uniform on $S\cap E$. Thus for any $\ve>0$ we find  $\bar{\rho}_1>0$ such that
\begin{equation*}
|\chi_{\bar{\rho}_1}(w)-\chi(w)|<\ve\;\mbox{ whenever }\;w\in S\cap E.
\end{equation*}
In particular, for $\ve:=\ell/2$ with $\ell>0$ taken from \eqref{soscp2} it follows that
\begin{equation}\label{hpos}
\chi_{\bar{\rho}_1}(w)>\frac{\ell}{2}\;\mbox{ whenever }\;w\in S\cap E.
\end{equation}
Now we claim that there exists an open set $V\subset\R^n$ such that $S\cap E\subset S\cap V$ and
\begin{equation*}
\chi_{\bar{\rho}_1}(w)>\frac{\ell}{4}\;\mbox{ for all }\;w\in S\cap V.
\end{equation*}
To justify it, pick $w\in S\cap E$ and remember that $\chi_{\bar{\rho}_1}$ is continuous at $w$ with $\chi_{\bar{\rho}_1}(w)>\ell/2$ due to \eqref{hpos}.
This gives us a neighborhood $U_w$ of $w$ in $\R^n$ for which
\begin{equation*}
\chi_{\bar{\rho}_1}(u)>\frac{\ell}{4}\;\mbox{ whenever }\;u\in U_w.
\end{equation*}
Setting $V:=\bigcup_{w\in S\cap E}U_{w}$, which is open in $\R^n$ but depends in $\bar\rho_1$, we see that $S\cap E\subset S\cap V$ and that $\chi_{\bar{\rho}_1}(w)>\frac{\ell}{4}$ for all $w\in S\cap V$, and hence our claim is verified. Pick now $\rho>\bar\rho_1$ and deduce from the monotonicity of the functions $\chi_\rho$ with respect to $\rho$ that
\begin{equation*}
\chi_{\rho}(w)\ge\chi_{\bar{\rho}_1}(w)>\frac{\ell}{4}>0\;\mbox{ for all }\;w\in S\cap V,
\end{equation*}
which therefore completes the proof of Step~1.\\[1ex]
{\bf Step~2:} {\em There exists a number $\bar\rho_2>0$ such that
\begin{equation*}
\chi_{\rho}(w)>0\;\mbox{ for all }\;w\in S\cap V^c\;\mbox{ and all}\;\rho>\bar\rho_2,
\end{equation*}
where $V$ is taken from Step~{\rm 1}, and where $V^c$ stands for the complement of $V$ in $\R^n$.}\\[1ex]
To prove this statement, note first that $V^c$ is a closed set and so $S\cap V^c$ is compact. It is not hard to check the implication
\begin{equation}\label{kdcf}
w\in S\cap V^c\implies\nabla f(\ox)w\notin K_\Th\big(f(\ox),\olm\big).
\end{equation}
Defining further the real quantities
\begin{equation*}
\al:=\min_{w\in S\cap V^c}\big\la\nabla_{xx}^2L(\bar x,\olm)w,w\big\ra\;\mbox{ and }\;\beta:=\min_{w\in S\cap V^c}\dist\big(\nabla f(\ox)w;K_\Th(f(\ox),\olm)\big)^2,
\end{equation*}
we observe from \eqref{kdcf} and the compactness of $S\cap V^c$ that $\beta>0$.  It follows from Theorem~\ref{pri}(i) that $\d^2\dd_\Th(f(\ox),\olm)(w)\ge 0$ for all $w\in\R^n$; since $\Th$ is convex, the constant $r$ in Theorem~\ref{pri}(i) is zero. Furthermore, the aforementioned theorem tells us that $\dom\d^2\dd_\Th(f(\ox),\olm)=K_\Th(f(\ox),\olm)$. Combining these results readily yields
\begin{equation*}
\d^2\dd_\Th\big(f(\ox),\olm\big)(w)\ge\dd_{K_\Th(f(\ox),\olm)}(w)\;\mbox{ for all }\;w\in\R^n.
\end{equation*}
Let $\bar\rho_2:=\max\{1,-\frac{\al}{\beta}\}$. Then whenever $\rho>\bar{\rho}_2 $ and $w\in S\cap V^c$ we get
\begin{eqnarray*}
\chi_{\rho}(w)&=&\big\la\nabla_{xx}^2L(\bar x,\olm)w,w\big\ra+e_{1/2\rho}\big(\d^2\dd_\Th(f(\ox),\olm)\big)\big(\nabla f(\ox)w\big)\\
&\ge&\big\la\nabla_{xx}^2L(\bar x,\olm)w,w\big\ra+e_{1/2\rho}\big(\dd_{K_\Th(f(\ox),\olm)}\big)\big(\nabla f(\ox)w\big)\\
&=&\big\la\nabla_{xx}^2L(\bar x,\olm)w,w\big\ra+{\rho}\,{\rm dist}\big(\nabla f(\ox)w;K_\Th(f(\ox),\olm)\big)^2\\
&\ge&\al+{\rho}\beta>0,
\end{eqnarray*}
which therefore verifies the statement of Step~2.\vspace*{0.05in}

Unifying the results established in Step~1 and Step~2 brings us to the inequality
\begin{equation*}
\d_x^2{\cal L}\big((\ox,\olm,\rho),0\big)(w)=\chi_{\rho}(w)>0\;\mbox{ for all }\;w\in S\;\mbox{ and all }\;\rho>\bar\rho:=\max\big\{\bar\rho_1,\bar\rho_2\big\}.
\end{equation*}
Taking finally any $w\in\R^n\setminus\{0\}$ and recalling that the second subderivative is positive homogeneous of degree $2$, we obtain for all $\rho>\bar\rho$ that
\begin{equation*}
\d_x^2{\cal L}\big((\ox,\olm,\rho),0\big)(w)=\|w\|^2\d_x^2{\cal L}\big((\ox,\olm,\rho),0\big)\Big(\frac{w}{\|w\|}\Big)>0,
\end{equation*}
which justifies (ii) and thus completes the proof of the theorem.
\end{proof}

Let us conclude this section with brief discussions on previous efforts to obtain the quadratic growth condition for augmented Lagrangians and the main assumptions of Theorem~\ref{augl}.

\begin{Remark}[\bf on quadratic growth for augmented Lagrangians]\label{ag-diss} {\rm Observe that:

{\bf(i)} There have been some developments in order to establish implication (i)$\implies$(iii) in Theorem~\ref{augl} for different classes of constrained optimization problems. Rockafellar  in \cite[Theorem~7.4]{r93} derived this implication for nonlinear programming problems without appealing to the second subderivative, the main player in our proof. Liua and Zhang obtained this result for second-order cone programming problems in \cite[Proposition~10]{lz} when in addiction the strict complementarity condition and some nondegeneracy condition were imposed. For semidefinite programming problems a similar result was achieved in \cite[Proposition~4]{ssz} by assuming a stronger version of the second-order sufficient condition \eqref{2gro} together with a nondegeneracy condition. Theorem~\ref{augl} provides an extension of Rockafellar's result for {\em any parabolically regular} constrained optimization problem including second-order cone programs, semidefinite programs, etc. Moreover, we also show that the quadratic growth condition \eqref{aug2} for the augmented Lagrangians is actually {\em equivalent} to the second-order sufficient condition \eqref{2gro}. The latter was not observed before even in nonlinear programming.

{\bf(ii)} Finally, we briefly discuss the main assumptions in Theorem~\ref{augl}. As mentioned earlier, the parabolic derivability and parabolic regularity hold for any convex polyhedral set (Example~\ref{poly}), for the second-order cone (Example~\ref{ice}), and---more generally---for any ${\cal C}^2$-cone reducible set (Theorem~\ref{twi20}). The continuity of the second subderivative relative to its domain is satisfied for any ${\cal C}^2$-cone reducible sets according to \eqref{tepired}. It is not clear at this stage for us whether or not such an assumption holds for any parabolically regular set in general. What we do know from Theorem~\ref{pri}(i) is that the second subderivative is always lower semicontinuous.}
\end{Remark}

\section{Subgradient Graphical Derivatives via Parabolic Regularity}\sce\label{sect07}

The section is devoted to precise calculating the graphical derivatives of the normal mappings generated the constraint systems. In other words, we intend to derive exact formulas for computing the subgradient graphical derivative of the indicator function for the set $\O$ from \eqref{CS}. Theorem~\ref{proto}(ii) gives us a road map to reach this goal. Indeed, by \eqref{gdpd} we should try to find the subdifferential of the second subderivative of $\dd_\O$ calculated in Theorem~\ref{epi}. This can be achieved by appealing to a recent result of \cite[Theorem~3]{chl}, which provides an advanced subdifferential formula for functions represented as the supremum of infinitely many convex ones. The key here is the last formula established in Theorem~\ref{epi} for the second subderivative that only requires to take the maximum over a compact subset of the collection of Lagrange multipliers \eqref{lagn}. As we see below, parabolic regularity and its properties established in the previous sections play a crucial role in our approach.\vspace*{0.05in}

We begin with the following result showing that if the mapping $f$ in \eqref{CS} is of class ${\cal C}^2$, then $\O$ from \eqref{CS} is prox-regular. This allows us to use \eqref{gdpd} for calculating the graphical derivative of the normal cone mapping $N_\O$. We omit the proof of this result that follows the lines of \cite[Corollary~2.12]{pr96} where it was done under metric regularity ensuring the boundedness of the Lagrange multiplier set $\Lambda(\ox,\ov)$. Although it is not the case under the imposed MSCQ \eqref{mscq}, we can proceed as the proof of \cite[Proposition~7.1]{mms} to alleviate the hardship.

\begin{Proposition}[\bf prox-regularity of constraint systems]\label{sm} Let in addition to the basic assumptions {\rm(H1)--(H4)} the mapping $f$ from \eqref{CS} be ${\cal C}^2$-smooth around $\ox$.
Then the set $\O$ in \eqref{CS} is prox-regular at $\ox$ for any normal vector $\ov\in N_\O(\ox)$.
\end{Proposition}

Using the prox-regularity of $\O$ and the second-order optimality conditions obtained in Theorem~\ref{nsop1}, we derive now a pointwise second-order characterization of the important notion of strong metric subregularity for subdifferential mappings associated with constrained optimization problems \eqref{coop} with parabolically regular sets $\Th$. Recall that a set-valued mapping $F\colon\R^n\tto\R^m$ is {\em strongly metrically subregular} at $(\ox,\oy)\in\gph F$ if there exist a constant $\kappa\in\R_+$ and a neighborhood $U$ of $\ox$ ensuring the distance estimate
\begin{equation*}
\|x-\ox\|\le\kappa\,{\rm dist}\big(\oy;F(x)\big)\;\mbox{ for all }\;x\in U.
\end{equation*}
The Levy-Rockafellar criterion (see \cite[Theorem~4E.1]{dr} and the commentaries therein) tells us that  the mapping $F$ is strongly metrically regular at $(\ox,\oy)$ if and only if we have the implication
\begin{equation}\label{sms8}
0\in DF(\ox,\oy)(w)\implies w=0.
\end{equation}
The next result was first observed in \cite[Theorem~4G.1]{dr} for a special subclass of nonlinear programming problems and then was extended in \cite[Theorem~4.6]{chnt} for ${\cal C}^2$-cone reducible constrained optimization problems. Now we are able to establish it for more general constrained problems \eqref{coop} generated by parabolically regular sets $\Th$.

\begin{Theorem}[\bf strong metric subregularity of subgradient mappings]\label{sms} Let the basic assumptions {\rm(H1)--(H4)} hold for $\O$ from \eqref{coop2}, and let $\ov:=-\nabla\ph(\ox)$.
Assume further that $\Th$ in \eqref{CS} is parabolically regular at $f(\ox)$ for every Lagrange multiplier $\lambda\in\Lambda(\ox,\ov)$, and that both $\ph$ and $f$ are ${\cal C}^2$-smooth around $\ox$.
Then the following assertions are equivalent:

{\bf(i)} The point $\ox$ is a local minimizer of $\psi:=\ph+\dd_\O$, and the subgradient mapping $\sub f$ is strongly metrically subregular at $(\ox,0)$.

{\bf(ii)} The second-order sufficient optimality condition \eqref{sscc} is satisfied.
\end{Theorem}
\begin{proof} It follows from the proof of  Theorem~\ref{nsop1} that the second-order sufficient optimality condition \eqref{sscc} amount to saying that
\begin{equation}\label{vbnm}
\d^2\psi(\ox,0)(w)>0\;\mbox{ for all }\;w\in\R^n\setminus\{0\}.
\end{equation}
Assume first that (i) holds and then deduce from the local minimality of $\ox$ in \eqref{coop2} and Theorem~\ref{nsop1}(i) that $\d^2\psi(\ox,0)(w)\ge 0$ whenever $w\in\R^n$. To verify \eqref{vbnm}, suppose on the contrary that there exists a vector $\ow\ne 0$ such that $\d^2\psi(\ox,0)(\ow)=0$. Consider the optimization problem
\begin{equation*}
\mbox{minimize}\;\sm\d^2\psi(\ox,0)(w)\;\mbox{ subject to }\;w\in\R^n
\end{equation*}
for which $\ow$ is clearly a minimizer. Furthermore, Proposition~\ref{sm} ensures that the set $\O$ is prox-regular at $\ox$ for $\ov$. Using the subdifferential Fermat rule and Corollary~\ref{tedi} together with the equalities in \eqref{gdpd} and \eqref{ns1} gives us the relationships
\begin{eqnarray*}
0\in\sub\big(\sm\d^2\psi(\ox,0)\big)(\ow)&=&\nabla^2\ph(\ox)\ow+\sub\big(\sm\d^2\dd_\O(\ox,\ov)\big)(\ow)\\
&=&\nabla^2\ph(\ox)\ow+DN_\O(\bar x,\ov)(\ow)\\
&=&D\big(\nabla\ph+N_\O\big)(\ox,0)(\ow)=\big(D\sub\psi\big)(\ox,0)(\ow).
\end{eqnarray*}
Since the subgradient mapping $\sub\psi$ is metrically subregular at $(\ox,0)$, we conclude from \eqref{sms8} that $\ow=0$, a contradiction. This justifies \eqref{vbnm}, and thus we arrive at (ii).

Assume now that (ii) holds. Then Theorem~\ref{nsop1}(ii) tells us that $\ox$ is a local minimizer of $\psi$. To prove the strong metric subregularity of $\sub\psi$, let $0\in(D\sub\psi)(\ox,0)(w)$. Criterion \eqref{sms8} reduces our task to checking that $w=0$. As argued above, we have the representation
\begin{equation*}
\big(D\sub\psi\big)(\ox,0)(w)=\sub\big(\sm\d^2\psi(\ox,0)\big)(w),\quad w\in\R^n,
\end{equation*}
which implies $0\in\sub\big(\sm\d^2\psi(\ox,0)\big)(w)$. Using \cite[Lemma~3.7]{chnt} gives us $\d^2\psi(\ox,0)(w)=\la 0,w\ra=0$. Combining this with \eqref{vbnm} yields $w=0$ and thus  completes the proof of the theorem.
\end{proof}

To establish the main result of this section, we first present the following lemma.

\begin{Lemma}[\bf convexity of a family of quadratic functions]\label{convexq} Let $f\colon\R^n\to\R^m$ be twice differentiable at $\ox$, and let $E\subset\R^m$ be a compact set. Given any $\rho\in\R$ and any $\lambda\in\R^m$, consider the quadratic form
\begin{equation*}
\xi_{\lambda}(w):=\big\la\lm,\nabla^2f(\bar x)(w,w)\big\ra+\rho\|w\|^2,\quad w\in\R^n.
\end{equation*}
There exists $\rho>0$ such that $\xi_{\lambda}\colon\R^n\to\R$ is a convex function for each $\lm\in E$.
\end{Lemma}
\begin{proof}
Observe for any $\lm\in\R^m$ the convexity of $\xi_{\lambda}$ on $\R^n$ amounts to its convexity relative to every line in $\R^n$. Pick any $w\in\R^n$ and $d\in\R^n$ with $\|d\|=1$. Then
\begin{equation*}
\xi_{\lambda}(w+t d)=\big(\big\langle\lm,\nabla^2f(\bar x)(d,d)\big\rangle+\rho\big)t^2+\theta(t),\quad t\in\R,
\end{equation*}
where $\th(t)$ is a polynomial of degree less than $2$. Selecting $\rho>0$ such that
\begin{equation*}
\rho>\max\big\{-\big\langle\lm,\nabla^2f(\bar x)(d,d)\big\rangle\big|\:\lambda\in E,\;\|d\|=1\big\},
\end{equation*}
we can easily check that $\xi_\lm$ is convex on every line in $\R^n$ and thus complete the proof.
\end{proof}

Now we are ready to derive the main result here, which presents precise formulas to calculate the subgradient graphical derivative associated with parabolically regular constraint systems.

\begin{Theorem}[\bf subgradient graphical derivative of constraint systems]\label{gd} In addition to the basic assumptions {\rm(H1)--(H4)}, let $f$ be a ${\cal C}^2$-smooth mapping. Then the normal cone mapping $N_\O$ generated by the constraint system $\O$ from \eqref{CS} is proto-differentiable at $\ox$ for $\ov$ and its graphical derivative is calculated by the formulas
\begin{eqnarray*}\label{gd0}
DN_\O(\bar x,\bar v)(w)&=&\bigcup_{\lm\in\Lambda(\bar x,\bar v,w)\cap(\kappa\,\|\ov\|\B)}\nabla^2\la\lm,f\ra(\bar x)w+\sub_w\Big(\sm\d^2\dd_\Th\big(f(\ox),\lambda\big)\big(\nabla f(\ox)\cdot\big)\Big)(w)\\
&=&\bigcup_{\lm\in\Lambda(\bar x,\bar v,w)}\nabla^2\la\lm,f\ra(\bar x)w+\sub_w\Big(\sm\d^2\dd_\Th(f(\ox),\lambda\big)\big(\nabla f(\ox)\cdot\big)\Big)(w)
\end{eqnarray*}
for all $w\in K_\O(\ox,\ov)$, where $\Lambda(\bar x,\bar v,w)$ stands for the set of optimal solutions to the dual problem \eqref{du2}. Moreover, we have $DN_\O(\bar x,\bar v)(w)=\emp$ if $w\notin K_\O(\ox,\ov)$.
\end{Theorem}
\begin{proof}
Proposition~\ref{sm} tells us that the normal cone mapping $N_\O$ is prox-regular at $\ox$ for $\ov$. Combining it with Theorem~\ref{proto}(ii) and Corollary~\ref{tedi} ensures that the mapping $N_\O$ is proto-differentiable at $\ox$ for $\ov$. By \eqref{gdpd} the graphical derivative of this mapping can be obtained by calculating the subdifferential of the second subderivative $\d^2\dd_\O(\ox,\ov)$. To proceed further, pick a real number $r$ with $r\ge\kappa\|\ov\|$, where $\kappa$ is taken from the MSCQ assumption (H3). As discussed in Remark~\ref{rem5}(i), for any $w\in\R^n$ the second subderivative $\d^2\dd_\O(\ox,\ov)(w)$ can be calculated by \eqref{epi4}. Define now the supremum function $\ph\colon\R^n\to\oR$ by
\begin{equation*}
\ph(w):=\disp\sup_{\lm\in\Lambda(\ox,\ov)\cap\,r\B}\ph_{\lm}(w)\;\mbox{ with }\;\ph_{\lm}(w) :=\big\langle\lm,\nabla^2f(\bar x)(w,w)\big\rangle+\rho\|w\|^2+\d^2\dd_\Th\big(f(\bar x),\lm\big)\big(\nabla f(\ox)w\big),
\end{equation*}
where $\rho>0$ is taken from Lemma~\ref{convexq} with $E:=\Lambda(\ox,\ov)\cap r\B$. We claim the following properties, where the abbreviation u.s.c.\ signifies the upper semicontinuity of a scalar function:\vspace*{0.05in}

{\bf(i)} $\dom\ph_{\lm}=\dom\ph=K_\O(\ox,\ov)$ as $\lm\in\Lambda(\ox,\ov)\cap r\B$, and $\ph_{\lm}$ is a proper convex function.

{\bf(ii)} For each $w\in K_\O(\ox,\ov)$ the function $\lm\mapsto\ph_{\lm}(w)$ is concave and u.s.c.\ on $\Lambda(\ox,\ov)$.

{\bf(iii)} For each $\ve\ge 0$ the subset of multipliers
\begin{equation*}
\Gamma_{\ve}^r(w):=\big\{\lambda\in\Lambda(\ox,\ov)\cap r\B\big|\;\ph_{\lm}(w)\ge\ph(w)-\ve\big\}
\end{equation*}
is compact in $\R^m$ whenever $w\in K_\O(\ox,\ov)$.\vspace*{0.05in}

To verify (i), deduce from Theorem~\ref{pri}(i) and \eqref{krie} that $\dom\ph_{\lm}=K_\O(\ox,\ov)$ for any $\lm\in\Lambda(\ox,\ov)\cap r\B$. Since $\ph(\cdot)=\d^2\dd_\O(\ox,\ov)(\cdot)+\rho\|\cdot\|^2$ by \eqref{epi4}, it follows from \eqref{critco} that $\dom\ph=\dom\d^2\dd_\O(\ox,\ov)=K_\O(\ox,\ov)$. Invoking assumption (H4) and Theorem~\ref{pri92} yields the twice epi-differentiability of $\dd_\Th$ at $f(\ox)$ for every $\lm\in\Lambda(\ox,\ov)$.
This allows us to employ \cite[Proposition~13.20(a)]{rw} to conclude  that the function $\d^2\dd_\Th(f(\bar x),\lm)(\nabla f(\ox)\cdot)$ is proper and convex on $\R^n$ and so are the functions $\ph_\lm$,
which proves (i).

To justify now assertion (ii), note that for each $w\in K_\O(\ox,\ov)$ the function
\begin{equation*}
\lm\mapsto\big\langle\lm,\nabla^2f(\bar x)(w,w)\rangle+\rho\|w\|^2-\sigma_{T^2_\Th(f(\ox),\nabla f(\ox)w)}(\lambda)
\end{equation*}
is clearly concave and upper semicontinuous on the set $\Lambda(\ox,\ov)$. Using  assumption (H4) along with Theorem~\ref{pri}(iii), we get the representation $\d^2\dd_\Th(f(\bar x),\lm)(\nabla f(\ox)w)=-\sigma_{T^2_\Th(f(\ox),\nabla f(\ox)w)}(\lambda)$ whenever $\lm\in\Lambda(\ox,\ov)$, which verifies (ii).

Turning finally to the proof of (iii), pick $\ve\ge 0$ and $w\in K_\O(\ox,\ov)$. Since the function $\lm\mapsto\ph_{\lm}(w)$ is u.s.c.\ on $\Lambda(\ox,\ov)$ by (ii), the set
\begin{equation*}
\big\{\lambda\in\Lambda(\ox,\ov)\big|\;\ph_{\lm}(w)\ge\ph(w)-\ve\big\}
\end{equation*}
is closed, and thus $\Gamma_{\ve}^r(w)$ is compact. This completes the proof of our claims (i)--(iii).\vspace*{0.05in}

Observe further that the established claims ensure that the imposed assumptions in \cite[Theorem~3]{chl} are satisfied in our setting. Thus for any $w\in K_\O(\ox,\ov)$ we have
\begin{eqnarray*}
\sub\ph(w)&=&\co\Big\{\bigcup_{\lambda\in\Gamma^r(w)}\sub\big(\ph_{\lambda}+\dd_{K_\O(\ox,\ov)}\big)(w)\Big\}=\co\Big\{\bigcup_{\lambda\in\Gamma^r(w)}\sub\ph_{\lambda}(w)\Big\}\\
&=&2\,\co\Big\{\bigcup_{\lambda\in\Gamma^r(w)}\nabla^2\la\lm,f\ra(\bar x)w+\sub_w\Big(\sm\d^2\dd_\Th\big(f(\ox),\lambda\big)\big(\nabla f(\ox)\cdot\big)\Big)(w)\Big\}+2\rho w,
\end{eqnarray*}
where $\Gamma^r(w):=\{\lambda\in\Lambda(\ox,\ov)\cap r\B\;|\;\ph_{\lm}(w)=\ph(w)\}$. Since $\Gamma^r(w)$ is the set of all $\lm\in\Lambda(\ox,\ov)\cap r\B$ at which the maximum in \eqref{epi4} is achieved, we get $\Gamma^r(w)=\Lambda(\ox,\ov,w)\cap r\B$ by Remark~\ref{rem5}(ii).

As mentioned earlier in the proof, it follows from \eqref{epi4} that $\ph(\cdot)=\d^2\dd_\O(\ox,\ov)(\cdot)+\rho\|\cdot\|^2$. This together with \eqref{gdpd} brings us to the expressions
\begin{eqnarray*}
D N_\O(\ox,\ov)(w)&=&\sub\big(\sm\d^2\dd_\O(\ox,\ov)\big)(w)=\sm\sub\ph(w)-\rho w\\
&=&\co\Big\{\bigcup_{\lambda\in\Lambda(\ox,\ov,w)\cap r\B}\nabla^2\la\lm,f\ra(\bar x)w+\sub_w\Big(\sm\d^2\dd_\Th\big(f(\ox),\lambda\big)\big(\nabla f(\ox)\cdot\big)\Big)(w)\Big\}.
\end{eqnarray*}
Now we are going to show that the convex hull can be dropped in the latter equality. To this end, pick a vector $q$ from the right-hand side of the this equality and find
$\olm_i\in\Lambda(\ox,\ov,w)\cap r\B$ with $s\in\N$ and $\tau_i\in\R_+$ with $i=1,\ldots,s$ such that $q=\sum_{i=1}^{s}\tau_i q_i$ and $\sum_{i=1}^s\tau_i=1$ with
\begin{equation}\label{dn03}
q_i\in\nabla^2\la\olm_i,f\ra(\bar x)w+\sub_w\Big(\sm\d^2\dd_\Th\big(f(\ox),\olm_i\big)\big(\nabla f(\ox)\cdot\big)\Big)(w).
\end{equation}
Define $\olm:=\sum_{i=1}^{s}\tau_i\olm_i$ and observe that the inclusions $\olm_i\in\Lambda(\ox,\ov,w)\cap r\B$ easily imply that $\olm\in\Lambda(\ox,\ov,w)\cap r\B$. This implies that for any $i=1,\ldots,s$ we have
\begin{equation*}
\big\langle\olm_i,\nabla^2f(\bar x)(w,w)\big\rangle+\d^2\dd_\Th\big(f(\bar x),\olm_i\big)\big(\nabla f(\ox)w\big)=\big\langle\olm,\nabla^2f(\bar x)(w,w)\big\rangle+\d^2\dd_\Th\big(f(\bar x),\olm\big)\big(\nabla f(\ox)w\big),
\end{equation*}
which in turn gives us the representation
\begin{equation}\label{dn02}
\sum_{i=1}^{s}\tau_i\d^2\dd_\Th\big(f(\bar x),\olm_i\big)\big(\nabla f(\ox)w\big)=\d^2\dd_\Th\big(f(\bar x),\olm\big)\big(\nabla f(\ox)w\big).
\end{equation}
Denoting $\th_i(\cdot):=\sm\d^2\dd_\Th(f(\ox),\olm_i)(\nabla f(\ox)\cdot)$ for $i=1,\ldots,s$ and arguing as in the proof of claim (i) above tell us that each function $\th_i$ is convex.
Hence by \eqref{dn03} there is $\tilde q_i\in\sub\th_i(w)$ with $q_i=\nabla^2\la\olm_i,f\ra(\bar x)w+\tilde q_i$. This implies by the subdifferential construction of convex analysis that
\begin{equation*}
\la\tilde q_i,u-w\ra\le\th_i(u)-\th_i(w)\;\mbox{ for all }\;u\in\R^n.
\end{equation*}
Thus for any $u\in\R^n$ we deduce from \eqref{dn02} that
\begin{eqnarray*}
\big\la\sum_{i=1}^{s}\tau_i\tilde q_i,u-w\big\ra&\le&\sum_{i=1}^{s}\tau_i\Big(\th_i(u)-\th_i(w)\Big)\\
&=&\sm\sum_{i=1}^{s}\tau_i\d^2\dd_\Th\big(f(\ox),\olm_i\big)\big(\nabla f(\ox)u\big)-\sm\sum_{i=1}^{s}\tau_i\d^2\dd_\Th\big(f(\ox),\olm_i\big)\big(\nabla f(\ox)w\big)\\
&\le&\sm\d^2\dd_\Th\Big(f(\ox),\sum_{i=1}^{s}\tau_i\olm_i\Big)\big(\nabla f(\ox)u\big)-\sm\d^2\dd_\Th\big(f(\bar x),\olm\big)\big(\nabla f(\ox)w\big)\\
&=&\sm\d^2\dd_\Th\big(f(\ox),\olm\big)\big(\nabla f(\ox)u\big)-\sm\d^2\dd_\Th\big(f(\bar x),\olm\big)\big(\nabla f(\ox)w\big),
\end{eqnarray*}
where the second inequality comes from the fact that the mapping $\lm\mapsto\sm\d^2\dd_\Th(f(\ox),\lm)(\nabla f(\ox)u)$ is concave by Proposition~\ref{ssp}(iii). Hence we arrive at the inclusion
\begin{equation*}
\sum_{i=1}^{s}\tau_i\tilde q_i\in\sub_w\Big(\sm\d^2\dd_\Th\big(f(\ox),\olm\big)\big(\nabla f(\ox)\cdot\big)\Big)(w),
\end{equation*}
which brings us in turn to
\begin{eqnarray*}
q=\sum_{i=1}^{s}\tau_i q_i=\sum_{i=1}^{s}\tau_i\Big(\nabla^2\la\olm_i,f\ra(\bar x)w+\tilde q_i\Big)&=&\nabla^2\la\olm,f\ra(\bar x)w+\sum_{i=1}^{s}\tau_i\tilde q_i\\
&\in&\nabla^2\la\olm,f\ra(\bar x)w+\sub_w\Big(\sm\d^2\dd_\Th\big(f(\ox),\olm\big)\big(\nabla f(\ox)\cdot\big)\Big)(w).
\end{eqnarray*}
This verifies that we can drop the convex hull in the obtained formula for the graphical derivative of $N_\O$. So for every number $r$ with $r\ge\kappa\|\ov\|$ and every vector $w\in K_\O(\ox,\ov)$ we have
\begin{equation*}
DN_\O(\ox,\ov)(w)=\bigcup_{\lambda\in\Lambda(\ox,\ov,w)\cap r\B}\nabla^2\la\lm,f\ra(\bar x)w+\sub_w\Big(\sm\d^2\dd_\Th\big(f(\ox),\lambda\big)\big(\nabla f(\ox)\cdot\big)\Big)(w).
\end{equation*}
Choosing there $r:=\kappa\|\ov\|$ gives us the first formula for the graphical derivative of $N_\O$ claimed in the theorem. Furthermore, taking the union over all the numbers $r$ with $r\ge\kappa\|\ov\|$, we arrive at the second formula for the graphical derivative of $N_\O$ claimed therein. If finally $w\notin K_\O(\ox,\ov)$, which means that $w\notin\dom\d^2\dd_\O(\ox,\ov)$, then it follows from \eqref{gdpd} that $DN_\O(\ox,\ov)(w)=\emp$, and thus we complete the proof of the theorem.
\end{proof}

As shown in the proof of Theorem~\ref{gd}, the parabolic regularity of $\Th$ together with the assumptions in (H2) and (H3) imposed on $\Th$ ensures the convexity of the second subderivative $\d^2\dd_\Th(f(\ox),\lambda)$. This implies that the mapping $w\mapsto\sm\d^2\dd_\Th(f(\ox),\lambda)(\nabla f(\ox)w)$ standing under the convex subdifferential sign in the subgradient graphical derivative formulas of Theorem~\ref{gd} is a composition of a convex function and a linear operator. This calls for using a subdifferential chain rule of convex analysis to further elaborate the representations of $DN_\O$ in Theorem~\ref{gd} entirely via the given data of constraint systems \eqref{CS}.\vspace*{0.05in}

The next theorem provides refined formulas for $DN_\O$ in the cases where $\Th$ in \eqref{CS} is either a polyhedral convex set, or the second-order cone $\Q$ defined by \eqref{Qm}. We select these settings since for them the subdifferential sum rules do not require any qualification condition. While for polyhedral sets it follows from the above developments due to the classical chain rule of convex analysis, the case of $\Q$ is based on quite recent results for second-order cone programming.

\begin{Theorem}[\bf subgradient graphical derivative for polyhedral and second-order cone constraint systems]\label{police} In the framework of Theorem~{\rm\ref{gd}}, suppose that the underlying convex set $\Th$ is either a polyhedral set, or the second-order cone $\Q$ from \eqref{Qm}. Then for any $w\in K_\Th(\ox,\ov)$ the graphical derivative of $N_\O$ at $(\ox,\ov)$ is calculated by
\begin{equation}\label{gd1}
DN_\O(\ox,\ov)(w)=\bigcup_{\lambda\in\Lambda(\ox,\ov,w)}\nabla^2\la\lm,f\ra(\bar x)w+\nabla f(\ox)^*DN_\Th\big(f(\ox),\lambda\big)\big(\nabla f(\ox)w\big).
\end{equation}
Furthermore, in the polyhedral case for $\Th$ the term $DN_\Th$ in \eqref{gd1} is specified by
\begin{equation}\label{gd2}
DN_\Th\big(f(\ox),\lambda\big)(u)=N_{K_\Th(f(\ox),\lm)}(u)\;\mbox{ for all }\;u\in\R^m.
\end{equation}
If $\Th=\Q$, then the graphical derivative of $N_\Q$ at $(f(\ox),\lambda)$ is calculated by
\begin{equation}\label{gd3}
DN_\Q\big(f(\ox),\lambda\big)(u)={\cal H}\big(f(\ox),\lm\big)(u)+N_{K_\Q(f(\ox),\lm)}(u)\;\mbox{ for all }\;u\in\R^m,
\end{equation}
where the first term is given for all $u=(y,u_m)\in\R^{m-1}\times\R$ as
\begin{equation*}
{\cal H}\big(f(\ox),\lm\big)(u)=\begin{cases}
0&\mbox{if }\;f(\ox)\in(\inte Q)\cup\{0\},\\
\disp\frac{\|\lm\|}{\|f(\ox)\|}(y,-u_m)
&\mbox{if }\;f(\ox)\in(\bd\Q)\setminus\{0\}.
\end{cases}
\end{equation*}
\end{Theorem}
\begin{proof} Consider first the case where $\Th$ is a polyhedral set. Then we get from \eqref{sdpo} that
\begin{equation*}
\d^2\dd_\Th\big(f(\ox),\lambda\big)(w)=\dd_{K_\Th(f(\ox),\lm)}(w)\;\mbox{ for all }\;w\in\R^m.
\end{equation*}
This together with \eqref{gdpd} justifies \eqref{gd2}. Employing again \eqref{gdpd} with $w\in K_\Th(\ox,\ov)$ gives us
\begin{equation*}
DN_\Th\big(f(\ox),\lambda\big)\big(\nabla f(\ox)w\big)=\sub\Big(\sm\d^2\dd_\Th\big(f(\ox),\lambda\big)\Big)\big(\nabla f(\ox)w\big)=N_{K_\Th(f(\ox),\lm)}\big(\nabla f(\ox)w\big).
\end{equation*}
Since $K_\Th(f(\ox),\lm)$ is polyhedral convex set, we employ for the mapping $w\mapsto\sm\d^2\dd_\Th(f(\ox),\lambda)(\nabla f(\ox)w)$ the subdifferential chain rule from \cite[Theorem~23.9]{r70} in the case of polyhedrality, which yields
\begin{eqnarray*}
\sub_w\Big(\sm\d^2\dd_\Th\big(f(\ox),\lambda\big)\big(\nabla f(\ox)\cdot\big)\Big)(w)&=&\nabla f(\ox)^*\sub\Big(\sm\d^2\dd_\Th\big(f(\ox),\lambda\big)\Big)\big(\nabla f(\ox)w\big)\\
&=&\nabla f(\ox)^*DN_\Th\big(f(\ox),\lambda\big)\big(\nabla f(\ox)w\big).
\end{eqnarray*}
This verifies the subgradient graphical derivative formula \eqref{gd1} in the case.

Consider now the case where $\Th$ in \eqref{CS} is the second-order cone $\Q$. Using Example~\ref{ice} and \eqref{gdpd} proves \eqref{gd3}. To verify \eqref{gd1}, we split our arguments into two settings. If $f(\ox)\in\Q\setminus\{0\}$, then it follows from our discussions in Example~\ref{ice} that $K_\Th(f(\ox),\lm)$ is a polyhedral convex set, and thus we can employ the polyhedral arguments as above with the usage of the formula for the second subderivative of $\dd_\Q$ given in Example~\ref{ice}. It remains to consider the most interesting setting where $f(\ox)=0$. It follows in this case from Example~\ref{ice} that
\begin{equation*}
\d^2\dd_\Q\big(f(\ox),\lambda\big)\big(\nabla f(\ox)w\big)=\dd_{K_\Q(f(\ox),\lm)}\big(\nabla f(\ox)w\big)=\dd_{K_\O(\ox,\ov)}(w)\;\mbox{ for all }\;w\in\R^n.
\end{equation*}
Combining this along with \cite[Theorem~4.4]{hms} leads us the equalities
\begin{eqnarray*}
\sub_w\Big(\sm\d^2\dd_\Q\big(f(\ox),\lambda\big)\big(\nabla f(\ox)\cdot\big)\Big)(w)&=&N_{K_\O(\ox,\ov)}(w)\\
&=&\nabla f(\ox)^*\big(T_{N_\Q(f(\ox))}(\lm)\cap\{\nabla f(\ox)w\}^\perp\big)\\
&=&\nabla f(\ox)^*\big[\cl\big(N_\Q(f(\ox))+\R\lm\big)\cap\{\nabla f(\ox)w\}^\perp\big]\\
&=&\nabla f(\ox)^*\big[\big(K_\Q(f(\ox),\lm)\big)^*\cap\{\nabla f(\ox)w\}^\perp\big]\\
&=&\nabla f(\ox)^*N_{K_\Q(f(\ox),\lm)}\big(\nabla f(\ox)w\big)\\
&=&\nabla f(\ox)^*\sub\Big(\sm\d^2\dd_\Q\big(f(\ox),\lambda\big)\Big)\big(\nabla f(\ox)w\big)\\
&=&\nabla f(\ox)^*DN_\Q\big(f(\ox),\lambda\big)\big(\nabla f(\ox)w\big),
\end{eqnarray*}
which verify the claimed formula \eqref{gd1} for $\O=\Q$ and thus complete the proof.
\end{proof}

We conclude this section with the following discussions on the subgradient graphical derivative calculations obtained in Theorems~\ref{gd} and \ref{police}.

\begin{Remark}[\bf discussions on the subgradient graphical derivatives]\label{gdc}{\rm Let us begin with a brief overview of previous major attempts to calculate of the subgradient graphical derivatives for the constraint systems of type \eqref{CS}.

{\bf(i)} The systematic study of subgradient graphical derivatives was started by Poliquin and Rockafellar \cite{pr93} whose results contain the calculation of the graphical derivative of \eqref{CS}, where $\Th$ is a polyhedral convex  set under the validity of the metric regularity constraint qualification. New attempts to calculate the subgradient graphical derivative of constraint systems under MSCQ \eqref{mscq} were initiated by Gfrerer and Outrata for (polyhedral) problems of nonlinear programming. Nonpolyhedral constraint systems under MSCQ were first comprehensively investigated in \cite{hms} for the case of second-order cone programming, and then the computation formulas for $DN_\O$ were extended in \cite{gm17} to more general ${\cal C}^2$-cone reducible parametric constraint systems.

{\bf(ii)} Observe that all the recent results to calculate the second-order construction $DN_\O$ for some classes of constraint systems \eqref{CS} utilize  devices that are different from the original one in \cite{pr93}. In this paper we extend the approach of \cite{pr93} to a broad class of parabolically regular constraint systems that surely encompasses ${\cal C}^2$-cone reducible ones. In this way we exploit in the proof of Theorem~\ref{gd} an advanced result established quite recently by Correa, Hantoute and L\'opez \cite{chl}, which gives us a nice formula for the calculation of subgradients for suprema of parametric families of convex functions under fairly mild assumptions.

{\bf(iii)} Finally, let us show that for ${\cal C}^2$-cone reducible constraint systems, Theorem~\ref{gd} can be justified in a much simpler way using a first-order subdifferential formula given in \cite[Theorem~10.31]{rw}. Assuming in the framework of Theorem~\ref{gd} that $\Th$ is ${\cal C}^2$-cone reducible at $f(\ox)$ in the sense of \eqref{red}, we claim that
\begin{eqnarray*}
DN_\O(\bar x,\bar v)(w)&=&\Big(\bigcup_{\lm\in\Lambda(\bar x,\bar v,w)\cap(\kappa\|\ov\|\B)}\nabla^2\la\lm,f\ra(\bar x)w+\nabla f(\ox)^*\nabla^2\la\mu,h\ra(\oz) \nabla f(\ox)w\Big)\\\\
&&+N_{\ss{K_\O(\bar x,\bar v)}}(w)\\\\
&=&\Big(\bigcup_{\lm\in\Lambda(\bar x,\bar v,w)}\nabla^2\la\lm,f\ra(\bar x)w+\nabla f(\ox)^*\nabla^2\la\mu,h\ra(\oz)\nabla f(\ox)w\Big)+N_{\ss{K_\O(\bar x,\bar v)}}(w),
\end{eqnarray*}
where $\mu$ is the unique solution to \eqref{mubar} with $\oz:=f(\ox)$. The second representation claimed above resembles the one obtained in \cite{gm17} and is equivalent to the formula given in Theorem~\ref{gd}.

To verify the claimed formulas, pick $r\in\R$ such that $r\ge \kappa\|\ov\|$. Combining \eqref{epi4} and Theorem~\ref{twi20} we get the equalities
\begin{equation*}
\begin{array}{lll}
\d^2\dd_\O(\ox,\ov)(w)\\=\disp\max_{\lm\in\Lambda(\ox,\ov)\cap\,r\B}\;\big\{\big\langle\lm,\nabla^2f(\bar x)(w,w)\big\rangle+\big\la\mu,\nabla^2 h(\oz)(\nabla f(\ox)w,\nabla f(\ox)w)\big\ra+\dd_{\ss{K_\Th(f(\ox),\lm)}}\big(\nabla f(\ox)w\big)\big\}\\
=\disp\max_{\lm\in\Lambda(\ox,\ov)\cap\,r\B}\;\big\{\big\langle\lm,\nabla^2f(\bar x)(w,w)\big\rangle+\big\la\mu,\nabla^2 h(\oz)\big(\nabla f(\ox)w,\nabla f(\ox)w)\big\ra+\dd_{\ss{K_\O(\bar x,\bar v)}}(w)\big\}\\
=\disp\max_{\lm\in\Lambda(\ox,\ov)\cap\,r\B}\;\big\{\big\langle\lm,\nabla^2f(\bar x)(w,w)\big\rangle+\big\la\mu,\nabla^2h(\oz)(\nabla f(\ox)w,\nabla F(\ox)w)\big\ra\big\}+\dd_{\ss{K_\O(\bar x,\bar v)}}(w).
\end{array}
\end{equation*}
Employing this together with \eqref{gdpd}  tells us that
\begin{eqnarray*}
DN_\O(\bar x,\bar v)(w)&=&\sub\Big(\sm\d^2\dd_\O(\ox,\ov)\Big)(w)\\
&=&\sm\sub_w\left(\max_{\lm\in\Lambda(\ox,\ov)\cap\,r\B}\;\big\{\big\langle\lm,\nabla^2f(\bar x)(w,w)\big\rangle+\big\la\mu,\nabla^2h(\bar z)\big(\nabla f(\ox)w,\nabla f(\ox)w)\big\ra\big\}\right)\\\\
&&+N_{\ss {K_\O(\bar x,\bar v)}}(w)\\\\
&=&\co\Big\{\bigcup_{\lm\in\Lambda(\bar x,\bar v,w)\,\cap\,r\B}\nabla^2\la\lm,f\ra(\bar x)w+\nabla f(\ox)^*\nabla^2\la\mu,h\ra(\oz)\nabla F(\ox)w\Big\}\\\\
&&+N_{\ss{K_\O(\bar x,\bar v)}}(w),
\end{eqnarray*}
where the last equality comes from the well known subdifferential rule for maxima of smooth functions over compact sets; see, e.g., \cite[Theorem~10.31]{rw}.
Arguing similarly to the proof of Theorem~\ref{gd} allows us to drop the convex hull in the above formula. This implies that for every real number $r$ with $r\ge\kappa\|\ov\|$ and every $w\in\R^n$ we get
\begin{equation*}
DN_\O(\ox,\ov)(w)=\Big(\bigcup_{\lm\in\Lambda(\bar x,\bar v*,w)\,\cap\,r\B}\nabla^2\la\lm,f\ra(\bar x)w+\nabla f(\ox)^*\nabla^2\la\mu,h\ra(\oz)\nabla f(\ox)w\Big)+N_{\ss{K_\O(\bar x,\bar v)}}(w).
\end{equation*}
Letting $r:=\kappa\|\ov\|$ gives us the first claimed formula for $DN_\O$. Taking further the union over all $r$ with $r\ge\kappa\|\ov\|$ brings us to the second one and thus completes the proof of the above representations of $DN_\O$ for the case of ${\cal C}^2$-cone reducible constraint systems.}
\end{Remark}

\section{Concluding Remarks}\sce\label{conc}

This paper develops a comprehensive theory of parabolic regularity for sets in geometric variational analysis with novel applications to optimization, well-posedness, and related topics. We show that parabolically regular sets encompass large classes of sets previously used in second-order variational analysis and enjoy nice properties, which are preserved under various operations on sets. Furthermore, we demonstrate that parabolic regularity is the key for developing extended calculus rules for major second-order generalized derivatives with obtaining precise formulas for their computation. The established calculus and computation results lead to broad applications to problems of constrained optimization with deriving in particular, no-gap second-order optimality conditions and establishing quadratic growth of augmented Lagrangians, which has been a goal for many previous efforts. The developed theory of parabolic regularity opens the gate for further applications to theoretical and algorithmic aspects of optimization, nonlinear analysis, and related areas of mathematics.\vspace*{0.1in}

{\bf Acknowledgements}. The authors are grateful to two anonymous referees for their helpful remarks and to the Managing Editor Alejandro Adem for his efficient handling the paper.


\begin{thebibliography}{10}

\bibitem{ab} {\sc S. Adly} and {\sc L. Bourdin}, {\em Sensitivity analysis of variational inequalities via twice epi-differentiability and
proto-differentiability of the proximity operator}, SIAM J. Optim. {\bf 28} (2018), 1699-1725.

\bibitem{bbl} {\sc H. H. Bauschke, J. M. Borwein} and {\sc W. Li}, {\em  Strong conical hull intersection property, bounded linear regularity, Jamesons
property (G), and error bounds in convex optimization}, Math. Program. {\bf 86} (1999),  135--160.

\bibitem{bz1} {\sc A. Ben-Tal} and {\sc J. Zowe}, {\em A unified theory of first- and second-order conditions for extremum problems in topological vector
spaces}, Math. Program.  {\bf 19} (1982), 39--76.

\bibitem{bz2} {\sc A. Ben-Tal} and {\sc J. Zowe}, {\em Directional derivatives in nonsmooth optimization}, J. Optim. Theory Appl. {\bf 47} (1985),
483--490.

\bibitem{bcs} {\sc J. F. Bonnans, R. Cominetti} and {\sc A. Shapiro}, {\em Second-order optimality conditions based on parabolic second-order tangent sets},
SIAM J. Optim. {\bf 9} (1998), 466--492.

\bibitem{bcs2} {\sc J. F. Bonnans, R. Cominetti} and {\sc A. Shapiro}, {\em Sensitivity analysis of optimization problems under second-order regular
constraints}, Math. Oper. Res. {\bf 23} (1998), 806--831.

\bibitem{bs} {\sc J. F. Bonnans} and {\sc A. Shapiro}, {\em Perturbation Analysis of Optimization Problems}, Springer, New York, 2000.

\bibitem{bou} {\sc G. Bouligand}, {\em Sur les surfaces d\'epourvues de points hyperlimits}, Ann. Soc. Polon Math. {\bf 9} (1930), 32--41.

\bibitem{ch0} {\sc R. W. Chaney},  {\em  On second derivatives for nonsmooth functions},  Nonlinear Anal. {\bf 9} (1985), 1189--1209.

\bibitem{ch1} {\sc R. W. Chaney},  {\em Second-order directional derivatives for nonsmooth functions}, J. Math. Anal.  Appl. {\bf 128} (1987), 495--511.

\bibitem{ch2} {\sc R. W. Chaney},  {\em Second-order sufficient conditions in nonsmooth optimization}, Math. Oper. Res. {\bf 13} (1988), 660--673.

 \bibitem{chi} {\sc N. H. Chieu} and {\sc L. V. Hien}, {\em Computation of graphical derivative for a class of normal cone mappings under a
very weak condition}, SIAM J. Optim. {\bf 27} (2017), 190--204.

\bibitem{chnt} {\sc N. H. Chieu, L. V. Hien, T. T. A. Nghia} and {\sc H. A. Tuan}, {\em Second-order optimality conditions for strong local minimizers
via subgradient graphical derivative} (2019), arXiv:1903.05746.

\bibitem{col-thi} {\sc G. Colombo} and {\sc L. Thibault}, {\em Prox-regular sets and applications}, in Handbook of Nonconvex Analysis (edited by D. Y. Gao and
D. Motreanu), pp.\ 99-182, International Press, Boston, MA.

\bibitem{chl} {\sc R. Correa, A. Hantoute} and {\sc M. A. L\'opez}, {\em Moreau-Rockafellar type formulas for the subdifferential of the supremum function},
SIAM J. Optim. {\bf 29} (2019), 1106--1130.

\bibitem{dr} {\sc A. L. Dontchev} and {\sc R. T. Rockafellar}, {\em Implicit Functions and Solution Mappings: A View from Variational Analysis}, 2nd edition,
Springer, Dordrecht, 2014.

\bibitem{fed} {\sc H. Federer}, {\em Curvature measures}, Trans. Amer. Math. Soc. {\bf 93} (1959), 418--491.

\bibitem{gf} {\sc H. Gfrerer}, {\em  First-order and second-order characterizations of metric subregularity and calmness of constraint set mappings}, SIAM
J. Optim. {\bf 21} (2011), 1439--1474.

\bibitem{gm17} {\sc H. Gfrerer} and {\sc B. S. Mordukhovich}, {\em Second-order variational analysis of parametric constraint and variational systems},
SIAM J. Optim. {\bf 29} (2019), 423--453.

\bibitem{go16} {\sc H. Gfrerer} and {\sc J. V. Outrata}, {\em On computation of generalized derivatives of the normal-cone mapping and their applications},
Math. Oper. Res. {\bf 41} (2016), 1535--1556.

\bibitem{hu1} {\sc J.-B. Hiriart-Urruty}, {\em  Approximating a second-order directional derivative for nonsmooth convex functions}, SIAM J. Control. Optim.
{\bf 20} (1982), 381--404.

\bibitem{hu2} {\sc J.-B. Hiriart-Urruty}, {\em Limiting behavior of the approximate first- and second-order directional derivatives for a convex
function}, Nonlinear Anal. {\bf 6} (1982), 1309--1326.

\bibitem{hms} {\sc N. T. V. Hang, B. S. Mordukhovich} and {\sc M. E. Sarabi}, {\em Second-order variational analysis in second-order cone programming},
Math. Program. {\bf 180} (2020), 75--116.

\bibitem{hms2} {\sc N. T. V. Hang, B. S. Mordukhovich} and {\sc M. E. Sarabi}, {\em Augmented Lagrangian method for second-order conic programs under second-order sufficiency},
(2020), arXiv:2005.04182.

\bibitem{hjo} {\sc R. Henrion, A. Jourani} and {\sc J. V. Outrata},  {\em On the calmness of a class of multifunctions}, SIAM J. Optim. {\bf 13}
(2002), 603--618.

\bibitem{ho} {\sc R. Henrion} and {\sc J. V. Outrata}, {\em Calmness of constraint systems with applications}, Math. Program. {\bf 104} (2005), 437--464.	

\bibitem{ln} {\sc C. Lemar\'echal} and {\sc E. Nurminskii}, {\em Sur la diff\'erentiabilit\'e de la fonction d'appui du sous-diff\'erential approach\'e},
C. R. Acad. Sci. Paris {\bf 90} (1980), 855--858.

\bibitem{lev-pol-thi} {\sc A. B. Levy, R. A. Poliquin} and {\sc L. Thibault}, A partial extension of Attouch's theorem and its applications to
second-order epi-differentiation, {\em Trans. Amer. Math. Soc.} {\bf 347} (1995), 1269--1294.

\bibitem{lz} {\sc Y. J. Liua} and {\sc L. Zhang}, {\em Convergence analysis of the augmented Lagrangian method for nonlinear second-order cone
optimization problems}, Nonlinear Anal. {\bf 67} (2007), 1359--1373.

\bibitem{loe} {\sc P. D. Loewen} and {\sc H. Zheng}, {\em Epi-differentiability of intergral functionals with applications}, Trans. Amer. Math. Soc.
{\bf 347} (1995), 443--459.

\bibitem{mm} {\sc A. Mohammadi} and {\sc B. S. Mordukhovich}, {\em Variational analysis in normed spaces with applications to constrained optimization} (2020), arXiv:2006.00462. 

\bibitem{mms} {\sc A. Mohammadi, B. S. Mordukhovich} and {\sc M. E. Sarabi}, {\em Variational analysis of composite models with applications to
continuous optimization} to appear in Math. Oper. Res. (2020), arXiv:1905.08837v2.

\bibitem{m76} {\sc B. S. Mordukhovich}, {\em Maximum principle in problems of time optimal control with nonsmooth constraints}, J. Appl. Math. Mech.
{\bf 40} (1976), 960--969.

\bibitem{m92} {\sc B. S. Mordukhovich}, {\em Sensitivity analysis in nonsmooth optimization}, in Theoretical Aspects of Industrial Design (edited by D. A.
Field and V. Komkov), pp.\ 32--46, SIAM Proc. Appl. Math. {\bf 58}, Philadelphia, PA.

\bibitem{mor93} {\sc B. S. Mordukhovich}, {\em Complete characterization of openness, metric regularity, and Lipschitzian properties of multifunctions},
Trans. Amer. Math. Soc. {\bf 340} (1993), 1--35.

\bibitem{mor06} {\sc B. S. Mordukhovich}, {\em Variational Analysis and Generalized Differentiation, I: Basic Theory, II: Applications}, Grundlehren
Series (Fundamental Principles of Mathematical Sciences), Vols.\ 330 and 331, Springer, Berlin, 2006.

\bibitem{mor18} {\sc B. S. Mordukhovich}, {\em Variational Analysis and Applications}, Springer Monographs in Mathematics, Springer, Cham, Switzerland, 2018.

\bibitem{pr93} {\sc R. A. Poliquin} and {\sc R. T. Rockafellar}, {\em A calculus of epi-derivatives applicable to optimization}, Canadian J. Math.{\bf 45}
(1993), 879--896.

\bibitem{pr96} {\sc R. A. Poliquin} and {\sc R. T. Rockafellar}, {\em Prox-regular functions in variational analysis}, Trans. Amer. Math. Soc. {\bf 348}
(1996), 1805--1838.

\bibitem{rob} {\sc S. M. Robinson}, {\em Generalized equations and their solutions, I: Basic theory}, Math. Program. Study {\bf 10} (1979), 128--141.

\bibitem{r70} {\sc R. T. Rockafellar}, {\em Convex Analysis}, Princeton University Press, Princeton, NJ, 1970.

\bibitem{r88} {\sc R. T. Rockafellar}, {\em First- and second-order epi-differentiability in nonlinear programming}, Trans. Amer. Math. Soc. {\bf 307}
(1988), 75--108.

\bibitem{r89} {\sc R. T. Rockafellar}, {\em Proto-differentiability of set-valued mappings and its applications in optimization}, in Analyse Non
Lin\'eaire (edited by H. Attouch et al.), pp.\ 449--482, Gathier-Villars, Paris, 1989.

\bibitem{r892} {\sc R. T. Rockafellar}, {\em Second-order optimality conditions in nonlinear programming obtained by way of epi-derivatives}, Math. Oper.
Res. {\bf 14} (1989), 462--484.

\bibitem{r93} {\sc R. T. Rockafellar}, {\em Lagrange multipliers and optimality}, SIAM Rev. {\bf 35} (1993), 183--238.

\bibitem{r20} {\sc R. T. Rockafellar}, {\em Second-order convex analysis}, J. Convex Nonlin. Anal. {\bf 1} (1999), 1--16.

\bibitem{rw} {\sc R. T. Rockafellar} and {\sc R. J-B. Wets}, {\em Variational Analysis}, Grundlehren Series (Fundamental Principles of Mathematical
Sciences), Vol.\ 317, Springer, Berlin, 2006.

\bibitem{ru} {\sc W. R. Rudin}, {\em Principles of Mathematical Analysis}, 3rd edition, McGraw-Hill, 1976.

\bibitem{se} {\sc F. Severi}, {\em Su alcune questioni di topologia infinitesimale}, Ann. Soc. Polon. Math. {\bf 9} (1930), 97--108.

\bibitem{ss} {\sc A. Shapiro} and {\sc J. Sun}, {\em Some properties of the augmented Lagrangian in cone constrained optimization}, Math. Oper. Res.
{\bf 29} (2004), 479--491.

\bibitem{ssz} {\sc D. Sun, J. Sun}  and  {\sc L. Zhang}, {\em The rate of convergence of the augmented Lagrangian method for nonlinear semidefinite
programming}, Math. Program. {\bf 114} (2008), 349--391.
\end{thebibliography}
\end{document}